\documentclass[11pt,final]{article}
\usepackage{amsmath}
\usepackage{amssymb}
\usepackage{amscd,todonotes}
\usepackage{latexsym,microtype}
\usepackage{theorem}
\usepackage{epsfig,euscript}
\usepackage{amsfonts,nicefrac}
\usepackage{xypic,xspace,tikz}
\usepackage{bbm,ifpdf}
\usepackage[notref,notcite]{showkeys}
\usepackage{setspace,mathrsfs}
\usepackage [margin=3cm]{geometry}
\ifpdf
  \usepackage{pdfsync,hyperref}
\fi

\newtheorem{theorem}{Theorem}[section]
\newtheorem{lemma}[theorem]{Lemma}
\newtheorem{proposition}[theorem]{Proposition}
\newtheorem{corollary}[theorem]{Corollary}
\newtheorem{cor}{Corollary}

\newtheorem{conjecture}[theorem]{Conjecture}

{
\theorembodyfont{\rmfamily}
\theoremstyle{plain}
\newtheorem{definition}[theorem]{Definition}
\newtheorem{example}[theorem]{Example}

\newtheorem{remark}[theorem]{Remark}
}

\newcommand{\qed}{\hfill \mbox{$\Box$}\medskip\newline}
\newenvironment{proof}{\noindent {\bf Proof:}}{\qed \par}

\newenvironment{proofZloc}{\noindent {\bf Proof of Theorem \ref{Zloc}:}}{\qed \par}

\newcommand{\Spec}{\operatorname{Spec}}

\newcommand{\Aut}{\operatorname{Aut}}
\newcommand{\Der}{\operatorname{Der}}

\newcommand{\id}{\operatorname{id}}

\newcommand{\codim}{\operatorname{codim}}
\renewcommand{\dim}{\operatorname{dim}}

\newcommand{\Sym}{\operatorname{Sym}}
\newcommand{\rk}{\operatorname{rk}}

\newcommand{\scrM}{\mathscr{M}}

\newcommand{\scrA}{\mathscr{A}}
\newcommand{\scrD}{\mathscr{D}}

\newcommand{\Z}{\mathbb{Z}}
\newcommand{\Q}{\mathbb{Q}}
\newcommand{\N}{\mathbb{N}}
\newcommand{\R}{\mathbb{R}}
\newcommand{\C}{\mathbb{C}}

\newcommand{\la}{\leftarrow}

\newcommand{\M}{\mathfrak{M}}
\newcommand{\sEnd}{{\cal E}\mathit{nd}}
\newcommand{\sHom}{{\cal H}\mathit{om}}

\renewcommand{\la}{\lambda}
\newcommand{\cs}{\C^\times}

\renewcommand{\a}{\alpha}
\renewcommand{\b}{\beta}

\newcommand{\Hom}{\operatorname{Hom}}

\newcommand{\HLCh}{H^{\dim\fM}_\fL\big(\fM; \C((h)) \big)}

\newcommand{\cP}{\mathcal{P}}

\newcommand{\cO}{\mathcal{O}}
\renewcommand{\cL}{\mathcal{L}}

\newcommand{\cT}{\mathcal{T}}
\newcommand{\cU}{\mathcal{U}}
\newcommand{\cM}{\mathcal{M}}

\newcommand{\cJ}{\mathcal{J}}

\newcommand{\becircled}{\mathaccent "7017}

\newcommand{\Ext}{\operatorname{Ext}}
\newcommand{\fS}{\mathfrak{S}}

\renewcommand{\cR}{\mathcal{R}}
\renewcommand{\cD}{\mathcal{D}}

\newcommand{\gr}{\operatorname{gr}}

\newcommand{\cQ}{\mathcal{Q}}

\newcommand{\Loc}{\operatorname{Loc}}

\newcommand{\cE}{\mathcal{E}}

\newcommand{\End}{\operatorname{End}}

\newcommand{\excise}[1]{}
\newcommand{\bT}{\mathbb{T}}

\newcommand{\bS}{\mathbb{S}}
\newcommand{\mg}{\mathfrak{g}}

\newcommand{\mt}{\mathfrak{t}}
\newcommand{\fp}{\mathfrak{p}}

\newcommand{\bi}{{_{\la'}\!T_{\la}}}
\newcommand{\kmi}{{_{k}\!\cT_{m}}}

\newcommand{\cN}{\mathcal{N}}
\renewcommand{\cH}{\mathcal{H}}

\newcommand{\fM}{\mathfrak{M}}

\newcommand{\fU}{\mathfrak{U}}

\newcommand{\fP}{\mathfrak{P}}
\newcommand{\fZ}{\mathfrak{Z}}
\newcommand{\fX}{\mathfrak{X}}
\renewcommand{\Loc}{\operatorname{Loc}}
\newcommand{\secs}{\Gamma_\bS}
\newcommand{\LLoc}{\mathbb{L}\!\operatorname{Loc}}
\newcommand{\Rsecs}{\mathbb{R}\Gamma_\bS}

\newcommand{\Zpmod}{\Zp\!\mmod}
\newcommand{\Zmod}{Z\!\mmod}
\newcommand{\Zpmodbd}{\!\Zpmod_{\operatorname{bd}}}
\newcommand{\Zmodbd}{\!\Zmod_{\operatorname{bd}}}
\newcommand{\Amod}{A\mmod}
\newcommand{\AMod}{A\operatorname{-Mod}}

\newcommand{\Dmod}{\cD\mmod}
\newcommand{\Dlmod}{\cD_\la\mmod}

\newcommand{\mmod}{\operatorname{-mod}}
\newcommand{\mMod}{\operatorname{-Mod}}

\renewcommand{\and}{\qquad\text{and}\qquad}

\newcommand{\mb}{\mathfrak{b}}
\newcommand{\Zp}{\Za^{(p)}}
\newcommand{\Za}{Z}

\newcommand{\suppc}{\operatorname{CC}}
\newcommand{\supp}{\operatorname{Supp}}

\usetikzlibrary{matrix}
\newcommand{\Ht}{H^2(\fM;\C)}
\newcommand{\Htr}{H^2(\fM;\R)}

\newcommand{\si}{\sigma}

\newcommand{\aone}{\mathbb{A}^{\! 1}}

\newcommand{\HCa}{\mathbf{HC}^{\operatorname{a}}}
\newcommand{\HCg}{\mathbf{HC}^{\operatorname{g}}}
\newcommand{\HCalala}{{_{\la'}^{\mbox{}}\HCa_{\la}}}
\newcommand{\HCglala}{{_{\la'}^{\mbox{}}\HCg_{\la}}}

\newcommand{\Pic}{\operatorname{Pic}}
\newcommand{\Mov}{\operatorname{Mov}}
\newcommand{\fin}{\operatorname{fin}}
\newcommand{\an}{\operatorname{an}}
\newcommand{\op}{\operatorname{op}}
\newcommand{\scr}{\mathscr}
\newcommand{\scrQ}{\scr{Q}}
\newcommand{\bigmid}{\;\Big{|}\;}

\newcommand{\e}{\epsilon}

\newcommand{\cDop}{\cD^{\op}}
\newcommand{\Aop}{A^{\op}}
\newcommand{\bcN}{\overline{\cN}}
\newcommand{\cDred}{\cD_{\red}}
\newcommand{\cQred}{\cQ_{\red}}
\newcommand{\fXred}{\fX_{\red}}
\newcommand{\Lotimes}{\overset{L}\otimes}
\newcommand{\red}{\operatorname{red}}

\newcommand{\Caldararu}{C\u{a}ld\u{a}raru\xspace}
\newcommand{\hon}{h^{\nicefrac{1}{n}}}
\newcommand{\hmon}{h^{\nicefrac{-1}{n}}}
\newcommand{\scrN}{\scr{N}}
\newcommand{\Qua}{\mathsf{Qua}}
\newcommand{\Quag}{\mathsf{Qua}^{\operatorname g}}
\newcommand{\Quaa}{\mathsf{Qua}^{\operatorname a}}

\newcommand{\Cat}{\mathsf{Cat}}
\newcommand{\CLg}{\mathcal{C}^\fL_\la}
\newcommand{\CLa}{C^{\fL_0}_\la}
\newcommand{\DCLg}{D_{\fL}^b(\Dlmod)}
\newcommand{\DCLa}{D_{\fL_0}^b(A_\la\mmod)}
\newcommand{\fL}{\mathfrak{L}}
\newcommand{\coho}{\mathbbm{H}}

\begin{document}
\spacing{1.5}

\noindent {\Large \bf 
Quantizations of conical symplectic resolutions 
I: \\ local and global structure
}
\\
\spacing{1.2}
\noindent
{\bf Tom Braden}\footnote{Supported by NSA grants H98230-08-1-0097 and H98230-11-1-0180.}\\
Department of Mathematics and Statistics, University of Massachusetts,
Amherst, MA 01003\smallskip \\
{\bf Nicholas Proudfoot}\footnote{Supported by NSF grant DMS-0950383.}\\
Department of Mathematics, University of Oregon,
Eugene, OR 97403\smallskip \\
{\bf Ben Webster}\footnote{Supported by NSA grant H98230-10-1-0199.
}\\
Department of Mathematics, University of Virginia, Charlottesville, VA 22904
\bigskip\\
{\small
\begin{quote}
\noindent {\em Abstract.}
We re-examine some topics in representation theory of Lie
algebras and Springer theory in a more general context, viewing
the universal enveloping algebra as an example of the section ring of a 
quantization of a conical symplectic resolution.  While some
modification from this classical context is necessary, many familiar
features survive.  These include a version of the Beilinson-Bernstein
localization theorem, a theory of Harish-Chandra bimodules and their relationship
to convolution operators on cohomology, and a discrete
group action on the derived category of representations, generalizing the braid group action
on category $\cO$ via twisting functors.

Our primary goal is to apply these results to other quantized symplectic
resolutions, including quiver varieties and hypertoric varieties.
This provides a new context for known results about Lie algebras, Cherednik
algebras, finite W-algebras, and hypertoric enveloping algebras, while also
pointing to the study of new algebras arising from more
general resolutions.
\end{quote}
}
\bigskip

\section{Introduction}
\renewcommand{\thetheorem}{\Alph{theorem}}
\renewcommand{\thecor}{\Alph{theorem}.\arabic{cor}}
\begin{quote}
  \textsl{The dazzling success of algebraic geometry\ldots has so much
  reorientated the field that one particular protagonist has suggested,
  no doubt with much justification, that enveloping algebras should
  now be relegated to a subdivision of the theory of rings of
  differential operators.}

\flushright --Anthony Joseph, {\em On the classification of primitive
  ideals in the enveloping algebra of a semisimple Lie algebra} \cite{Josephclass}
\end{quote}

In this paper, we argue against the relegation suggested above, in favor of a different
geometric context.  While viewing universal enveloping
algebras as differential operators is unquestionably a powerful
technique, the differential operators on
flag varieties are odd men out in the world of differential operators
as a whole. For example, the only known examples of projective varieties
that are D-affine are homogeneous spaces for semi-simple complex Lie groups,
and it is conjectured that no other examples exist.
On the other hand, in this paper we consider a world
where this special case is very much at home: quantizations of symplectic
resolutions of affine singularities.

Differential operators on a smooth projective variety $X$ form a deformation quantization
of the cotangent bundle $T^*X$.  If $X$ is a homogeneous space for a
semi-simple complex Lie group $G$, its cotangent bundle is a resolution of the closure
of a nilpotent orbit in $\mg^*$ (or an affine variety finite over this
one).  If $X$ is the flag variety, this is known as the {\bf Springer resolution}. 
This is yet another sense in which these spaces are misfits;
homogeneous spaces for semi-simple complex Lie groups are conjecturally the only examples
of projective varieties whose cotangent bundles resolve affine singularities.  For most projective varieties $X$,
$T^*X$ does not have enough global functions.

There are, however, many other examples of symplectic algebraic varieties that
resolve affine cones.  While the Springer resolution is the
most famous, other examples include the minimal resolution of a
Kleinian singularity, the Hilbert scheme of points on such a
resolution, Nakajima quiver varieties, and hypertoric varieties.  One can study deformation quantizations
of these varieties, and many of them have the same affinity property enjoyed by the Springer resolution.
This paper is a study of these deformation
quantizations and their representation theory.

Several examples have been studied extensively by other authors.  
Universal enveloping algebras have been
considered from an enormous number of angles for decades, and other
examples such as spherical Cherednik algebras and finite W-algebras have
been active fields of research for many years.  The hypertoric
case has recently been studied by Bellamy and Kuwabara
\cite{BeKu} and by the authors of this paper, jointly with Licata \cite{BLPWtorico}.  
On the other hand, very few works attempt to view
all these examples in a single coherent theory.  Kashiwara and Rouquier began
to develop such a theory \cite{KR}, and our paper might be regarded as a continuation
of their work.  A recent preprint of McGerty and Nevins \cite{MN} addresses similar issues, with results
that are complementary to ours.

In Section \ref{sec:resolutions}, we discuss the algebraic geometry of conical
symplectic resolutions; this is essentially all material already in
the literature, but we collect it here for the convenience of the
reader.  Particularly important for us are deformations which appear
in the work of Kaledin and Verbitsky; these show that any symplectic
resolution flatly deforms to a smooth affine variety, which is key to many
properties of its quantization.  One ingredient we will use
systematically is the conical structure: a choice of
$\C^*$-action which makes the base into a cone and acts with
positive weight on the symplectic form.  

In Section \ref{sec:quantizations}, we discuss equivariant quantizations of a conical
symplectic resolution $\fM$, which are classified by $H^2(\fM;\C)$ \cite{BK04a,Losq}.
We prove some basic results about the ring
 $A$ of $\bS$-invariant global sections, a filtered algebra whose associated graded is isomorphic
to $\C[\fM]$.  We also study the behavior of quantizations under (quantum) Hamiltonian reduction,
proving a quantum version of the Duistermaat-Heckman theorem
(Proposition \ref{dui-heck}).

In Section \ref{sec:modules} we introduce the appropriate category $\Dmod$ of modules over a quantization $\cD$,
which one may regard as the quantum analogue of the category of coherent sheaves
(in particular, there is a finiteness assumption built into the definition).  
In the case where $\fM$ is a cotangent bundle, we show that this category is equivalent
to the category of finitely generated twisted D-modules on the base, where the twist is determined by the
period of the quantization.  The rest of the section is dedicated to the study of the sections and localization
functors that relate the category of modules over a quantization to the category of modules
over the section ring $A$.  We establish in Theorem \ref{derived-local} that these functors
induce derived equivalences for generic periods.

\begin{theorem}
Let $\fM$ be a conical symplectic resolution, and fix two classes $\eta,\la\in\Ht$ such that $\eta$
is the Chern class of an ample line bundle, or the strict transform of an ample line bundle on any other
conical symplectic resolution of $\fM_0$.
For all but finitely many complex numbers $k$,
the quantization of $\fM$ with period $\la+k\eta$ is derived affine; that is,
the derived functors of global sections and localization are inverse equivalences.
\end{theorem}

In order to obtain an equivalence of abelian (rather than derived) categories that works for all (rather than only
generic) periods, we replace the section ring $A$ with a $\Z$-algebra, which mimics in a non-commutative
setting the homogeneous coordinate ring of a projective variety.  Given a quantized symplectic resolution
along with a very ample line bundle, we construct a $\Z$-algebra $Z$ and prove the following result (Theorem \ref{Zloc}).

\begin{theorem}
Let $\fM$ be a conical symplectic resolution, let $\cL$ be a very ample line bundle on $\fM$, and
let $Z$ be the associated $\Z$-algebra.  Then the category $\Dmod$ is equivalent to the category
of finitely generated modules over $Z$ modulo the subcategory of bounded modules.
\end{theorem}

Theorem B has three nice consequences.  First, we use it to prove the following abelian analogue
of Theorem A (Corollary \ref{large quantizations}).

\begin{cor}\label{first-cor}
Let $\fM$ be a conical symplectic resolution, and fix two classes $\eta,\la\in\Ht$ such that $\eta$
is the Chern class of an ample line bundle.
For all but finitely many positive integers $k$,
the quantization of $\fM$ with period $\la+k\eta$ is affine; that is,
the (abelian) functors of global sections and localization are inverse equivalences.
\end{cor}

Next, we prove a version of Serre's GAGA
theorem \cite{GAGA}.  
More precisely, we consider the analytic quantization $\cD^{\an}$ with the same period as $\cD$,
define the appropriate module category $\cD^{\an}\mmod$, and prove that it is equivalent to $\Dmod$
(Theorem \ref{th:GAGA}).  The existing literature is fairly evenly divided between working in the algebraic and analytic
categories, and this corollary is an indispensable tool that allows us to import previous results from both sides.

\begin{cor}
If $\fM$ is a conical symplectic resolution, then the analytification functor
from $\cD\mmod$ to $\cD^{\an}\mmod$ is an equivalence of categories.  
\end{cor}

Finally, we use Theorem B to
prove a categorical version of Kirwan surjectivity, 
relating the category of equivariant modules on a quantization to the category of modules on the
Hamiltonian reduction.  We consider a restriction functor defined by Kashiwara and Rouquier,
and we use our $\Z$-algebra formalism to construct left and right adjoints, thus proving that the restriction
functor is essentially surjective (Theorem \ref{surjectivity}).
In particular, this result establishes that our category $\Dmod$ is the same
as the analogous category considered by McGerty and Nevins (Remark \ref{ksmn}).  For a precise statement of the hypotheses of the following result, see the beginning of Section \ref{sec:kirwan}.

\begin{cor}
If $\fM$ is obtained via symplectic reduction from an action of a reductive group $G$ on $\fX$,
then every object of $\Dmod$ extends to a twisted $G$-equivariant module over a quantization of $\fX$.
\end{cor}

Let $\fM_0 := \Spec \C[\fM]$ be the cone resolved by $\fM$, and consider the {\bf Steinberg variety} 
$\fZ := \fM\times_{\fM_0}\fM$.
The cohomology $H^{2\dim\fM}_\fZ(\fM\times\fM)$ with supports in $\fZ$, 
which by Poincar\'e duality can
be identified with the Borel-Moore homology group $H^{B\!M}_{2\dim\fM}(\fZ)$, has a natural
algebra structure via convolution \cite[\S 2.7]{CG97}.  Furthermore,
if $\fL\subset\fM$ is a Lagrangian subscheme that is equal to the preimage of its image in $\fL_0\subset \fM_0$, then
the convolution algebra acts on $H^{\dim\fM}_\fL(\fM)$.  
In the special case where $\fM = T^*(G/B)$ and $\fL$ is the conormal variety to the Schubert stratification of $G/B$, 
the convolution algebra is isomorphic to the group algebra of the Weyl group,
and $H^{\dim\fM}_\fL(\fM)$ is isomorphic to the regular representation.  More generally,
there is a natural algebra homomorphism from the group algebra $\C[W]$ of the Namikawa Weyl group of $\fM$ to the convolution algebra $H^{2\dim\fM}_\fZ(\fM\times\fM)$.

Section \ref{sec:bimodules} is devoted to categorifying the picture described in the paragraph above.
The convolution algebra is replaced by the monoidal category of Harish-Chandra bimodules, 
which comes in both an algebraic and a geometric version.
The module $H^{\dim\fM}_\fL(\fM)$ is replaced by a subcategory $\mathcal{C}^\mathfrak{L}\subset \Dmod$
(respectively $C^{\fL_0}\subset\Amod$) which is a module category for the category of geometric 
(respectively algebraic) Harish-Chandra bimodules.
Following Kashiwara and Schapira \cite{KSdq},
we define the characteristic cycle of a geometric Harish-Chandra bimodule, which lies in $\fZ$,
and the characteristic cycle of an object of $\mathcal{C}^\mathfrak{L}$, which lies in $\fL$.
Using the machinery developed in \cite{KSdq}, we prove that these cycles are compatible with convolution.

\begin{theorem}
The characteristic cycle map intertwines convolution of geometric Harish-Chandra bimodules
with convolution in the Borel-Moore homology of the Steinberg variety (Proposition \ref{K-conv});
it also intertwines the action of Harish-Chandra bimodules on $\mathcal{C}^\mathfrak{L}$ with
the action of $H^{2\dim\fM}_\fZ(\fM\times\fM)$ on $H^{\dim\fM}_\fL(\fM)$ (Proposition \ref{K-act}).
\end{theorem}

There is particularly nice collection of algebraic Harish-Chandra bimodules which appear
naturally from changing the period of the quantization. 
Let $A_\la$ be the section ring of the quantization with period $\la\in\Ht$.
Derived tensor products with these special bimodules
give derived equivalences between the derived categories of modules over
$A_\la$ for various different $\la$.
These equivalences are far from being unique; instead, they
induce a large group of autoequivalences of
$D(A_\la\mmod)$ for each fixed $\la$, called {\bf twisting functors}.
There is a hyperplane arrangement in $\Htr$ whose chambers are the Mori chambers of $\fM$;
let $E\subset\Ht$ be the complement of the complexification of this arrangement.
The Namikawa Weyl group $W$ acts on $\Ht$ preserving $E$.  

\begin{theorem}
  There is a weak action of $\pi_1(E/W, [\la])$ on $D(A_\la\mmod)$ by twisting functors
  (Theorem \ref{twisting braid}); this action preserves the subcategory $D(C^{\fL_0})$ (Remark \ref{preserves subcategory}).
  The subgroup $\pi_1(E, \la)$ preserves the characteristic cycle of a module, thus $W \cong \pi_1(E/W, [\la]) / \pi_1(E, \la)$
  acts on $H^{\dim\fM}_\fL(\fM)$ (Proposition \ref{K-triv}).  This action agrees with the action
  induced by the natural map from $\C[W]$ to the convolution algebra (Remark \ref{W-action}).
\end{theorem}

In the case where $\fM$ is the Springer resolution for $G$, the space
$E$ is the complement of the complexified Coxeter arrangement, $W$ is
the classical Weyl group, and $\pi_1(E/W)$ is the generalized braid
group.  If $\fL\subset\fM$ is taken to be the conormal variety to the Schubert stratification
and the period of the quantization is regular, then
$C^{\fL_0}$ is equivalent to a regular block of category $\cO$ (Example \ref{catO}).  In this case, 
the action of the generalized braid group coincides with Arkhipov's twisting action (Proposition \ref{shuf-twist}),
which categorifies the regular representation of $W$.\\

\noindent
{\em Acknowledgments:}
The authors would like to thank Roman Bezrukavnikov, Dmitry Kaledin, Ivan Losev, and especially Anthony Licata
for useful conversations.  Additional thanks are due to Kevin McGerty and Thomas Nevins for bringing their work
to the authors' attention.  
We are very grateful to the anonymous referee for many insightful comments and suggestions.
Finally, the authors are grateful to the
Mathematisches Forschungsinstitut Oberwolfach for its hospitality
and excellent working conditions during the initial stages of work on this paper.

\section{Conical symplectic resolutions}\label{sec:resolutions}
\renewcommand{\thetheorem}{\arabic{section}.\arabic{theorem}}
Let $\fM$ be a smooth, 
symplectic, complex algebraic variety.  By this
we mean that $\fM$ is equipped with a closed, nondegenerate, algebraic
2-form $\omega$.  Suppose further that $\fM$ is equipped with an action of the
multiplicative group $\bS := \cs$ such that $s^*\omega = s^n\omega$
for some integer $n\geq 1$.
We also assume that $\bS$ acts on the coordinate ring $\C[\fM]$ with
only non-negative weights, and that the trivial weight space
$\C[\fM]^\bS$ is 1-dimensional, consisting only of the constant
functions.  Geometrically, this means that the affinization $\fM_0 :=
\Spec \C[\fM]$ is a cone, and the $\bS$-action contracts $\fM_0$ to
the cone point $o\in \fM_0$.  Finally, we assume that the canonical
map $\nu\colon\fM\to\fM_0$ is a projective resolution of singularities.  (That is,
it must be projective and an isomorphism over the smooth locus of
$\fM_0$.)  We will refer to this collection of data as a {\bf conical
symplectic resolution of weight \boldmath$n$}.

Examples of conical symplectic resolutions include the following:
\begin{itemize}
\item $\fM$ is a crepant resolution of $\fM_0 = \C^2/\Gamma$, where $\Gamma$
is a finite subgroup of $\operatorname{SL}(2;\C)$.
The action of $\bS$ is induced by the inverse of the diagonal action on $\C^2$, and $n=2$.
\item $\fM$ is the Hilbert scheme of a fixed number of points on the crepant resolution of $\C^2/\Gamma$,
and $\fM_0$ is the symmetric variety of unordered collections of points on the singular space.
Once again, $\bS$ acts by the inverse diagonal action on $\C^2$, and $n=2$.
\item $\fM = T^*(G/P)$ for a reductive algebraic group $G$ and a parabolic subgroup $P$,
and $\fM_0$ is the affinization of this variety.  (If $G=\operatorname{SL}(r;\C)$, then $\fM_0$
is isomorphic to the closure of a nilpotent orbit in the Lie algebra of $G$.)
The action of $\bS$ is the inverse scaling action on the cotangent fibers, and $n=1$.
\item $\fM$ is a hypertoric variety associated to a simple, unimodular,
hyperplane arrangement in a rational vector space \cite{BD,Pr07}, and $\fM_0$ is the hypertoric
variety associated to the centralization of this arrangement.
If the arrangement is coloop-free, then it possible to define an $\bS$-action with $n=1$ \cite{HP04};
it is always possible to define an action with $n=2$ \cite{BeKu, BLPWtorico}.
\item $\fM$ and $\fM_0$ are Nakajima quiver varieties \cite{Nak94,Nak98}.  If the quiver is acyclic, then
there is a natural action with $n=1$ \cite[\S 5]{Nak94}; it is always possible to define
an action with $n=2$ \cite[\S 2.7]{Na}.
\item $\fM_0$ is a transverse slice to one Schubert variety
  $\operatorname{Gr}^\mu$ in an affine Grassmannian inside another
  $\operatorname{Gr}^\la$.  When $\la$ is a sum of minuscule
  coweights, this variety has a natural conical symplectic resolution
  constructed from a convolution variety; in most other cases, it
  seems to possess no such resolution.  This example is discussed in
  greater generality in \cite{KWWY}.
\end{itemize}

\begin{remark}
  The fifth class of examples overlaps significantly with each of the
  first four.  The first two examples are special
  cases of quiver varieties, where the underlying graph of the quiver is the extended Dynkin diagram corresponding
  to $\Gamma$.  When the group $G$ of the third example is $\operatorname{SL}(r;\C)$, then
  $T^*(G/P)$ is a quiver variety.  Finally, a hypertoric variety associated to a
  cographical arrangement is a quiver variety.
\end{remark}

\begin{example}\label{quotient construction}
Almost all of the examples above arise as symplectic quotients of vector spaces.  This applies
to the first, second, fourth, and fifth classes of examples, as well as the third class when 
$G=\operatorname{SL}(r;\C)$.  More precisely, 
let $G$ be a reductive algebraic group and $V$ a faithful linear representation
of $G$.  Then $G$ acts on the cotangent bundle $T^*V\cong V\times V^*$
with moment map $$\mu\colon V\times V^*\to\mg^*$$ given by the formula
$\mu(z,w)(x) := w(x\cdot z)$ for all $x\in\mg$, $z\in V$, and $w\in V^*$.
Choose a character $\theta$ of $G$, and let $\fM$ be the associated GIT quotient of $\mu^{-1}(0)$.
If $G$ acts freely on the semistable locus of $T^*V$, then $\fM$ is
symplectic and smooth.  Its affinization $\fM_0$ is a normal affine
variety, and the map $\nu\colon \fM\to \fM_0$ is automatically projective; if
it is furthermore birational, then it is a symplectic resolution of
singularities.  We also have a natural map from $\fM_0$ to the 
the categorical quotient of $\mu^{-1}(0)$
with no stability condition imposed, which is not always an
isomorphism, but will be in many interesting cases. 
The variety $\fM$ inherits a conical action of $\bS$ of weight 2 from the inverse scaling action on
$V\times V^*$.  If $V$ has no $G$-invariant functions, then we may take $\bS$ to act only on $V^*$ and obtain
a conical action of weight 1.
\end{example}

\begin{remark}
  All of these examples admit complete hyperk\"ahler metrics, and in fact we know of no examples
  that do not admit complete hyperk\"ahler metrics.  (Such examples do exist if we drop the hypothesis
  that $\fM$ is projective over $\fM_0$; these examples will appear in subsequent work by 
  the second author and Arbo.)
  The unit circle in $\bS$ acts by hyperk\"ahler isometries, but is
  Hamiltonian only with respect to the real symplectic form.  Our
  assumptions about the $\bS$-weights of $\C[\fM]$ translate to the
  statement that the real moment map for the circle action is proper and bounded below.
\end{remark}

\begin{proposition}\label{GR}
For all $i>0$, $H^i(\fM; \fS_\fM) = 0$, where $\fS$ is the structure sheaf\footnote{Throughout this paper we will use the symbol
$\fS$ for the structure sheaf of a variety.  We avoid the usual symbol $\cO$ because
this symbol will needed for the analogue of BGG category $\cO$ in the sequel to this paper \cite{BLPWgco}.}
of $\fM$.
\end{proposition}

\begin{proof}
This follows from the Grauert-Riemenschneider theorem; see, for example, \cite[2.1]{Kal00}.
\end{proof}

\begin{proposition}\label{cohomology}
All odd cohomology groups of $\fM$ vanish, and for all non-negative integers $p$
we have $H^{2p}(\fM; \C) = H^{p,p}(\fM; \C)$.  In particular, the class of the symplectic
form, which lies in $H^{2,0}(\fM; \C)$, is trivial.
\end{proposition}

\begin{proof}
The analogous result with $\fM$ replaced by a fiber of $\nu$ is proven in \cite[1.9]{Kal09}, 
thus it suffices to prove that $\nu^{-1}(0)$ is homotopy equivalent to $\fM$.  
To see this, let $\Phi\colon \M_0 \to \R$ be a real algebraic function 
which takes non-negative values and which is $\bS$-equivariant for
an action of the form $z \cdot t = |z|^k \cdot t$ of $\bS$ on $\R$, where $k$ is some positive integer. 
Such a function can be found of the form $\Phi = \sum_{i=1}^r |f_i|^{d_i}$, where $f_i$ are homogeneous generators of $\C[\M]$, with the grading induced by the action of $\bS$.
The argument from \cite[1.6]{Durf83} shows that the inclusions
$\nu^{-1}(0) \hookrightarrow  (\Phi\circ\nu)^{-1}[0, t] \hookrightarrow \M$  
induce isomorphisms of homotopy groups, and so are homotopy equivalences.
\end{proof}

\begin{remark}\label{core}
The subvariety $\nu^{-1}(0)\subset \fM$ is often called the {\bf core} or {\bf compact core},
see for example \cite[\S 4]{AB02} or \cite[\S 2.2]{Pr04}.  If $\fM$ is the cotangent bundle of a projective
variety $X$, then the core of $\fM$ is simply the zero section.  If $\fM$ is a crepant resolution of $\C^2/\Gamma$,
then the core of $\fM$ is a union of projective lines in the shape of the Dynkin diagram for $\Gamma$.
If $\fM$ is the Hilbert scheme of points on such a resolution, then the core of $\fM$ consists of configurations
supported on the core of the resolution.
If $\fM$ is the hypertoric variety associated to a real hyperplane arrangement, then the core of $\fM$
is a union of toric varieties corresponding to the bounded chambers of the arrangement \cite[6.5]{BD}.
\end{remark}

\subsection{Deformations}\label{sec:deformations}
We next collect some results of Namikawa and Kaledin on deformations of conical symplectic resolutions.
The following proposition is due to Namikawa (see Lemma 12, Proposition 13, and Lemma 22 of \cite{Namiflop}).

\begin{proposition}[Namikawa]\label{universal family}
  The variety $\fM$ has a universal Poisson deformation $\pi\colon\scrM\to H^2(\fM;\C)$ which
  is flat.  The variety $\scrM$ admits an action of $\bS$ extending the action on $\fM\cong \pi^{-1}(0)$,
  and $\pi$ is $\bS$-equivariant with respect to the weight $-n$ action on $H^2(\fM;\C)$.
\end{proposition}

\begin{remark}
A formal version of this result appears in the work of Kaledin and
Verbitsky \cite{KV02}; the work of Kaledin on twistor families
contains a very similar result, but not quite in the form we need.
\end{remark}


\begin{example}\label{quotient deformation}
Suppose that $\fM$ arises from the quotient construction of Example \ref{quotient construction}.
Let $\chi(\mg)$ denote the vector space of characters $\mg\to\C$,
and consider the Kirwan map $\mathsf{K}:\chi(\mg) \to \Ht$ that takes an integral character to the Euler class
of the induced line bundle on $\fM$.  If the Kirwan map is an isomorphism (this is known when $\fM$
is a hypertoric variety, and conjectured in all cases), 
then $\scrM$ is isomorphic to the GIT quotient of $\mu^{-1}((\mg^*)^G)$,
with the map to $\Ht\cong\chi(\mg)\cong (\mg^*)^G$ given by $\mu$.
\end{example}

Given any class $\eta\in H^2(\fM; \C)$, let $\scrM_\eta:= \scrM\times_{H^2(\fM;\C)}\aone$, where $\aone$ maps
to $H^2(\fM;\C)$ via the linear map that takes 1 to $\eta$.  Of particular interest is the case where $\eta$
is the Euler class of a line bundle $\cL$ on $\fM$.
In this case, the following result follows
from the work of Kaledin \cite[1.4-1.6]{KalPois}.

\begin{proposition}[Kaledin]\label{twistor def}
There exists a unique $\bS$-equivariant Poisson line bundle $\mathscr{L}$ on $\scrM_\eta$ extending the bundle $\cL$ on $\fM$
such that the Poisson action of the coordinate function $t\in \C[\aone]$ on the space of sections of $\mathscr{L}$
is the identity.
\end{proposition}

\begin{remark}
Kaledin refers to the pair $(\scrM_\eta,\mathscr{L})$ as a {\bf twistor family}.
The second half of the proposition can be stated more geometrically as the
condition that the complement $\mathscr{L}^\times$ of the zero section in the total space of $\mathscr{L}$
(the relative spectrum of the algebra sheaf $\bigoplus_{m\in \Z} \mathscr{L}^m$) carries a symplectic structure
coinducing the Poisson structure on $\scrM_\eta$ such that the
Hamiltonian vector field $\{t,-\}$ is the infinitesimal rotation of the fibers.  In
particular, $\fM$ is the symplectic reduction of $\mathscr{L}^\times$ by this Hamiltonian vector field.
\end{remark}

Kaledin also tells us that $\scrM_\eta$ is symplectic over $\aone$, and he computes the class of the 
relative symplectic form 
as follows \cite[1.7]{KalPois}.

\begin{proposition}[Kaledin]\label{twistor}
The Poisson structure on $\scrM_\eta$ is nondegenerate over $\aone$, and the relative symplectic form $\omega_{\!\scrM_\eta}\in\Omega^2(\scrM_\eta/\aone)$
satisfies $$[\omega_{\!\scrM_\eta}] = t\eta \in H_{DR}^2(\scrM_\eta/\aone) \cong H^2(\fM; \C)[t].$$
\end{proposition}

\begin{remark}\label{rmk:twistor}
Proposition \ref{twistor} may be easily extended to say that $\scrM$ has a nondegenerate Poisson structure over $H^2(\fM;\C)$
with relative symplectic form
$$[\omega_{\!\scrM}] = I \in H_{DR}^2(\scrM/ H^2(\fM;\C)),$$
where we identify the latter cohomology group with the space of polynomial maps from $H^2(\fM; \C)$ to itself,
and $I$ is the identity map.
\end{remark}

Note that the $\bS$-action may be used to identify all of the nonzero fibers of $\fM_\eta$ with a single symplectic
variety $\scrM_\eta(\infty) := \left(\scrM_\eta \smallsetminus \fM\right)/\;\bS$.  
The following result of Kaledin \cite[2.5]{KalDEQ} will be crucial to our proof of Proposition \ref{asympt-local}.

\begin{proposition}[Kaledin]\label{ample-affine}
If $\cL$ is ample, then $\scrM_\eta(\infty)$ is affine.
\end{proposition}

\subsection{The Weyl group}\label{sec:Weyl-gp-def}
Next, we put some results of Namikawa \cite{NamiaffII} into a form which is 
convenient for our purposes.
Let $\{\Sigma_j\}$ be the codimension 2 connected components of the smooth part of the
singular locus of $\fM_0$.  At any point $\sigma_j\in \Sigma_j$, there exists a normal slice to $\Sigma_j$
at $\sigma_j$ which is isomorphic to a Kleinian singularity, thus the preimage $\nu^{-1}(\sigma_j)\subset \fM$
is a union of projective lines in the shape of a simply-laced finite-type Dynkin diagram $D_j$.
The monodromy representation of the fundamental group $\pi_1(\Sigma_j)$ defines an action on $D_j$
by diagram automorphisms.  Let $W_j$ be the centralizer of $\pi_1(\Sigma_j)$ in the Coxeter group
associated to $D_j$, and let $W :=\prod W_j$.  We will call $W$
the {\bf Weyl group} of $\fM$ (see Remark \ref{Weyl} for motivation).
Namikawa constructs an action of $W$ on $H^2(\fM; \R)$; he
proves that the natural restriction map
\begin{equation}\label{restriction-eq}
H^2(\fM; \R) \to\bigoplus_j H^2\!\left(\nu^{-1}(\sigma_j); \R\right)^{\pi_1(\Sigma_j)}
\end{equation}
is $W$-equivariant and that $W$ acts trivially on the kernel \cite[1.1]{NamiaffII}.\footnote{
This statement is equivalent to Namikawa's statement that the map 
$\iota$, which he defines in the proof of his theorem, is an isomorphism.}

\begin{remark}\label{Weyl}
Let $G$ be the reductive algebraic group associated to a simply-laced finite-type Dynkin diagram $D$,
and let $B$ be a Borel subgroup.  If $\fM = T^*(G/B)$, then $\fM_0$
is isomorphic to the nilpotent cone in $\mg := \operatorname{Lie}(G)$.  The singular locus of $\fM_0$
is irreducible, and its smooth locus is called the subregular nilpotent orbit. 
The normal slice to the subregular orbit is isomorphic to the Kleinian singularity associated to $D$ and $W$ is
isomorphic to the Weyl group of $G$.  
The action of $W$ on $H^2(\fM; \C)$ is isomorphic to the action on
the dual of a Cartan subalgebra of $\mg$ and the restriction map \eqref{restriction-eq} is an isomorphism.
\end{remark}

Let $\scrN := \Spec \C[\scrM]$.\footnote{It would be natural to use the notation $\scrM_0$ rather than $\scrN$,
but unfortunately that notation has already been used in the previous section to mean something else.}  
Then the map $\pi:\scrM\to\Ht$ factors canonically through $\scrN$.
Namikawa \cite{Namiaff, NamiaffII, NamikawaNote} proves that the action of $W$ on $\Ht$ lifts to a symplectic action on $\scrN$,
and that the quotient map $\scrN/W\to \Ht/W$ is the universal Poisson deformation of the central fiber $\fM_0$.

\begin{remark}
The quotient $H^2(\fM; \C)/W$ is itself a vector space, which may be
identified by a theorem of Namikawa \cite[1.1]{NamiaffII} with the Poisson cohomology
group $H\! P^2(\fM_0; \C)$ as defined in \cite[\S 2]{Namiflop}.
\end{remark}

\subsection{Birational geometry}
\label{sec:systems-resolutions}

Let $P := \Pic(\fM)$ be the Picard group of $\fM$.  Proposition \ref{cohomology} tells us that 
$$P_\R:= P\otimes_\Z\R \cong H^2(\fM; \R);$$ in particular, $P$ has finite rank.
A class $\eta\in P$ is called {\bf movable} if the associated line bundle is globally generated away from a codimension
2 subvariety of $\fM$.  Let $\Mov\subset P_\R$ be the movable cone (the convex hull of the images of 
movable classes), and let $\overline\Mov$ be its closure.

\begin{proposition}\label{fundamental}
The cone $\overline\Mov\subset P_\R$ is a fundamental domain for $W$.
\end{proposition}

\begin{proof}
Consider the restriction map \eqref{restriction-eq}.  Since $W$ acts trivially on the kernel, any fundamental domain
for the action on the target pulls back to a fundamental domain for the action on the source.
The space $H^2\!\left(\nu^{-1}(\sigma_j); \R\right)^{\pi_1(\Sigma_j)}$
may be identified $W_j$-equivariantly with the real part of the dual of the Cartan subalgebra of the Lie algebra determined by 
the Dynkin diagram $D_j$.  The standard fundamental domain is the positive Weyl chamber, which may be characterized
as the set of classes that are non-negative on the fundamental classes of the components of $\nu^{-1}(\sigma_i)$.

We have thus reduced the proposition to showing that a class $\eta\in P$ is movable if and only if $\eta\cdot E\geq 0$
for every curve $E\subset \fM$ such that $E$ is a component of $\nu^{-1}(\sigma_j)$ for some $j$.
Suppose first that $\eta\cdot E<0$ for some such curve $E$.
Since $E\cong\mathbb{P}^1$, this implies that every section of the line bundle associated to $\eta$ vanishes on $E$,
and therefore on the component of $\nu^{-1}(\Sigma_j)$ containing $E$.  Since this component has codimension 1 in $\fM$,
$\eta$ cannot be movable.  On the other hand, suppose that $\eta\cdot E\geq 0$ for every such curve.
This implies that the associated line bundle is globally generated over $\nu^{-1}(\Sigma_j)$ for every $j$.
It is obviously globally generated over the preimage of the smooth locus of $\fM_0$, since $\fM_0$ is affine.
It is therefore globally generated over an open set whose complement has codimension $2$, thus $\eta$ is movable.
\end{proof}

We will wish to consider not just a single conical symplectic resolution, but
rather a collection of varieties $\fM_1,\dots, \fM_\ell$, all conical
symplectic resolution of the same cone $\fM_0$; for any two of these,
there is a birational map $f_{ij}\colon \fM_i \dasharrow\fM_j$, given
by composing the resolution of $\fM_0$ by $\fM_i$ with the inverse of
the resolution by $\fM_j$.

\begin{proposition}\label{different resolutions}
  Each $\fM_i$ contains an open subvariety $U_i$ with
  $\codim(\fM_i\setminus U_i)\geq 2$ such that
  $f_{ij}$ induces an isomorphism $U_i\cong U_j$ for all $j$, and thus a canonical
  isomorphism between Picard groups of the different
  resolutions.
\end{proposition}

\begin{proof}
Since the spaces in question are symplectic and therefore Calabi-Yau,
there exist open subsets $U_i^j\subset \fM_i$ and $U_j^i\subset \fM_j$
with complements of codimension $\geq 2$ such that $f_{ij}$ induces an isomorphism
from $U_i^j$ to $U_j^i$; see, for example \cite[4.2]{Kawa}. 
Briefly, one can take any resolution $Q\to
\fM_0$ (no longer symplectic!) which factors through $\fM_i$ and $\fM_j$ and
pull out all irreducible components of the canonical divisor
of $Q$; the remainder of $Q$ maps isomorphically to subsets of $\fM_i$
and $\fM_j$ with complements of codimension $\geq 2$; there
is a canonical largest such set, so we can take $U^j_i$ to be that one.
We then let $U_i := \bigcap_{j=1}^\ell U_i^j$.
\end{proof}

Note that any class $\eta\in P$ which is movable
for $\fM_i$ is also movable for $\fM_j$, thus we have a well-defined movable cone $\Mov\subset P_\R$.
The following result of Namikawa \cite{NamikawaNote} can be roughly summarized by the statement
that $\fM$ is a {\bf relative Mori dream space} over $\fM_0$ \cite[2.4]{AW}.

\begin{theorem}[Namikawa]\label{dream}
There are finitely many isomorphism classes of conical symplectic resolutions of $\fM_0$.
Furthermore, there exists a finite collection $\cH$ of hyperplanes in $P_\R$, preserved by the action of $W$, with the following
properties:
\begin{itemize}
\item For each conical symplectic resolution $\fM$, the ample cone of $\fM$ is a chamber of $\cH$
(and different resolutions have different ample cones).
\item The union of the closures of these ample cones is equal to $\overline\Mov$.
\item The union $\bigcup_{H\in\cH}H_\C\subset P_\C \cong \Ht$ is precisely equal to the locus over which the map $\scrM\to\scrN$
fails to be an isomorphism.   Equivalently, it is the locus
over which the fibers are not affine.
\end{itemize}
\end{theorem}

\begin{remark}\label{chambers}
Note that, by Proposition \ref{fundamental} and Theorem \ref{dream},
the chambers of $\cH$ are in bijection with the set of pairs $(\fM, w)$, where $\fM$ is a conical symplectic
resolution of $\fM_0$ and $w$ is an element of $W$.  This bijection sends the pair $(\fM,w)$ to the $w$ translate of
the ample cone of $\fM$.
\end{remark}

\begin{remark}
If $\fM$ is a quotient as in Example \ref{quotient construction} and the 
Kirwan map of Example \ref{quotient deformation} is an isomorphism in degree 2, then
the chambers of $\cH$ are exactly the top dimensional cones in the GIT fan in $\chi(G)_\R\cong P_\R$.
\end{remark}

\section{Quantizations}
\label{sec:quantizations}


Throughout the remainder of the paper, we will always use  $S$ to
denote a scheme of finite type over $\C$ and $\fX$ to
denote a smooth finite type $S$-scheme, projective over an affine scheme\footnote{In \cite{KalPois}, Kaledin uses 
the terminology ``algebraically convex", but in other papers this term
  allows the map to only be {\it proper}; we emphasize
  that projectivity is essential.} $\fX_0$, equipped with a symplectic form
$\omega_\fX\in \Omega^2(\fX/S)$.  
After
Section \ref{sec:period}, we will also assume
throughout that $\fX$ and $S$ carry compatible actions of $\bS$ such that:
\begin{itemize}
\item The function algebra $\C[\fX]$ has no elements of negative $\bS$-weight.
\item The symplectic form satisfies $s^*\omega_\fX = s^n\omega_\fX$ for some positive integer $n$.
Equivalently, the induced Poisson bracket $\{-,-\}$ on $\fS_{\fX}$ is homogeneous of weight
  $-n$.  
\item We have $H^1(\fX;\fS_\fX)^{\bS}\cong H^2(\fX;\fS_\fX)^{\bS}=0$.
\end{itemize}

The cases that will be of primary interest to us arise in connection with a conical
symplectic resolution $\fM$:
 \begin{itemize}
 \item $\fX = \fM$ and $S$ is a point
 \item $\fX=\scrM_\eta$ and $S=\aone$
 \item $\fX = \scrM$ and $S = H^2(\fM;\C)$.
 \end{itemize}
 Here $\scrM_\eta$ is the twistor deformation and $\scrM$
 is the universal deformation, as in Section
 \ref{sec:deformations}.  We'll use these notations consistently
 throughout the paper. Each of these examples satisfies our assumptions for $\fX$; the
 only assumption which needs explanation is the cohomology
 vanishing, which holds for all three as a
 consequence of Grauert-Riemenschneider.

\subsection{The period map}\label{sec:period}
A {\bf quantization} of $\fX$ consists of 
\begin{itemize}
\item sheaf $\cQ$ of flat $\pi^{-1}\fS_{\! S}[[h]]$-algebras
on $\fX$, complete in the $h$-adic topology
\item an isomorphism from $\cQ/h\cQ$ to the structure sheaf $\fS_\fX$ of $\fX$
\end{itemize}
satisfying the condition that, if $f$ and $g$ are functions over some open set and $\tilde f$ and $\tilde g$
are lifts to $\cQ$, the image in $\fS_\fX \cong \cQ/h\cQ\cong h\cQ/h^2\cQ$ of the element $[\tilde f, \tilde g]\in h\cQ$
is equal to the Poisson bracket $\{f,g\}$.  Note that while we have
assumed that $\fX$ is smooth over $S$ and that the base field is $\C$, the notion of a quantization
makes sense for any Poisson variety.

  If $H^1(\fX;\fS_\fX)\cong
H^2(\fX;\fS_\fX)=0$ then, Bezrukavnikov and Kaledin \cite[1.8]{BK04a} show that the set of quantizations of $\fX$ is in natural bijection
via the {\bf period map} with the vector space\footnote{Following the
  conventions of \cite{BK04a}, we will mean here the $h$-adic
  completion of $H^2_{DR}(\fX/S; \C)\otimes \C[[h]]$.  This applies
  whenever we use the notation $V[[h]]$ for some vector space
  $V$.} $$[\omega_\fX] + h\cdot H^2_{DR}(\fX/S; \C)[[h]].$$ More
concretely, by Propositions \ref{cohomology} and \ref{twistor} and
Remark \ref{rmk:twistor}, 
\begin{itemize}
\item the period map for $\fM$ takes values in $h\cdot
  H^2(\fM;\C)[[h]]$
\item the period map for $\scrM_\eta$ takes values
  in $t\eta\, +\, h\cdot H^2(\fM;\C)[t][[h]]$
\item  the period map for
  $\scrM$ takes values in $I\, +\; h\cdot
  \left(H^2(\fM;\C)\otimes\C[H^2(\fM;\C)]\right)\![[h]]$.
\end{itemize}
The (unique) quantization with period $[\omega_\fX]$ is called the {\bf canonical quantization} of $\fX$.

Let $\scrQ$ be a quantization of $\scrM_\eta$.  There is an obvious way to recover a quantization of
$\fM$ from $\scrQ$:  if we divide by the ideal sheaf of $\pi^{-1}(0)$, we obtain a sheaf supported on $\pi^{-1}(0)\cong\fM$,
and this sheaf is clearly a quantization.
However, this is not the only quotient of $\scrQ$ which is supported on $\pi^{-1}(0)$. 
Fix an element $P(h) \in h\cdot \C[[h]]$.
The map from $\C[t]$ to $\C[[h]]$ taking $t$ to $P(h)$ induces a map from
$\Delta:= \Spec\C[[h]]$ to $\aone$ sending the closed point to 0, and therefore a section $\si_{\!
  P}$ of  the projection $\aone\times \Delta\to
\Delta$ which sends the closed point of $\Delta$ to $0$.
Dividing $\scrQ$ by the ideal sheaf in $\fS_{\aone}[[h]]$ of
the image of $\si_{\! P}$ also gives a quantization of $\fM$.  
Following Bezrukavnikov and Kaledin, we denote this quantization by $\si_{\! P}^*\scrQ$.
Note that the first construction in this paragraph corresponds to the
choice $P=0$.

More generally, for any quantization $\scrQ$ of $\fX/ S$, let 
$\sigma\colon \Delta\to \Delta\times S$ be any section of the
projection $\Delta\times S\to \Delta$.  If $*$ is the unique closed point of $\Delta$
and $\sigma(*) = (*,s)$, then we may define $\sigma^*\scrQ$
to be the quotient of $\scrQ$ by the ideal sheaf of this section, thought
of as a sheaf on $\pi^{-1}(s)$.

Let $\scrQ_R$ be the quantization of $\fX$ with period given by
$R(h)\in H^2_{DR}(\fX/S; \C)[[h]]$.  If $\fX=\scrM_\eta$, we can think
of this as a two variable function $R(t,h)\in t\eta\, +\, h\cdot
H^2(\fM;\C)[t][[h]]$.  The following proposition is an easy
modification of \cite[6.4]{BK04a}; it follows immediately from the naturality of
periods under pullback.

\begin{proposition}[Bezrukavnikov and Kaledin]\label{general pullback}
The period of the quantization $\si^*\scrQ_R$ is $\sigma^*R(h)\in
H^2(\pi^{-1}(s);\C)[[h]]$.  In particular, if $\fX=\scrM_\eta$ and $S = \aone$, then $\si_{\! P}^{*}\scrQ_R$ has period $R(P(h),h)\in h\cdot
  H^2(\fM;\C)[[h]]$.
\end{proposition}

Let us collect one more fact about quantizations which will be
important for us. If $\cQ$ is a quantization of $\fX/S$, we let $\cQ^{\op}$ be the
opposite algebra of $\cQ$, thought of as a $\C[[h]]$-algebra with the
action twisted by the automorphism $h\mapsto -h$; this convention is
necessary to assure that $\cQ^{\op}$ again quantizes the same Poisson structure. 
\begin{proposition}\label{opp-zero}
  If $P(h)\in [\omega_\fX] + h\cdot H^2_{DR}(\fX/S; \C)[[h]]$ is the
  period of $\cQ$, then the period of $\cQ^{\op}$ is $P(-h)$. 
\end{proposition}
\begin{proof}
  A proof of this fact is given in the proof of
  \cite[2.3.2]{Losq}, but the result is not stated as a theorem.  As defined in \cite[4.1]{BK04a},
  the period map is the localization of a universal class $c\in H^2((\Aut
  D,\Der D),h\C[[h]])$  in the
  cohomology of the Harish-Chandra pair $(\Aut D,\Der D)$, where $D$
  is the Weyl algebra.  The existence of a particular anti-automorphism sending
  $h\mapsto -h$ and $c(h)\mapsto c(-h)$ given in \cite{Losq} shows
  that the period transforms the same way.
\end{proof}

\begin{remark}
In this Remark, contrary to our usage elsewhere, we will not assume
{\it a priori} that the symbols $\fM$ and $\fX$ denote smooth
varieties.  Not every symplectic variety (in the sense of Beauville \cite{Beau})
admits a symplectic resolution; for
example, closures of non-Richardson nilpotent orbits do not \cite{Fu}.  
On the other hand, every symplectic variety has
a crepant partial resolution $\fM$ which is terminal and
$\Q$-factorial; this is again a symplectic variety, since it is
dominated by some resolution of $\fM_0$. The fact that this variety is
$\Q$-factorial means that it cannot be resolved further without
introducing discrepancy: a crepant partial resolution of $\fM$ would
have to be isomorphic to $\fM$ in codimension 1 so their group of Weil divisors
would be the same; thus an ample line bundle on the resolution would have to
correspond to a Weil divisor on $\fM$, some power of which is a
Cartier divisor, showing that the resolution is in fact $\fM$.

While the theory of periods we have discussed thus far cannot be
directly applied to $\fM$, it can be applied to the smooth locus
$\becircled{\fM}$. 
More generally, let $\fX/S$ be a convex symplectic (not necessarily
smooth) variety with
terminal singularities \cite[\S 1]{Namiflop}, and
$\becircled{\fX}$ its smooth locus.  
 As noted by Namikawa in the proof of
\cite[Lemma 12]{Namiflop}, $H^1(\becircled{\fX};\fS_{\becircled{\fX}})=H^2(\becircled{\fX};\fS_{\becircled{\fX}})=0$,
so $\becircled{\fX}$ satisfies our running assumptions.  By
\cite[1.8]{BK04a}, the quantizations of $\becircled{\fX}$ are in
bijection with $[\omega_{\fX}] + h\cdot H^2_{DR}(\becircled{\fX}/S; \C)[[h]]$.
Let $i\colon \becircled{\fX}\to \fX$ be the
inclusion map.

\begin{proposition}
If $\becircled\cQ$ is a quantization of $\becircled \fX$, then $i_*\becircled\cQ$ is a quantization of $\fX$.
If $\cQ$ is a quantization of $\fX$, then $i^{-1}\cQ$ is a quantization of $\becircled\fX$.
These two operations induce inverse bijections
between isomorphism classes of quantizations of $\becircled\fX$ and $\fX$.
\end{proposition}
\begin{proof}
The fact that $i_*\becircled\cQ$ is a quantization follows from normality of symplectic varieties;
the fact that $i^{-1}\cQ$ is a quantization is trivial, as is the isomorphism $i^{-1}i_*\becircled\cQ\cong\becircled\cQ$.
In the other direction, the natural map
$i_*i^{-1}\cQ\to \cQ$ is an isomorphism mod $h$, and thus is an isomorphism by Nakayama's lemma.
\end{proof}

In most sections of this paper (with the exception of Section \ref{sec:index}), we could 
allow our conical symplectic resolutions to be terminal and $\Q$-factorial rather than smooth.  For ease of exposition,
however, we will continue to assume smoothness.
\end{remark}

\subsection{$\bS$-structures}\label{S-structures}
From this point forward, we will assume
that $\fX$ and $S$ carry compatible actions of $\bS$ such that:
\begin{itemize}
\item The function algebra $\C[\fX]$ has no elements of negative $\bS$-weight.
\item The symplectic form satisfies $s^*\omega_\fX = s^n\omega_\fX$ for some positive integer $n$.
Equivalently, the induced Poisson bracket $\{-,-\}$ on $\fS_{\fX}$ is homogeneous of weight
  $-n$.  
\item We have $H^1(\fX;\fS_\fX)^{\bS}\cong H^2(\fX;\fS_\fX)^{\bS}=0$.
\end{itemize}
In this section we define the notion of an $\bS$-structure on a quantization of $\fX/S$,
and we consider the question of which quantizations carry
$\bS$-structures.

Let $a\colon\bS\times\fX \to\fX$ be the action map, let $p\colon\bS\times\fX
\to\fX$ be the projection onto $\fX$, and let $e\colon \bS\times \fX\to \bS$ be
the projection onto $\bS$.
If $\cQ$ is a quantization of $(\fX, \omega_\fX)$, then the 
naive pullback $a^*\cQ:=a^{-1}\cQ\otimes_{\C[[h]]}e^{-1}\fS_{\bS}[[h]]$ 
is a quantization of $\fX\times\bS$ over $\bS$ with the relative symplectic form
$a^*\omega_\fX=z^np^*\omega_\fX$, where $z$ is the coordinate function
on $\bS$.  Since
forms are contravariant and bivectors covariant, the corresponding
Poisson brackets are related by $\{-,-\}_a=z^{-n}\{-,-\}_p$.
As long as the Poisson bracket on $\fX$ is nontrivial,
the sheaves $a^*\cQ$ and $p^*\cQ$ are quantizations of different Poisson brackets on $\fX\times\bS$, 
thus they are never isomorphic.

This difference between the two Poisson brackets can be resolved by twisting the action of $h$.
More precisely, let $a^*_{\text{tw}}\cQ:=a^{-1}\cQ\otimes_{\C[[h]]}e^{-1}\fS_{\bS}[[h]]$, where this time the action of
$\C[[h]]$ on $e^{-1}\fS_\bS[[h]]$ is given by sending $h$ to $z^nh$.  Put differently, $a^*_{\text{tw}}\cQ$ and $a^*\cQ$
are isomorphic as sheaves of vector spaces, but the endomorphism given by 
multiplication by $h$ in $a^*_{\text{tw}}\cQ$ corresponds to the endomorphism given by
multiplication by $z^{-n}h$ in $a^*\cQ$.
Then $a^*_{\text{tw}}\cQ$ is a quantization of the Poisson
bracket $z^n\{-,-\}_a=\{-,-\}_p$, that is, corresponding to the
relative  symplectic form $p^*\omega_\fX$.  

An {\bf $\bS$-structure} on $\cQ$ is an isomorphism $a_\text{tw}^*\cQ\cong p^*\cQ$ as
$(\id_{\bS}\times \pi)^{-1}\fS_{\bS\times S}[[h]]$-algebras, satisfying the natural cocycle condition.
That is, the above isomorphism induces an isomorphism $s^*\cQ\cong \cQ$ for every $s\in\bS$,
and we require that for any three elements of $\bS$ with $s\cdot s'\cdot s'' = 1$, 
the composition of the three isomorphisms is the identity.  In
\cite{Losq}, this is called a ``grading'' on the quantization.  We will often refer to a quantization endowed with an $\bS$-structure as an {\bf $\bS$-equivariant} quantization.

As a
general principle, quantizations have $\bS$-structures whenever their
period does not obstruct this possibility.  More precisely, Losev \cite[2.3.3]{Losq} proves 
the following result.\footnote{Losev assumes that $n=2$, but his proof works for arbitrary $n$.}

\begin{proposition}[Losev]\label{S-structure}
A quantization of $\fX$ admits an $\bS$-structure if and only if 
its period lies in the vector space $[\omega_{\fX}] + h\cdot H^2_{DR}(\fX/S; \C)\subset [\omega_{\fX}] + h\cdot H^2_{DR}(\fX/S; \C)[[h]]$, in which case its $\bS$-structure
is unique.
\end{proposition}  
As noted in \cite[\S 6.1]{BK04a}, as long as we have the assumptions
$H^{1}(\fX;\fS_{\fX})^\bS=H^{2}(\fX;\fS_{\fX})^\bS=0$, the variety
$\fX$ is $\bS$-equivariantly admissible.  Even if there are vectors
of non-zero weight 
in $H^{1}(\fX;\fS_{\fX})$ or $H^{2}(\fX;\fS_{\fX})$, we can still apply the
theory of \cite{BK04a} to $\bS$-equivariant quantizations; in particular, every period in
$[\omega_{\fX}] + h\cdot H^2_{DR}(\fX/S; \C) $
has a corresponding unique $\bS$-equivariant quantization.  

\subsection{The section ring}\label{section ring}
Let
$\cQ$ be an $\bS$-equivariant quantization of $\fX$.  Define 
$$\cD(0) := \cQ[\hon],\qquad\cD:=\cQ[\hmon],
\and
\cD(m) := h^{\nicefrac{-m}{n}}\cD(0)\subset\cD\;\text{for all $m\in\Z$}.$$
We will frequently abuse notation by referring to $\cD$ as a quantization of $\fX$.

Let $A := \Gamma_\bS(\cD)$
be the ring of $\bS$-invariant sections of $\cD$.  This ring
inherits a $\mathbb{Z}$-filtration $$\ldots \subset A(-1)\subset A(0)\subset A(1)\subset \ldots\subset A$$
given by putting 
$$A(m) := \Gamma_\bS\big(\cD(m)\big).$$
The associated graded of $A$ may be canonically identified with $\C[\fX]$ as a 
$\Z$-graded ring via the maps
$$A(m) = \Gamma_\bS\big(\cD(m)\big)
\overset{\cdot h^{\nicefrac{m}{n}}}\longrightarrow \Gamma\big(\cD(0)\big) \twoheadrightarrow
\Gamma\big(\cD(0)/\cD(-1)\big) \cong
\Gamma(\fS_\fX) = \C[\fX].$$

Many of the examples of conical symplectic resolutions we gave at the beginning of 
Section \ref{sec:resolutions} admit
quantizations for which the ring $A$ is of independent interest.  (In all of these examples $S$ is a point.)

\begin{itemize}
\item Let $\Gamma\subset \operatorname{SL}_2(\C)$ be a finite subgroup.
Any quantization of
the Hilbert scheme of $m$ points on a crepant resolution of
$\C^2/\Gamma$ has its invariant section ring $A$ isomorphic to a spherical symplectic 
reflection algebra for the wreath product $S_m\wr\Gamma$, with parameters
corresponding to the period of the quantization \cite[1.4.4]{EGGO}, \cite[1.4]{Gorremark}.
\item Let $G$ be a reductive Lie group and $B\subset G$ a Borel
  subgroup.  Then each quantization of $T^*(G/B)$ has its invariant section ring $A$ isomorphic to a 
  central quotient of the universal enveloping algebra $U(\mg)$.  All central quotients arise this way,
  and two quantizations give the same central quotient if their periods are related by the action
  of the Weyl group \cite[Lemma 3]{BB}.
\item Many quantizations of a resolution of a Slodowy slice to a
  nilpotent orbit in $\mg$ 
have invariant section ring $A$ isomorphic to a central quotient of a
finite W-algebra; the cases where every quantization has this property
(which includes all slices in type A) are classified by Ambrosio,
Carnovale, Esposito and Topley \cite{ACET}.  Again,
all central quotients of the W-algebra arise this way,
  and two quantizations give the same central quotient if their periods are related by the action
  of the Weyl group \cite[6.4]{Sk}.
\item Any quantization of a hypertoric variety 
  has its invariant section ring $A$ isomorphic to a central quotient of the
  hypertoric enveloping algebra.  Once more, all central quotients arise this way,
  and two quantizations give isomorphic central quotients if their periods are related by the action
  of the Weyl group \cite[\S 5]{BeKu},
  \cite[5.9]{BLPWtorico}.
\item In \cite{KWWY}, it is conjectured that the algebra arising from
  the slices in the affine Grassmannian can be described as a quotient
  of a {\bf shifted Yangian}, a variant of the usual Yangian of Drinfeld. 
\end{itemize}

Consider the universal Poisson deformation $\pi\colon\scrM\to\Ht$ of $\fM$.
Let $\scrD$ be the canonical quantization of $\scrM$, and let $\scrA := \secs(\scrD)$
be its invariant section algebra.
The $\pi^{-1}\fS_{H^2(\fM;\C)}$-structure on $\scrD$ induces a map
$$c\colon \C[H^2(\fM;\C)]\to \Gamma(\scrM;\scrD)$$ which is
$\bS$-equivariant for the weight $n$ action on
$\C[H^2(\fM;\C)]$.  In particular, if $x\in \Ht^*$ is a linear
function on $H^2(\fM;\C)$, we have that $h^{-1}c(x)\in \scrA$.

Let $\la\in\Ht$ be the period of $\cD$.  By Proposition \ref{general pullback}, $\cD = \sigma_{h\la}^*\scrD$,
and this induces a restriction map from $\scrA$ to $A$.

\begin{proposition}\label{restriction of sections}
The map from $\scrA$ to $A$ is surjective with kernel generated by $h^{-1}c(x)-\la(x)$ for all $x\in\Ht^*$.
\end{proposition}

\begin{proof}  Let $\C_\la$ be the evaluation module at $\la$ of
  $\C[\Ht]$.  The sheaf $\cD$ can be rewritten as
  the cohomology of the tensor product of
  $\scrD$ with the Koszul resolution of
  $\C_\la$.  Thus, the sheaf cohomology of $\cD$
  is the hypercohomology of this complex. Filtering this complex by
  degrees in the Koszul resolution, we obtain the spectral 
  sequence \[\operatorname{Tor}^i_{\C[\Ht]}(H^j(\scrD)^\bS,\C_\la)\Rightarrow
H^{j-i}(\cD) ^\bS\]
converging to the cohomology of $\cD$. 
Since $\scrD$ has
trivial higher cohomology, this spectral sequence collapses
immediately, and we obtain the desired isomorphism.
\end{proof}

\vspace{-\baselineskip}
\begin{lemma}\label{codimension-3}
Let $\fX$ be a smooth symplectic variety over a smooth base
$S$.  Let $i\colon U\hookrightarrow \fX$ be an open inclusion,
and let $d$ be the codimension of the complement of $U$.
\begin{itemize}
\item If $d\geq 2$, then for any
  quantization $\cQ$ of $\fX$, the restriction $i^*\cQ$ to $U$ is a
  quantization of $U$ with $\Gamma(U;i^*\cQ)\cong\Gamma(\fX;\cQ)$.
\item If $d\geq 3$, then for any
  quantization $\cQ'$ of $U$, the pushforward $i_*\cQ'$ is a
  quantization of $\fX$ with $\Gamma(U;\cQ')\cong\Gamma(\fX;i_*\cQ')$.
\end{itemize}
\end{lemma}

\begin{proof} 
Let $j\colon \fX\setminus U\to \fX$ be the inclusion.  As usual for complementary closed and open embeddings, we have an
exact triangle $j_*j^!\fS_\fX \to \fS_\fX\to i_*i^*\fS_\fX\to j_*j^!\fS_\fX[1]$.
The induced long exact sequence takes the form of a short exact sequence
\[0\to \fS_\fX\to i_*\fS_U \to  j_*\mathbb{R}^{1}j^!\fS_\fX\to 0\]
along with isomorphisms
$\mathbb{R}^ki_*\fS_{U}\cong j_*\mathbb{R}^{k+1}j^!(\fS_\fX)$
for all $k>0$.
The local cohomology sheaf $\mathbb{R}^kj^!\fS_\fX$ vanishes
for all $k<d$, so we may conclude that $i_*\fS_{U}\cong
\fS_{\fX}$ if $d\geq 2$, and $\mathbb{R}^1i_*\fS_{U}=0$ if $d\geq 3$.

Assume that $d\geq 2$, and
consider a quantization $\cQ$ on $\fX$.  It is clear that
$i^*\cQ$ is a quantization of $U$, so we need only show that the
sections are unchanged.  For each $m\geq 0$, the natural map
$\cQ/h^m\cQ\to i_*i^*(\cQ/h^m\cQ)$ is an isomorphism; this follows
from induction and the five-lemma applied to
the diagram:
\[\tikz[->,thick]{
\matrix[row sep=10mm,column sep=20mm,ampersand replacement=\&]{
\node (a) {$h \cQ/h^m\cQ$}; \& \node (b) {$\cQ/ h^m\cQ$}; \& \node (c) {$\fS_{\fX}$}; \\
\node (d) {$i_*i^*(h \cQ/h^m\cQ)$}; \& \node (e) {$i_*i^*(\cQ/ h^m\cQ)$}; \& \node (f) {$i_*i^*(\fS_{\fX})$}; \\
};
\draw (a) -- (b) ; 
\draw (b) -- (c) ;
\draw (a) -- (d) ; 
\draw (b) -- (e) ; 
\draw (c) -- (f) ;
\draw (d) -- (e) ; 
\draw (e) -- (f) ; 
}\]
Since $U$ is open, $i_*i^*$ commutes with projective limits, so we have an isomorphism
$i_*i^*\cQ\cong \cQ$.  The isomorphism of sections of $\cQ$ and $i^*\cQ$ now follows by the
functoriality of push-forward.

Now assume that $d\geq 3$, and let $\cQ'$ be a
quantization on $U$.  
  The flatness of $ i_*\cQ'$ is automatic, so
we need only show that $i_*\cQ'/ i_*(h \cQ' )\cong \fS_\fX$.  The short exact sequence \begin{equation*}
h \cQ'/h^m\cQ'\to
\cQ'/ h^m\cQ'\to\fS_{U}
\end{equation*} similarly shows inductively that $\mathbb{R}^1i_*(\cQ'/ h^m\cQ')=0$ for all
$m$.  An argument as in \cite[2.12]{KR}, using the Mittag-Leffler
condition, shows that thus $\mathbb{R}^1i_*\cQ'=0$.  

Consider the long exact sequence \[0\to i_*(h\cQ')\to
i_*\cQ'\to\fS_{\fX}\to \mathbb{R}^1i_*(h\cQ')\to \cdots. \] Since 
$\mathbb{R}^1i_*(h\cQ')\cong \mathbb{R}^1i_*\cQ'=0$, we get an isomorphism $i_*\cQ'/ i_*(h \cQ' )\cong \fS_\fX$, and so the $\C[[h]]$-module $i_*\cQ'$ is a quantization.
\end{proof}

 Now we turn to the case of a conical symplectic resolution $\fM$. In
 this case, the ring $A$ depends
only on the cone $\fM_0$, and not on the choice of resolution.  

More precisely, 
let $\fM$ and $\fM'$ be two conical symplectic resolutions of the same affine cone.
By Proposition \ref{different resolutions}, the groups $\Ht$ and $H^2(\fM'; \C)$ are canonically isomorphic.
Let $\cD$ and $\cD'$ be quantizations of $\fM$ and $\fM'$ with the
same period, and $\scrD$ and $\scrD'$ the corresponding quantizations
of the universal quantizations $\scrM$ and $\scrM'$.  

\begin{proposition}\label{deformed same sections}
  There is a canonical isomorphism
between the section rings $\scrA := \secs(\scrM; \scrD)$
and $\scrA' := \secs(\scrM'; \scrD')$.
\end{proposition}
\begin{proof}
We have a canonical rational map $\scrM\dasharrow\scrM'$.  This
induces an isomorphism between the fiber over a generic point in
$\Ht\cong H^2(\fM';\C)$, and gives a pair of crepant resolutions of
each fiber.  Thus, applying Proposition \ref{different
  resolutions} to  each fiber, we find that the exceptional locus of
this map is codimension 2 in each fiber.  Combining this with the fact
that the generic fiber avoids the exceptional locus, we see that it
has codimension 3.  Let $U\subset \scrM, U'\subset \scrM'$ be the
complements to the exceptional loci, so that $\scrM\dasharrow\scrM'$
induces an isomorphism $U\cong U'$.  

Let $i\colon U \to \scrM$ and $i'\colon U' \to \scrM'$ be the
inclusions of these sets.  By Lemma \ref{codimension-3}, $\scrD'' := i'_*i^*\scrD$ is a
quantization of $\scrM'$ with section ring $$\secs(\scrM'; \scrD'') \cong \secs(U; \scrD) \cong \secs(\scrM; i_*\scrD)\cong \scrA.$$
Since $(i')^*\scrD''\cong i^*\scrD$ and $(i')^*\scrD'$
have the same period (by definition), the quantizations $\scrD'$ and
$\scrD''$ must also have
the same period and thus are isomorphic.  Thus, we have that $\scrA\cong
\secs(\scrM'; \scrD'')\cong \secs(\scrM'; \scrD')\cong \scrA'.$
\end{proof}

Propositions \ref{restriction of sections} and \ref{deformed same sections} have the following corollary.

\begin{corollary}\label{same-sections}
There is a canonical isomorphism
between the section rings $A := \secs(\fM; \cD)$
and $A' := \secs(\fM'; \cD')$.
\end{corollary}

We may now use Proposition \ref{restriction of sections} to show that the ring $A$ does not change 
when the period of $\cD$ changes by an element of
the Weyl group; this unifies the isomorphisms mentioned in three of the four examples above.
For any $\la\in\Ht$, let $A_\la$ be the invariant section algebra of the quantization with period $\la$.

\begin{proposition}\label{Weyl rings}
For any $\la\in\Ht$ and $w\in W$, we have an isomorphism $A_{\la} \cong A_{w\cdot \la}$.
Furthermore, these isomorphisms may be chosen
to be compatible with multiplication in the Weyl group.
\end{proposition}

\begin{proof}
As in Section \ref{sec:deformations}, let
$\scrN := \Spec\C[\scrM]$ be the affinization of the universal deformation of $\fM$, and
let $\becircled{\scrM}\subset\scrM$
be the locus on which the map to $\scrN$ is a local isomorphism.
Since this map is a crepant resolution of singularities, it induces
an isomorphism from $\becircled{\scrM}$ to the smooth
locus of $\scrN$.
Thus, $\becircled\scrM$ inherits a $W$-action from $\scrN$ and the canonical quantization $\scrD$
of $\scrM$ restricted to $\becircled{\scrM}$ is also $W$-equivariant.
Note that $\scrA:=\secs(\scrM;\scrD)$ is isomorphic to 
$\secs(\becircled{\scrM};\scrD)$ by Lemma \ref{codimension-3}, since the codimension of the
complement of $\becircled{\scrM}$ is at least $2$.  Thus $\scrA$
carries a natural $W$-action.
The proposition now follows from Proposition \ref{restriction of sections} and the $W$-equivariance of $h^{-1}c$.
\end{proof}

\subsection{Quantum Hamiltonian reduction}\label{QHr}
Let  $\cQ$ be an $\bS$-equivariant quantization of $\fX$. 
Let $G$ be a connected reductive algebraic group over $\C$, 
and assume that $\fX$ is equipped with a $G$-action
commuting with the action of $\bS$.  
We will assume 
that the action of $G$ is Hamiltonian with 
moment map $\mu:\fX\to\mg^*$, and that $\mu$ is $\bS$-equivariant with respect to the weight $n$
scalar action on $\mg^*$.  A {\bf Hamiltonian
  \boldmath$G$-action} on the pair $(\fX, \cQ)$ consists of
\begin{enumerate}
\item an action of $G$ on $\fX$ as above
\item a $G$-equivariant structure on $\cQ$ so that the algebra map $\cQ\to
  \fS_\fX$ is equivariant
\item a $G$-equivariant filtered\footnote{We filter $U(\mg)$ so that the associated graded $\C[\mg^*]$
has $\mg$ sitting in degree $n$.}
 $\C[S]$-algebra homomorphism
  $\eta\colon U(\mg)\to \secs(\cQ[h^{-1}])\subset A$
\end{enumerate}
such that for all $x\in \mg$, 
the adjoint action of $\eta(x)$ on $\cQ$ agrees with the action of $x$ induced by the $G$-structure on $\cQ$.
The map $\eta$ is called a {\bf quantized moment map} because the associated graded
$$\gr\eta \colon \C[\mg^*] \cong \gr U(\mg)\longrightarrow\gr A \cong\C[\fX]$$ induces a
$G\times\bS$-equivariant classical moment map 
$\mu\colon\fX\to\mg^*$, where $\bS$ acts on $\mg^*$ with weight $-n$.  
We note that for any $x\in\mg$, we will have
$$\eta(x)\in \secs(h^{-1}\cQ)\subset\secs(\cD(n))= A(n) \subset A.$$
The following proposition says that
the condition of admitting a quantized moment map
is no stronger than the condition of admitting a classical moment map.
Recall that we use $\chi(\mg)$ to denote the vector space of characters $\mg\to\C$.

\begin{proposition}\label{Hamilton}
For any $\bS$-equivariant quantization $\cQ$ of $\fX$, the pair $(\fX,\cQ)$ admits a Hamiltonian $G$-action
that induces $\mu$ in the manner described above and
the set of quantized moment maps compatible with a given $G$-action on
$\cQ$ is a torsor for $\chi(\mg)\otimes \C[\fX]^{\bS\times G}$.
\end{proposition}

\begin{proof} We break this up to into proofs of the different components of this result.
First, we prove that $\cQ$ has a unique $G$-equivariant structure.  Since $G$ is connected, its action on the de Rham cohomology $ H^2_{DR}(\fX/S; \C)$ is trivial.
As noted in the proof of \cite[Prop. 6.2]{BK04a} and expanded on in more detail in \cite[\S 2.3]{Losq}, this means that the quantization for every class in $H^2_{DR}(\fX/S; \C)$ has a unique equivariant structure.  As discussed in the sources above, this can be seen by thinking about inserting $G$-equivariance at each point  in the original proof of \cite[Thm. 1.8]{BK04a}.    Since the discussions in \cite{BK04a} and \cite{Losq} are quite brief, we will add some more details below.

This proof of \cite[Thm. 1.8]{BK04a} proceeds by interpreting the
problem of lifting a given quantization modulo $h^n$ to be a
quantization modulo $h^{n+1}$ as finding an extension of the
corresponding Harish-Chandra torsor over the pair $\langle (\Aut
D)_n,(\Der D)_n\rangle $ to $\langle (\Aut D)_{n+1},(\Der
D)_{n+1}\rangle$.  The set of such extensions is a torsor for the
sheaf cohomology $H^1(\fX;\fS_{\fX}/\pi^{-1}\fS_{S})$ of the quotient
of  functions on $\fX$ by the pullback of functions on $S$.  This torsor is
identified with the de Rham cohomology $H^1_{DR}(\fX/S;\mathcal V)$ of a relative local system
$\mathcal{V}$ on $\fX/S$, and thus can be identified with the set of extensions of
relative local systems $\mathcal{V}\to \mathcal{W}\to \fS_{\fX}$.  As
usual, a $\mathcal{V}$-valued closed 1-form $\theta$ gives a natural
relative local system, whose flat sections are identified with the set of
sections of $\mathcal V$ satisfying $df=k\theta$ for $k\in \C$. Two such extensions $\mathcal{W}, \mathcal{W}'$
arising from $1$-forms $\theta$, $\theta'$ are isomorphic if and only $\theta-\theta'=dg$ for some $g$ in the
holomorphic sections of $\mathcal{V}$, with the isomorphism given by
$f\mapsto f+g$. We have a short exact sequence $0\to \pi^{-1}\fS_{S}\to
\fS_{\fX}\to \fS_{\fX}/\pi^{-1}\fS_{S}\to 0$.
The boundary map in the long exact sequence induced on sheaf
cohomology gives an isomorphism
$H^1(\fX;\fS_{\fX}/\pi^{-1}\fS_{S})\cong H^2_{DR}(\fX/S; \C)\cong H^2(\fX; \pi^{-1}\fS_{S})$.

In order to pass to the equivariant case, we only need to think about the same lifting problem equivariantly for $G$; thus, the lifts are a torsor for the groups of $G$-equivariant extensions of local systems $\mathcal{V}\to \mathcal{W}\to \fS_{\fX}$.  Of course, any such extension corresponds to a $1$-form $\theta$ whose cohomology class is $G$-invariant.  
Since $G$ is reductive, we can assume $\theta$ itself is invariant by 
projecting it to invariant 1-forms, which doesn't change its class in cohomology.  The extension of local systems discussed above has an induced $G$-equivariant structure, and using reductivity again, this structure is unique, because the projection of a not necessarily $G$-equivariant isomorphism of extensions to invariants is a $G$-equivariant isomorphism.  This completes the proof that  $\cQ$ has a unique $G$-equivariant structure.

 Now, we turn to the question of the existence 
 of a quantized moment map for this $G$-equivariant structure.  Since $h$ has $\bS$-weight $n>0$, the lack of functions on $\fX$ of negative $\bS$-weight shows that
$\secs(\fX; \cQ)$ is a commutative algebra, canonically isomorphic to
the $\bS$-invariants $\C[\fX]^\bS$.  Furthermore,
we have a natural Lie algebra structure on $\secs(\fX; h^{-1}\cQ)$
induced by the bracket, since sections of $\cQ$ commute modulo $h$. 
We have a short exact sequence of Lie algebras
\begin{equation}\label{eq:LA-sequence}
0\to\secs(\fX; \cQ)\to\secs(\fX; h^{-1}\cQ)\to\C[\fX]_n\to 0,	
\end{equation}
where the Lie bracket on $\C[\fX]_n$ is the Poisson bracket.  To show a quantized moment map exists, we wish to show that the moment map $\mu\colon \mg \to \C[\fX]_n$ lifts to a $G$-equivariant Lie algebra homomorphism $\mg\to \secs(\fX; h^{-1}\cQ)$.

Note that the subsheaf $h\cQ$ satisfies $\secs(h\cQ)=0$ and $H^1(\fX;h\cQ)=0$, so the exact sequence \eqref{eq:LA-sequence} is isomorphic to  
$$0\to\secs(\fX; \cQ/h\cQ)\cong \C[\fX]_0\to\secs(\fX; h^{-1}\cQ/h\cQ)\to\C[\fX]_n\to 0.$$ 
Thus we aim to show that there is a lift  $\mg\to \secs(\fX; h^{-1}\cQ/h\cQ)$.  

If such a lift exists, then the set of such lifts forms a torsor over $\chi(\mg)\otimes \C[\fX]^{\bS\times G}$, the group of Lie algebra homomorphisms $\mg \to \C[\fX]^{\bS\times G}$. 
Thus, to prove the result, it only remains that establish that a lift exists.  
The obstruction to the existence of a lift is a Lie algebra cohomology class $\nu\in H^1(\mg; \C[\fX]^{\bS\times G})$.  Since $\mg\cong \mathfrak{z}(\mg)\oplus [\mg,\mg]$ and the latter Lie algebra is semi-simple, Whitehead's first lemma gives
an isomorphism $H^1(\mg; \C[\fX]^{\bS\times G})\cong H^1(\mathfrak{z}(\mg); \C[\fX]^{\bS\times G})$.
So it suffices to check that a lift exists after replacing $G$ by the connected component of the identity in its center.  That is, we can assume without loss of generality that $G$ is commutative.  

First, we will prove the existence of a quantum moment map in the case where $\cQ$ is the canonical quantization.  In this case, we have an isomorphism $\eta\colon \cQ\cong \cQ^{\op}$ as quantizations, where $\cQ^{\op}$ has its $\C[[h]]$-module structure twisted by $h\mapsto -h$.  
 The isomorphism $\eta$ induces a Lie algebra anti-automorphism 
 of the exact sequence
\[0 \to\cQ/h\cQ\cong \fS_{\fX}\to h^{-1}\cQ/h\cQ\to h^{-1}\cQ/\cQ\cong \fS_{\fX}\to 0,\]
and it induces the identity on $\cQ/h\cQ$, and $-1$ on $h^{-1}\cQ/\cQ$.  Thus, considering the eigenspaces of $\eta$ splits this sequence.  This induces an isomorphism of sheaves of Lie algebras \begin{equation}\label{eq:mod-h}
	h^{-1}\cQ/h\cQ\cong (\fS_{\fX}\otimes h^{-1}\C[[h]]/h\C[[h]])
\end{equation} 
where the RHS has the Lie algebra structure
\[[f_0+h^{-1}f_{-1},g_0+h^{-1}g_{-1}]=\{f_0,g_{-1}\}+\{f_{-1},g_0\}+h^{-1}\{f_{-1},g_{-1}\}.\]    We can also deduce this result from \cite[Lem. 3.6]{BK04a}, which has effectively the same proof.  In this case, $h^{-1}\mu$ gives us the desired splitting.  

\newcommand{\can}{\operatorname{can}} Finally, we prove that a quantum moment map exists for a quantization $\cQ$ with general period $\la$.  
By standard results, we can $G$-equivariantly embed $\fX$ as a locally
closed subvariety of the projectivization of a representation of $G$.
Using the usual affine cover of projective space associated to a
basis of weight vectors, we can
cover $\fX$ with $G$-invariant open affines $U_i$, and by taking the cover fine enough, we can trivialize the period $\la$ of the quantization on each
$U_i$.  Thus we have isomorphisms
$\beta_i\colon \cQ|_{U_i}\overset{\sim}\longrightarrow \cQ^{\op}|_{U_i}$. We can assume the $\beta_i$ are $G$-equivariant, since $G$ acts on the set of ring homomorphisms $\cQ|_{U_i}\to \cQ^{\op}|_{U_i}$, and the projection to the invariants of an isomorphism which is the
identity modulo $h$ (and thus $G$-equivariant modulo $h$) remains an isomorphism by Nakayama's
lemma.  The compositions $\gamma_{ij}=\beta_j^{-1}\circ \beta_i$ give a
\v Cech 1-cocycle of automorphisms of $\cQ$.  We can view $\cQ^{\op}$ as the
regluing (as in \cite[\S 4.1]{LosHC}) of $\cQ$ with respect to this cocycle; that is, the sheaf of algebras $\cQ^{\op}$ is isomorphic to the sheaf \[\cQ^{\operatorname{glue}}(V)=\{(a_1,\dots a_n) \mid a_i\in \cQ(V\cap U_i),  \gamma_{ij}(a_j)=a_i\hspace{2mm} \forall i,j\}\] via the map $a\mapsto (\beta_1^{-1}(a),\dots \beta_n^{-1}(a))$.  

By \cite[Prop. 4.1]{LosHC}, we can assume that
$\gamma_{ij}=\exp(\operatorname{ad} c_{ij})$ for $c_{ij}\in
\cQ(U_i\cap U_j)$.   Again, we can assume that $c_{ij}$ is
$G$-invariant by projection to invariants, since $\gamma_{ij}$ is
$G$-invariant. By the cocycle equation, for any $f$ in
$\cQ(U_i\cap U_j\cap U_k)/h^2$ we have
\[f=\gamma_{ij}\gamma_{jk}\gamma_{ki}(f)=f+h
  \{\bar{c}_{ij}+\bar{c}_{jk}+\bar{c}_{ki},\bar f\}\]
where we use $\bar{g}$ to mean the reduction of a section of $\cQ$ mod $h$. 
This shows that $\alpha_{ijk}=\bar{c}_{ij}+\bar{c}_{jk}+\bar{c}_{ki}$ is constant
on $U_i\cap U_j\cap U_k$, and thus defines a class $\alpha\in
H^2(\fX;\C)$ constructed via \v Cech cohomology. 

Let us now compute the effect this has on the period of the regluing
of $\cQ$.  In the sentences below we use the notation of \cite[\S 4]{BK04a}. It's enough to calculate the reduced period map
$\mathsf{Per}_1$ since by $\bS$-equivariance, all higher terms in the
period vanish.  Poisson bracket with the  classes $\bar{c}_{ij}$ defines a \v Cech
1-cocycle of Hamiltonian vector fields, i.e.\ a class in
$H^1(\fX;\mathcal{W})$, and the regluing is exactly the torsor
structure of this group on the set of quantizations modulo $h^2$ (that
is, $H^1_{\mathcal{M}}(\fX,\langle (\mathsf{Aut}
D)_{2},(\operatorname{Der} D)_2\rangle)$.  The class $\alpha$ is the
image of $\bar{c}_{ij}$ under the boundary map, so it follows
from the compatibility of the period map with the
$H^1(\fX;\mathcal{W})$-torsor structure that the period of the
regluing is given by $\la+\alpha$.
 Since the period of $\cQ^{\op}$ is $-\la$, we have $\alpha=-2\la$.

Let $\gamma_{ij}'=\exp(\frac{1}{2}\operatorname{ad} c_{ij})$.  Consider the regluing $\cQ'$ of $\cQ$ with respect to $\gamma_{ij}'$; this has period $\la+\frac{1}{2}\alpha=0$, and so it is the canonical quantization $\cQ^{\can}$.  
Note that this defines a $G$-structure on $\cQ^{\can}$ by taking the obvious $G$-structure on the regluing.  

We have seen that $\cQ^{\can}$ admits a quantum moment map $\mu'\colon \mathfrak{g}\to \secs(\fX; h^{-1}\cQ^{\can}/h\cQ^{\can})$. Viewing $\cQ^{\can}$ as a  regluing, this means we have maps $\mu_i\colon \mathfrak{g}\to\secs(U_i; h^{-1}\cQ/h\cQ)$ that satisfy $\mu_j=\gamma_{ji}'\circ \mu_i=\exp(\frac{1}{2}\operatorname{ad} c_{ji}) \mu_i$ on the overlap $U_i\cap U_j$.  However, since $c_{ji}$ is $G$-invariant, $[c_{ji},\mu_i(X)]=0$ for all $X\in \mg$.  Thus $\exp(\frac{1}{2}\operatorname{ad} c_{ji})\mu_i=\mu_i$.    This shows that $\mu_j$ and $\mu_i$ agree on $U_i\cap U_j$, so they define a Lie algebra map $\mu\colon \mathfrak{g}\to \secs(\fX; h^{-1}\cQ/h\cQ)$, which gives the desired lifting.  This completes the proof.
\end{proof}

Assume that $(\fX,\cQ)$ carries a Hamiltonian $G$-action with quantized moment map $\eta\colon U(\mg)\to A$ and associated
classical moment map $\mu\colon\fX\to\mg^*$.  Fix a $G$-equivariant ample line
bundle $\cL$ on $\fX$, and let $\fU\subset\fX$ be the associated
semistable locus.  We will assume through the end of the section that the action of $G$ on $\fU$
is free; in particular, semistability and stability coincide.

Let $\fXred := (\mu^{-1}(0)\cap\fU)/G$ with its induced relative symplectic form 
and $\bS$-action, and let $\psi\colon\mu^{-1}(0)\cap\fU\to\fXred$ be the natural
projection.  We'll further assume that 
the natural map $\C[\mu^{-1}(0)]^G\to \C[\fX_{\red}]$ is an isomorphism.

\begin{samepage}
 Let $\cD_\fU$ and $\cQ_\fU$ denote the restrictions of $\cD$ and $\cQ$ to $\fU$, and for any $\xi\in\chi(\mg)$, let
$$\cR_\xi := \cQ_{\fU}\Big/\cQ_{\fU}\cdot \langle h\eta(x) - h\xi(x)\mid x\in\mg\rangle,$$
$$\cE_\xi(0) := \cD_{\fU}(0)\Big/\cD_{\fU}(-n)\cdot \langle \eta(x) - \xi(x)\mid x\in\mg\rangle,$$
$$\cE_\xi := \cD_{\fU}\Big/\cD_{\fU}\cdot \langle \eta(x) - \xi(x)\mid
x\in\mg\rangle.$$
\end{samepage}
These are all sheaves on $\fU$ with support $\mu^{-1}(0)\cap\fU$, which we use to define sheaves of algebras on $\fXred$ as follows: 
$$\cQred := \psi_*\sEnd_{\cQ_\fU}(\cR_\xi)^{\op},$$
$$\cDred(0) := \psi_*\sEnd_{\cD(0)_\fU}(\cE_\xi(0))^{\op},$$ 
$$\cDred := \psi_*\sEnd_{\cD_\fU}(\cE_\xi)^{\op}.$$
Kashiwara and Rouquier \cite[2.8(i)]{KR}
show that the first sheaf is an $\bS$-equivariant quantization of $\fXred$ of weight $n$, and the 
second and third are related to the first in the usual way.
Kashiwara and Rouquier work in the classical topology, but their
argument works equally well in the Zariski topology.  

\begin{remark}
Kashiwara and Rouquier
also take the fixed points of $G$. Since we have assumed that $G$ is
connected, this is redundant; the pushforward is automatically
invariant under $\mg$.  Of course, a reader interested in quotients by
disconnected groups can apply our results to the connected component
of the identity, and then consider the residual action of the
component group.
\end{remark}

We observe that this geometric operation of symplectic reduction is closely related
to an algebraic one. 
Let $Y_\xi:=A\big{/}A\cdot
    \langle \eta(x) - \xi(x)\mid x\in\mg\rangle$, where as before we let $A = \secs(\cD)$.

\begin{proposition}\label{quotient}
If $A_{\red} = \secs(\cD_{\red})$, then
  $A_{\red}\cong \End_{A}\!\left(Y_\xi\right).$
\end{proposition}

\begin{proof}
Restriction gives a natural map $A\to \secs(\cD_{\fU})$, which induces a map
\[A^G\to \secs(\fU;\sEnd_{\cD_\fU}(\cE_\xi)^{\op})\cong A_{\red}.\]
This map kills any $G$-invariant element of the left ideal generated by
$\eta(x) - \xi(x)$ for $x\in \mg$
 and thus induces a map
$Y_\xi^G\cong \End_{A}\!\left(Y_\xi\right)\to A_{\red}$.  We wish to
show that this map is an isomorphism.

By Nakayama, it's enough to check this after passing to associated
graded.  The associated graded of $A^G$ is $\C[\fX]^G$ (since $G$ is
reductive), and the map $\C[\fX]^G\to \gr(A_{\red})\subset
\C[\fX_{\red}]$ is the obvious quotient map.  The associated graded of
$Y_\xi^G$ is a quotient of $\C[\fX]^G/(\mu^*(\mg))\cong \C[\mu^{-1}(0)]^G$, so we have maps
\[\C[\mu^{-1}(0)]^G\twoheadrightarrow \gr(Y_\xi^G)\to
\gr(A_{\red})\hookrightarrow \C[\fX_{\red}].\]  The composition of these maps
is a isomorphism. Since the first map is a surjection and the last is an injection,
 each of the intermediate steps is
an isomorphism.  
\end{proof}

Next we describe the period of $\cQ_{\red}$ in terms of the parameter $\xi$;
this will prove to be an important technical tool that is needed for the proofs of 
Proposition \ref{cot-period} and Lemma \ref{minimal-polynomial}.
For simplicity, we assume that $\fX$ is symplectic
over $\Spec \C$ (rather than over an arbitrary base) and $\C[\fX]^{\bS\times G}=\C$, that 
$\cQ$ is the canonical quantization of $\fX$, and that $\fX_{\red}$
satisfies our running assumptions on $\fX$.  

The following general result about opposites
and quantum Hamiltonian reduction will be used to prove Lemma \ref{canonical},
and may also be of independent interest.

\begin{lemma}\label{lem:left-right}
  Let $A$ be an algebra with an action of a connected reductive affine algebraic
  group $G$ with noncommutative moment map $\eta\colon U(\mg)\to A$.
  Then we have natural isomorphisms
  \begin{equation}
  \End_A(A/A\eta(\mg))^{\op}\cong \End_{A^{\op}}(A/\eta(\mg)A)\cong \End_{A^{\op}}(A^{\op}/A^{\op}\eta(\mg)).\label{eq:opp-reduction}
\end{equation}
That is,
  the left and right quantum Hamiltonian reductions are opposite to
  each other.
\end{lemma}

\begin{proof}
We can freely replace $G$ with a finite cover, and thus assume that
$G$ is a product of simple groups.  Since reducing by $G_1\times G_2$
can be done in stages as reduction by $G_1$ and then by $G_2$, we can
reduce to the case where $G$ is simple.

Right (resp. left) multiplication define homomorphisms 
\[\End_A(A/A\eta(\mg))^{\op}\cong (A /A\eta(\mg))^G\longleftarrow
A^G\longrightarrow (A /\eta(\mg)A)^G\cong  \End_{A^{\op}}(A/\eta(\mg)A).\]
Since $G$ is reductive, the functor of invariants is exact and these maps are surjective, so we need only
show their kernels agree.  The kernel $K_1$ of the left map is $A^G\cap
A\eta(\mg)$ and the kernel $K_2$ of the right map is $A^G\cap \eta(\mg)A$.  If $G$
is abelian then \[A^G\cap \eta(\mg)A=\eta(\mg)A^G=A^G\eta(\mg)=A^G\cap
A\eta(\mg),\] so we can assume that $G$ is non-abelian.  

Thus, assume that $a=\sum_i y_i\eta(x_i)$ is an element of $K_1$, where $x_i$ ranges over a basis of $\mg$.  
We can replace
$y_i$ with its projection to the isotypic component of $A$
corresponding to the adjoint representation $\mg\cong \mg^*$ (since
any other simple tensored with $\mg$ has no invariants).  In this
case, invariance shows that there is an equivariant map $\pi\colon \mg\to A$
sending $\pi(x^i)=y_i$ where $x^i$ is the dual basis to $x_i$ under
the Killing form.  Thus we have $a=\sum_i \eta(x_i) y_i+x_i\cdot
y_i=\sum_i \eta(x_i) y_i+\pi([x_i,x^i])$ by the equivariance of $\pi$.  Since $\sum_i [x_i,x^i]$ is
invariant under the adjoint action, it is trivial, and we have that
$a=\sum_i \eta(x_i) y_i\in K_2$.  Applying a symmetric argument, we
see that $K_1=K_2$, so the first equality of \eqref{eq:opp-reduction}
follows immediately.  The second is just the equivalence of categories
between right $A$-modules and left $A^{\op}$-modules.
\end{proof}

A quantized moment map $\eta\colon U(\mg)\to A$ is called {\bf balanced} if, when $\xi = 0$,
$\cQ_{\red}$ is the canonical quantization of $\fX_{\red}$.

\begin{lemma}\label{canonical}
The canonical quantization of the variety $\fX$ admits a 
balanced quantized moment map.
\end{lemma}

\begin{proof}
By Proposition \ref{Hamilton}, the set of quantized moment maps is a torsor for $\chi(\mg)$.
Since $\cQ$ is the canonical quantization, we know that $\cQ\cong \cQ^{\op}$, and any choice of such an
isomorphism (that is, any algebra anti-automorphism $\phi$ of $\cQ$)
sends a quantized moment map to minus a
quantized moment map.  Thus, $-\phi$ preserves the set of quantized
moment maps, and is an anti-automorphism of $\chi(\mg)$-torsors, so
it fixes a unique point. 

Recall that 
$$\cQ_{\red} = \psi_*\sEnd_{\cQ_\fU}\left( \cQ_{\fU}\Big/\cQ_{\fU}\cdot \langle h\eta(x) - h\xi(x)\mid x\in\mg\rangle \right)^{\op}.$$
By Lemma \ref{lem:left-right}, the opposite ring of $\cQ_{\red}$ is obtained as the analogous reduction of the
opposite ring of $\cQ$:
$$\cQ_{\red}^{\op} \cong \psi_*\sEnd_{\cQ_\fU}\left( \cQ_{\fU}^{\op}\Big/\cQ_{\fU}^{\op}\cdot \langle -h\eta(x) + h\xi(x)\mid x\in\mg\rangle \right)^{\op}.$$
Twisting the action of $\cQ$ by the action of $\phi$, this sheaf is also isomorphic to
$$\psi_*\sEnd_{\cQ_\fU}\left( \cQ_{\fU}\Big/\cQ_{\fU}\cdot \langle -h\phi(\eta(x)) + h\xi(x)\mid x\in\mg\rangle \right)^{\op}.$$
Thus, if we choose $\eta$ to be the fixed point of $-\phi$ and take
$\xi=0$, the quantization $\cQ_{\red}$ is isomorphic to its own opposite, and therefore to the canonical quantization.
\end{proof} 

The following proposition is implicit in the principal results of \cite{Losq},
but does not seem to be explicitly stated in the generality that we need.
Our proof is similar to the proof of \cite[5.3.1]{Losq}.

\begin{proposition}\label{dui-heck}
If $\fX$ is canonically quantized, $\eta$ is a balanced quantized moment map, 
and $\xi\in\chi(\mg)$ is arbitrary,
then the period of $\cQ_{\red}$ is equal to $[\omega_{\red}]+h\mathsf{K}(\xi)$, where
$\mathsf{K}\colon \chi(\mg)\to H^2(\fX_{\red}; \C)$  is the Kirwan map.
\end{proposition}

\begin{proof}
Consider the inclusion $\chi(\mg)\cong (\mg^*)^G \subset \mg^*$,
and let $\fP := \big(\mathfrak{U}\cap \mu^{-1}(\chi(\mg))\big)/G$, which is equipped with a natural map $\gamma\colon\fP\to\chi(\mg)$.
Since $G$ acts freely on $\fU$, $\gamma$ is a submersion and $\mathfrak{P}$ is a
flat deformation of $\fX_{\red}=\gamma^{-1}(0)$, symplectic over the base $\chi(\mg)$.   
The quantization \[\hat{\cQ}_{\red}\cong \gamma_*\sEnd\Big( \cQ_{\fU}\Big/\cQ_{\fU}\cdot \langle h\eta(x) \mid x\in[\mg,\mg]\rangle\Big)^G\]
of  $ \mathfrak{P}$ is self opposite, and thus canonical, so its period is equal to the class of the relative symplectic form
  $\omega_{\mathfrak{P}}\in \Omega^2(\mathfrak{P}/\chi(\mg))$.  The quotient
  \[\cQ_{\red}=\hat{\cQ}_{\red}/\hat{\cQ}_{\red}\cdot \langle h\eta(x)-h\xi(x) \mid
  x\in\mg/[\mg,\mg]\rangle,\] which is supported on $\fX_{\red}$, can be thought of as the pullback of
  $\hat{\cQ}_{\red}$ by the map $s\colon \Delta\to \Delta\times
  \chi(\mg)$ which is the identity on $\Delta$ and has the property that 
  $s^*x=h\cdot \xi(x)$ for any element $x\in \mg/[\mg,\mg]\cong \chi(\mg)^*$.
By Proposition \ref{general pullback},
this quantization of $\fX_{\red}$ has period
$s^*[\omega_{\mathfrak{P}}]$. The usual
 Duistermaat-Heckman theorem implies that $s^*[\omega_{\mathfrak{P}}] =[\omega_{\red}]+h\mathsf{K}(\xi)$.
\end{proof}

\section{Modules over quantizations}
\label{sec:modules}

 Let $\cQ$ be an $\bS$-equivariant quantization of $\fX$, and consider the sheaves $\cD$ and $\cD(m)$
defined in the beginning of Section \ref{section ring}.
An $h$-adically complete module over $\cQ$ (respectively $\cD(0)$) is called 
{\bf coherent} if it is locally a quotient of a sheaf which is free of
finite rank.  By Nakayama's lemma, this is equivalent to the property that one obtains a coherent sheaf by setting
$h$ (respectively $\hon$) to zero.

\begin{remark}
   Some other sources on modules over deformation quantizations
   contain an {\em a priori} stronger notion of ``coherent'' as in defined in \cite[\S
   1.1]{KSdq}.  However, since $\fX$ (and thus $\cD$) is Noetherian,
   \cite[1.2.5]{KSdq} shows that this notion coincides with the one we
   have given above.  In general, we simplify many issues around
   finiteness by assuming that the modules we consider are coherent.  Removing this
   condition would complicate matters substantially.
\end{remark}

A $\bS$-equivariant $\cD$-module is a $\cD$-module equipped with an $\bS$-structure in the sense of Section \ref{S-structures},
compatible with the $\bS$-structure on $\cD$.  More precisely, it is a $\cD$-module $\cN$ along with an isomorphism
$a_{\text{tw}}^*\cN\cong p^*\cN$ satisfying the natural cocycle condition, such that the following diagram commutes.
\[\tikz[->,very thick]{
\matrix[row sep=10mm,column sep=20mm,ampersand replacement=\&]{
\node (a) {$a_{\text{tw}}^*\cD\otimes a_{\text{tw}}^*\cN$}; \& \node (c) {$a_{\text{tw}}^*\cN$};\\
\node (b) {$p^*\cD\otimes p^*\cN$}; \& \node (d) {$p^*\cN$};\\
};
\draw (a) -- (c) ;
\draw (a) -- (b) node[left,midway]{$\cong$};
\draw[->] (b) --(d) ;
\draw (c)--(d) node[right,midway]{$\cong$};
}\]
An $\bS$-equivariant $\cD$-module $\cN$ is called {\bf good} if it 
admits a coherent $\bS$-equivariant $\cD(0)$-lattice $\cN(0)$.  
Let $\cD\mMod$ be the category of arbitrary $\bS$-equivariant modules over $\cD$, and let $\Dmod\subset\cD\mMod$
be the full subcategory consisting of good modules.  
Note that the choice of lattice is not part of the data of an object of $\cD\mmod$.
The reason for this is that we want an abelian category, which would fail if we worked with lattices:
the quotient of a lattice by a sublattice is only a lattice after killing torsion.

Many of our important results require considering derived categories;
unfortunately, there seems to be no single choice of finiteness
condition on derived categories which will suit us once and for all.  
In order to define the cohomology of sheaves of $\cD$-modules, it is
most convenient to work in unbounded derived category
$D(\cD\mMod)$ of arbitrary $\cD$-modules (in order to use \v Cech
resolutions), but in most cases of interest to us, we can restrict to
the bounded derived category $D^b(\Dmod)$ of good $\cD$-modules.  

\begin{remark}\label{rem:bounded}
Note that if $\mathcal{C}$ is an abelian category and
  $\mathcal{C}_0$ an abelian subcategory closed under taking subobjects, we can consider both the
  derived category $D^b(\mathcal{C}_0)$ and the category
  $D^b_{\mathcal{C}_0}(\mathcal{C})$ of bounded complexes in
  $\mathcal{C}$ with cohomology in $\mathcal{C}_0$.  There is an
  obvious functor
  $D^b(\mathcal{C}_0)\to D^b_{\mathcal{C}_0}(\mathcal{C})$ which is
  sometimes an equivalence and sometimes not.  If 
  $\mathcal{C}_0$ has enough projectives which remain projective in $\mathcal{C}$, then every complex in
  $D^b_{\mathcal{C}_0}(\mathcal{C})$ can be replaced by a
  quasi-isomorphic projective resolution in $\mathcal{C}_0$, which
  shows that this functor is an equivalence.
In particular, this argument carries through when $\mathcal{C}$ is the category of all modules
over some ring, and $\mathcal{C}_0$ is the subcategory of finitely generated modules.

If $\mathcal{C}$ is the category of quasi-coherent sheaves on a
projective (over affine) scheme and $\mathcal{C}_0$ is the subcategory of coherent sheaves, then this
functor is still an equivalence, even though coherent sheaves do not
have enough projectives; this follows from considering the
corresponding modules over the projective coordinate ring.  
Similarly, we will show that $\Dmod$ admits an analogous description
(Theorem~\ref{Zloc}), which implies that $D^b(\Dmod)$ is equivalent to
$D^b_{\Dmod}(\cD\mMod)$ (Corollary \ref{cor:D-full}).
\end{remark}

\begin{remark}\label{fukaya}
If $\fX = \fM$ is a conical symplectic resolution,
there are heuristic reasons to treat $\Dmod$ as an algebraic version of the
  Fukaya category of $\fM$ twisted by the
  B-field defined by $e^{2\pi i\la}\in H^2(\fM;\C^{\times})$, where $h\la$ is the period of $\cD$.  
  The firmest justification at moment lies
  in the physical theory of A-branes, which the Fukaya category is an
  attempt to formalize.  Kapustin and Witten \cite{KW07} suggest that
  on a hyperk\"ahler manifold, there are objects in an
  enlargement of the Fukaya category which correspond not just to
  Lagrangian submanifolds, but higher dimensional coisotropic
  submanifolds.  In particular, there is an object in this category
  supported on all of $\fM$ called the {\bf canonical coisotropic
    brane}.  Following the prescription of Kapustin and Witten further
  shows that $\cD$ is isomorphic to the sheaf of endomorphisms of this object.
  Nadler and Zaslow \cite{NZ} prove a related result in which $\fM$ is replaced by
  the cotangent bundle of an arbitrary real analytic manifold.
\end{remark}

\subsection{Cotangent bundles}\label{sec:cotangent}
Let us consider the special case of quantizations of $\fX = T^*X$ for
some smooth projective variety $X$, where
$\bS$ acts by inverse scaling of the cotangent fibers\footnote{This
  variety may not satisfy the property of being projective over an
  affine variety $\fX_0$, but we will not use that assumption in this section.}. 
Quantizations of cotangent bundles have been
considered many times before in different contexts, but for the sake
of completeness, we wish to show in detail how it fits in our schema.  
We will assume that $H^1(X)=0$ and $H^2(X)\cong H^{1,1}(X)$; in particular
\[H^i(\fX;\fS_{\fX})^\bS\cong  H^i(X;\fS_X)\cong H^{i,0}(X)=0 \text{ for
$i=1,2$}\] and $H^2(X)\cong \operatorname{Pic}(X)\otimes \C$.

A {\bf Picard Lie algebroid} $\cP$ on $X$ is an extension in the abelian category of Lie algebroids of the tangent sheaf
$\cT_X$, with its tautological Lie algebroid structure, by the structure sheaf $\fS_X$, with
the trivial Lie algebroid structure.  Such an extension in the category of coherent sheaves
is classified by
$$\Ext^1(\cT_X, \fS_X) \cong H^1(X; \cT^*_X) \cong H^{1,1}(X; \C) \cong H^2(\fX; \C).$$
Since we have that 
$H^0(X; \wedge^{\!2}\,\cT^*_X) = 0,$
there is a unique Picard Lie algebroid $\cP_\la$ on $X$ for each
$\la\in H^2(\fX; \C)$.  

Let $\cU_\la$ be the universal enveloping algebra of $\cP_\la$ modulo the ideal that identifies 
the constant function $1\in\fS_X$ with the unit of the algebra.  If $\la$ is the image of the Euler
class of a line bundle $\cL$ on $X$, then $\cU_\la$ is isomorphic to the sheaf of differential operators on $\cL$. 
More generally, $\cU_\la$ is referred to as the {\bf sheaf of $\la$-twisted differential
operators} on $X$.  A coherent sheaf of $\cU_\la$-modules is called a
{\bf $\la$-twisted D-module} on $X$. The sheaf $\cU_\la$ has an order filtration, and any coherent sheaf of $\cU_\la$-modules admits
a compatible filtration.

However, $\cU_\la$ is a sheaf on $X$, and we wish to find one on
$T^*X$.  This requires the technique of {\bf microlocalization} (see,
for example, \cite{KashMC,AVV} for more detailed discussion of this technique). The associated graded of $\cU_\la$ with respect to the order filtration is isomorphic to
$\Sym_{\fS_X}\!\!\cT_X$; put differently, if \[R_\la:=\left\{ \sum u_ih^i\in \cU_\la[h] \:\Big\vert \:u_i \text{ has order $\leq i$}\right\}\] is the Rees algebra of the order filtration on $\la$, then $R_\la/hR_\la\cong \Sym_{\fS_X}\!\!\cT_X$.   Given an open subset $U\subset T^*X$, we obtain a multiplicative subset $S_U\subset \Sym_{\fS_X}\!\!\cT_X(\pi(U))$ consisting of functions 
on $\pi^{-1}(\pi(U))$ which are invertible on $U$.  

We can give a non-commutative version of this construction using an
associated multiplicative system in $R_\la(\pi(U))$.
Let \[S'_U=\big\{r\in R_\la(\pi(U))\;\big\vert\; \exists\, m \text{
  such that } r\in h^mR_\la \text{ with } \overline{h^{-m}r}\in
S_U\big\}.\] This is a multiplicative system because
$\Sym_{\fS_X}\!\!\cT_X$ is a sheaf of domains.  Furthermore, since
$[r,R_\la]\subset hR_\la$, the operation of bracket with any algebra
element is topologically nilpotent (the successive powers converge to
0 in the $h$-adic topology).  Thus, in any quotient $R_\la/h^mR_\la$,
the reduction of this set $S'_U$ satisfies the Ore condition, and we
can define the localization of $R_\la$ by $S_U'$ as the inverse limit
$\cR_\la(U):=\varprojlim(R_\la/h^mR_\la)_{S'_U}$.  This defines an
$\bS$-equivariant sheaf of rings $\cR_\la$ on
$\Spec\!\big(\Sym_{\fS_X}\!\!\cT_X\big)\cong T^*X = \fX$, which is
free over $\C[[h]]$ and satisfies $\cR/h\cR \cong \fS_{T^*X}$, and is
therefore a quantization of $\fX$.

\begin{proposition}\label{cot-period}
  The period of $\cR_\la$ is $h(\la-\nicefrac \varpi 2)$, where
  $\varpi=c_1(T^*X)\in H^2(X;\C)\cong H^2(\fX;\C)$ is the canonical
  class.
\end{proposition}

\begin{proof}
We begin by choosing line bundles $\cL_1,\ldots,\cL_k$ on $X$ and complex numbers $\zeta_1,\ldots,\zeta_k$ such that
$\la=\sum_{i=1}^k\zeta_ic_1(\cL_i)$.  Let $Y$ be the total space of $\oplus\cL_i$ and let $T:= (\cs)^k$ act on $Y$
by scaling the fibers of the individual lines.  Let $\bS$ act on $T^*Y$ via the inverse scaling action on the fibers, and
let $\tilde{\cR}$ be the  $T\times\bS$-equivariant quantization
of $T^*Y$ obtained by microlocalizing the sheaf of (untwisted) differential operators on $Y$.
The action of $T$ on $(T^*Y, \tilde{\cR})$ admits a quantized moment map
$$\varphi\colon U(\mt)\to\secs(\tilde{\cR}) \cong \Gamma(Y, D_Y)$$
given on $\mt$ by the equation 
\begin{equation}\label{moment map}
\varphi(a_1,\dots,a_k)=
\sum_{i=1}^ka_i t_i\frac{\partial}{\partial t_i},
\end{equation}
where $t_i$ is any coordinate on the fiber of $\cL_i$ (the operator $t_i
\frac{\partial}{\partial t_i}$ is independent of this choice).
If we take $\zeta := (\zeta_1,\ldots,\zeta_k)\in\C^k\cong\chi(\mt)$, then symplectic reduction of $(T^*Y,\tilde{\cR})$
at the parameter $\zeta$ yields the pair $(\fX, \cR_{\la})$, as noted by
Beilinson and Bernstein in \cite[\S 2.5]{BBJant}.

First, consider the special case where $k=1$ and $\cL_1=\omega_X^{-1}$, the anti-canonical bundle of $X$.
Then $Y$ is Calabi-Yau and $\tilde{\cR}$ is the canonical
quantization, and so we can apply Proposition \ref{dui-heck}.  In
order to do this, we must find a quantized moment map with self-opposite reduction.  By \cite[\S 2.5]{BBJant},  the
reduction by $\varphi$ at the parameter $\xi\in\C\cong\chi(\mt)$ is
isomorphic to the sheaf of differential operators on $X$ twisted by $-\xi\varpi\in H^2(X; \C)$, and this
sheaf is self-opposite when $\xi= -\nicefrac{1}{2}$.  This implies that 
$$\eta(a):=a\Big(t_1\frac{\partial}{\partial t_1}+\frac{1}{2}\Big)$$
is a canonical quantized moment map.
By Proposition \ref{dui-heck}, the reduction by $\eta$ at the parameter $\xi$ has period
equal to $-h\xi\varpi$, and is isomorphic to the sheaf of differential operators twisted by
$(-\xi + \nicefrac{1}{2})\varpi$,
confirming the result for multiples of the canonical class.

Now, assume that $\cL_1=\omega_X^{-1}$, which we can always arrange.  If
$\sigma\colon T^*Y\to T^*Y/G$ is the projection, then $\sigma_*\tilde{\cR}^T$
is an $\bS$-equivariant quantization of the relative Poisson scheme $T^*Y/G\to \mathfrak{t}^*$, and thus
has period $[\omega_{T^*Y/G}]+h\epsilon$ for some $\epsilon\in H^2(X;\C)$.  If
$s\colon \Delta\to \Delta\times \mt^*$ is the section corresponding to
$\zeta_1=-1/2$ and $\zeta_i=0$ for $i>1$, then we arrive at the conclusion that
$s^*\sigma_*\tilde{\cR}^T\cong \cR_{\nicefrac{-\varpi}2}$, which already know
has period 0. Thus, we must have
$s^*([\omega_{T^*Y/G}]+h\epsilon)=h(\nicefrac{\varpi}2+\epsilon)=0$, so
$\epsilon=-\nicefrac{\varpi}2$.

For arbitrary $\zeta_i$, we have a section $s_\zeta\colon \Delta\to
\Delta\times \mt^*$, and
\[s_\zeta^*([\omega_{T^*Y/G}]-h\nicefrac{\varpi}2)=h\Big(\sum_{i=1}^n\zeta_ic_1(\cL_i)-\nicefrac{\varpi}2\Big)=h(\la-\nicefrac{\varpi}2).\]
Thus $\cR_{\la}$ has the desired period.
\end{proof}

There is a natural $\bS$-equivariant map $p^{-1}\cU_\la\to \cR_\la[h^{-1}]$, where $p\colon\fX\to X$ is the projection
and $\bS$ acts trivially on $p^{-1}\cU_\la$.
For any $\la$-twisted D-module $\cN$ on $X$, the {\bf microlocalization} of $\cN$ is defined
to be the $\cR_{\la}[h^{-1}]$-module $\cR_{\la}[h^{-1}]\otimes_{p^{-1}\cU_\la}\cN$.
Proposition \ref{microlocalization}, which is well-known to the experts, may be regarded as a non-commutative version of the equivalence between coherent sheaves on $\fX$ and sheaves of coherent $\Sym_{\fS_X}\!\!\cT_X$-modules on $X$.
 
\begin{proposition}\label{microlocalization}
Microlocalization defines an equivalence of categories from the category of finitely generated
$\la$-twisted D-modules on $X$ to $\cR_{\la}[h^{-1}]\mmod$.
\end{proposition}

\begin{proof}
The adjoint equivalence is $\cJ\mapsto (p_*\cJ)^{\bS}$; we need only check this on the algebras themselves.  
It is clear that the microlocalization of $\cU_\la$ is $\cR_\la[h^{-1}]$.  On the other hand, 
we have a map $\cU_\la\to (p_*\cR_\la[h^{-1}])^{\bS}$ which is  injective, and whose surjectivity 
is easily verified by passing to the associated graded.
\end{proof}

\begin{remark}\label{curves}
While the cotangent bundles of smooth projective varieties provide a large
supply of conical symplectic varieties, these varieties very rarely are conical symplectic resolutions.
In general they do not have enough global functions to be resolutions of their affinizations.
For example, consider the case of a curve:
  \begin{itemize}
  \item If $X=\mathbb{P}^1$, $T^*X$ is a resolution of a singular quadric.
  \item If $X$ is elliptic, $T^*X\cong X\times \aone$, so the
    affinization of $T^*X$ is isomorphic to $\aone$.
  \item If $X$ has genus greater than 1, then $T^*X$ is a line bundle of positive
    degree, and thus has no nonconstant global functions.
  \end{itemize}
\end{remark}

\begin{example}\label{G-mod-P}
One class of projective varieties whose cotangent bundles {\em are} conical symplectic resolutions
are varieties of the form $X = G/P$, where $G$ is a reductive algebraic group and $P\subset G$ is a parabolic subgroup.
Philosophically, the reason is that $X$ has a lot of vector fields (induced by the action of $\mg$), therefore
its cotangent bundle has a lot of functions.  It is conjectured (see for example \cite[1.3]{Kal09}) that these are the {\em only}
such projective varieties.  
  
If $P$ is a Borel subgroup, then $T^*X$ is the Springer resolution of the nilpotent cone in $\mg$.
More generally,
the moment map $\mu\colon T^*X\to \mg^*\cong\mg$ is always generically finite, and its
image is the closure $\bar{O}_P=G\cdot \fp^\perp$ of the Richardson orbit $O_P$ associated with $P$. 
If $G=\operatorname{SL}(r;\C)$, or if $O_P$ is
simply-connected, then $\mu$ is generically one to one, and $T^*X$ is a symplectic resolution of $\bar{O}_P$
\cite[1.3]{Hes}.  In other cases, it is still a symplectic resolution of its affinization, but this affinization 
may be a finite cover of $\bar{O}_P$.
\end{example}

\subsection{Localization}
\label{sec:localization}
We return to considering a general $\fX/S$ satisfying the
assumptions of Section \ref{sec:quantizations}. We fix a quantization $\cD$ of $\fX$,
and we let $A := \secs(\cD)$ be its section algebra.
Let $\AMod$ be the category of arbitrary $A$-modules, and let $\Amod$ be the full subcategory
of finitely generated modules.  As in the case of $\cD$-modules, we will be interested in the unbounded derived category
$D(\AMod)$ and the bounded derived category $D^b(\Amod)$;
by Remark \ref{rem:bounded}, $D^b(\Amod)$
is equivalent to the full subcategory
of $D(\AMod)$ consisting of objects whose cohomology is both bounded
and finitely generated.

We have a functor $$\secs\colon\Dmod\to\Amod$$ given by taking $\bS$-invariant global sections.
The left adjoint functor $$\Loc\colon\Amod\to\Dmod$$ is defined by putting
$\Loc(N) := \cD\otimes_AN,$ with the $\bS$-action induced from the action on $\cD$.
To see that $\Loc(N)$ is indeed an object of $\Dmod$, let
$Q\subset N$ be a finite generating set and
define a filtration of $N$ by putting $N(m) := A(m)\cdot Q$.  
We define the {\bf Rees algebra} $R(A)$ to be the $h$-adic completion of 
$$A(0)[[\hon]] + \hon A(1)[[\hon]] + h^{\nicefrac{2}{n}}A(2)[[\hon]] + \ldots\subset A[[\hon]]$$
and the {\bf Rees module} $R(N)$ to be the $h$-adic completion of
$$N(0)[[\hon]] +\hon N(1)[[\hon]] + h^{\nicefrac{2}{n}}N(2)[[\hon]] + \ldots\subset N[[\hon]].$$
Note that $R(N)$ is a module over $R(A)\cong \Gamma(\cD(0))$,
and $\cD(0)\otimes_{R(A)}R(N)$ is a coherent lattice in $\Loc(N)$.

\begin{remark}\label{lattice-filtration}
If $N$ is an object of $\Amod$, we have shown that $\Loc(N)$ always admits a coherent lattice,
but the construction of that lattice depends on a choice of filtration of $N$.  Conversely,
any coherent lattice $\cN(0)$ for an object $\cN$ of $\Dmod$ induces a filtration of
$N:= \secs(\cN)$ by putting $N(m) := \secs\big(h^{\nicefrac{-m}{n}}\cN(0)[h^{\nicefrac{1}{n}}]\big)$.
\end{remark}

If $\secs$ and $\Loc$ are biadjoint equivalences of categories, we
will say that {\bf localization holds for \boldmath$\cD$} or that {\bf localization holds at \boldmath$\la$},
where $[\omega_\fX] + h\la$ is the period of $\cD$.  
Localization is known to hold for certain parameters in many special cases, including quantizations of the Hilbert
scheme of points in the plane \cite[4.9]{KR}, the cotangent bundle of $G/P$ \cite{BB},
resolutions of Slodowy slices  \cite[3.3.6]{Gin08} \& \cite[7.4]{DK}, and hypertoric varieties \cite[5.8]{BeKu}.
We conjecture that any conical symplectic resolution $\fM$ admits many quantizations for which localization holds.

\begin{conjecture}\label{localization}
Let $\Lambda\subset H^2(\fM; \C)$ be the set of periods of quantizations for which
localization holds.
There exists 
\begin{itemize}
\item a finite list of effective classes $x_1,\ldots,x_r\in H_2(\fM;\Z)$
\item a finite list of rational numbers $a_i\in \Q$
\end{itemize}
such that $\Lambda=H^2(\fM;\C)\smallsetminus \displaystyle\bigcup_{i=1}^r D_i$, where 
\[D_i := \big\{\la\in H^2(\fM;\C)\;\;\big{|}\;\;\langle
x_i,\la\rangle-a_i\in \Z_{\leq 0}\}.\]
\end{conjecture}

\begin{remark}
%
The classes $x_1,\ldots,x_r$ should exactly correspond to the effective curve classes in 
``generic non-affine deformations'' of $\fM$ in the sense of \cite[1.15]{BMO}.
These classes play an important role in the formula for quantum cohomology of the Springer resolution \cite[1.1]{BMO},
and conjecturally of any conical symplectic resolution.
\end{remark}

Though we cannot prove Conjecture \ref{localization}, we will establish asymptotic results both in the derived
(Theorem \ref{derived-local}) and abelian (Corollary \ref{large quantizations}) settings. 

\subsection{Derived localization}\label{sec:derived}
In this section, we continue the assumptions of Section
\ref{sec:localization}.  We next wish to consider the derived functors $\Rsecs$ and  $\LLoc$
relating the triangulated categories $D(A\mMod)$ and
$D(\cD\mMod)$.  
Note that these derived functors are well-defined
by \cite[Th. A]{SpaRes}.  First, let us establish certain homological properties
of these functors.

\begin{lemma}\label{projlim}
For any good $\bS$-equivariant module $\cN$, the module
$\mathbb{R}^k\Gamma(\fX;\cN(0))$ is finitely generated over $R(A)$ and 
the map \[\mathbb{R}^k\Gamma(\fX;\cN(0))\to \varprojlim
  \mathbb{R}^k\Gamma(\fX;\cN(0)/\cN(-nm))\] 
is an isomorphism for all $k$.
\end{lemma}
\begin{proof}
Let $$G^k(m):=\mathbb{R}^k\Gamma(\fX;\cN(-nm))\and
G^k(m|p):=\mathbb{R}^k\Gamma(\fX;\cN(-nm)/\cN(-np))$$ for $p\geq m$.
We claim that the cohomology $G^k(0)$ is a finitely
  generated $\bS$-equivariant $R(A)$-module. 
  To see this, note that the cohomology long exact sequence of $$0 \to \cN(-n) \to \cN(0)\to \cN(0)/\cN(-n) \to 0$$  
  gives an
  injective map $G^k(0)/hG^k(0)\hookrightarrow
  G^k(0|1)=\mathbb{R}^k\Gamma(\fX;\cN(0)/\cN(-n))$. The latter is the
  cohomology of a coherent sheaf, and thus finitely generated over
  $\C[\fX]$.  Let $P$ be a submodule of $G^k(0)$ generated by representatives of 
  a finite generating set of 
  $G^k(0)/hG^k(0)$, so we have $G^k(0) = P + hG^k(0)$.
  Then given any $x \in G^k(0)$, we can inductively find $p_i \in P$, $i = 0,1,2,\dots$, so that
  $x - \sum_{j=0}^N h^jp_j$ lies in $h^{N+1}G^k(0)$.  Since $R(A)$ is complete in the $h$-adic topology, we can take the limit to obtain $p \in P$ such that $x - p$ lies in $\bigcap_{i=0}^\infty h^i G^k(m)$.  But this intersection is zero, since 
  $\bigcap_{i=0}^\infty h^i\cN(-nm)=0$, and so
  $G^k(0)=P$. Thus $G^k(0)$ is finitely
  generated as desired.

Thus, $G^k(0)$ is a quotient of a finite rank free module 
$R(A)^{\oplus n}$ by a submodule $K$.  Consider the short exact
sequence of projective systems
\[0\to K/(K\cap h^m R(A)^{\oplus n}) \to (R(A)/h^m R(A))^{\oplus n}
\to G^k(0)/h^mG^k(0)\to
0. \]
Since the kernel satisfies Mittag-Leffler, we obtain an isomorphism 
\[ G^k(0)\cong R(A)^{\oplus n}/K\cong \Big(\varprojlim (R(A)/h^m
R(A))^{\oplus n}\Big)/ \Big(\varprojlim K/(K\cap h^m R(A)^{\oplus
  n})\Big)\cong \varprojlim G^k(0)/h^mG^k(0).\]
Note that the long exact sequence associated to the short exact
sequence of projective systems $h^mG^k(0)\to G^k(0) \to G^k(0)/h^mG^k(0)$
further shows that the first derived functor of $\varprojlim$ vanishes on the left hand system:
\[\varprojlim\nolimits^{1}h^mG^k(0)\cong \big(\varprojlim G^k(0)/h^mG^k(0) \big)
/ G^k(0)=0.\]
All higher derived functors vanish, since this holds for any projective
system over $\Z_{\geq 0}$ in the category of modules over a ring.  
Now, we consider the long exact sequence 
\begin{equation}
\cdots \to G^{k-1}(0|m)\to G^k(m)\to G^k(0) \to G^k(0|m) \to
G^{k+1}(m)\to \cdots\label{eq:1}.
\end{equation}
This breaks into a series of short exact
sequences 
\[0\to\operatorname{Tor}^1(\C[h]/(h^m),G^k(m))\to G^k(m)\to  G^k(0)\to
G^k(0)/h^mG^k(0)\to 0.\]
The submodule of all $h$-torsion elements in $G^k(0)$ is finitely
generated, so it is killed by $h^M$ for some $M$.  For $m>M$, the
group $\operatorname{Tor}^1(\C[h]/(h^m),G^k(m))$ stabilizes, and the
induced map in the projective system is multiplication by $h$.  This
projective system satisfies the
property that the image of
$\operatorname{Tor}^1(\C[h]/(h^{m+M}),G^k(M+m))$ in
$\operatorname{Tor}^1(\C[h]/(h^m),G^k(m))$ is trivial, so the projective
system has $\varprojlim \operatorname{Tor}^1(\C[h]/(h^m),G^k(m))\cong
\varprojlim\nolimits^{1} \operatorname{Tor}^1(\C[h]/(h^m),G^k(m))=0$ by
Mittag-Leffler again.  
The short exact sequence 
\[0\to \operatorname{Tor}^1(\C[h]/(h^m),G^k(m))\to G^k(m)\to
h^mG^k(0)\to 0\]
shows that $\varprojlim G^k(m)=\varprojlim\nolimits^{1}G^k(m)=0$ as well. 

Since $G^k(0|m)$ is the extension of two projective
systems with higher derived limits vanishing, the higher projective limits
of $G^k(0|m)$ vanish as well.  The long exact sequence \eqref{eq:1} thus remains exact
when we take the projective limit, since the higher derived functors
of all its terms vanish.  Therefore, we obtain the desired isomorphism
$G^k(0)\cong \varprojlim G^k(0|m)$.
\end{proof}

\begin{proposition}\label{prop:Rsecs-bounded}
  The functor $\Rsecs$ induces a functor
  $D^b(\cD\mmod)\to D^b(A\mmod)$.
\end{proposition}
\begin{proof}
  Since any complex of $A$-modules with cohomology that is
  finitely generated and bounded is quasi-isomorphic to a bounded
  complex, we need only prove that $\Rsecs$ applied to any good
  $\cD$-module $\cN$ has finitely generated cohomology in finitely many
  degrees.  By Lemma \ref{projlim}, the cohomology is finitely
  generated, and we need only check that
  $G^k(0|m)$ (using the notation of the lemma) is only non-zero in
  finitely many degrees.  Since
  $\cN(0)/\cN(-nm)$ is just an iterated extension of $\cN(0)/\cN(-1)$
  it suffices to show the same for $H^i(\fX;\cN(0)/\cN(-1))$.  
  Since $\fX$ is projective over $\fX_0$, this group is finitely generated over $\C[\fX]$
  and can only be non-zero if $0\leq i\leq \dim\fX$.
  \end{proof}
If the functors $\Rsecs$ and $\LLoc$ induce biadjoint equivalences
$D^b(\cD\mmod)\cong D^b(A\mmod)$, we say that 
{\bf derived localization holds for \boldmath$\cD$} or that {\bf derived localization holds at \boldmath$\la$},
where $[\omega_\fX] + h\la$ is the period of $\cD$.
The following result of Kaledin \cite[\S 3.1]{KalDEQ} gives a sufficient condition for derived localization to hold.
Let $\cDop$ be the opposite ring of $\cD$, and let $\Aop$ be its section algebra.
Consider the sheaf of algebras $\cD\hat\boxtimes_{\C((h))}\cDop$ on $\fX\times\fX$, 
which has section algebra $A\otimes\Aop$.
Let $\cD_{\operatorname{diag}}$ be the $\cD\hat\boxtimes_{\C((h))}\cDop$-module obtained by pushing $\cD$ forward along the diagonal inclusion
from $\fX$ to $\fX\times\fX$, and let $A_{\operatorname{diag}}$ be the algebra $A$, regarded as a module over $A\otimes\Aop$.

\begin{theorem}[Kaledin]\label{kaledin-DL}
Suppose that the higher cohomology groups of $\cD$ vanish.\footnote{By Proposition \ref{GR},
this condition is satisfied by any conical symplectic resolution.}
 Then derived localization holds if and
  only if the natural map $\LLoc(A_{\operatorname{diag}})\to
  \cD_{\operatorname{diag}}$ is a quasi-isomorphism.
\end{theorem}

\begin{remark}\label{rem:bounded-enough}
Kaledin uses the bounded above derived categories $D^-(\cD\mmod)$
  and $D^-(A\mmod)$; however, this equivalent to the claim that the
  equivalence holds on bounded derived categories, by the argument of \cite[1.2]{KalDEQ}.
  \end{remark}

We will use Kaledin's result to prove Theorem A from the introduction.
To do this, we first need to establish a technical result.  As usual, we let $\fM$ be a conical symplectic resolution. Let $\scrD$ be any $\bS$-equivariant
quantization of its twistor deformation $\scrM_\eta$, and let $t$ be the coordinate on $\aone$. Let  $\cN$ be a $\scrD$-module supported on a Lagrangian
 subvariety of $\fM\subset \scrM_\eta$ for $\eta\in H^2(\fM;\Z)$ (in
 the sense that its pullback to the complement of this Lagrangian is
 zero).  

\begin{lemma}\label{minimal-polynomial}
There exists a nonzero polynomial $q(x)\in \C[x]$ such that $q(h^{-1}t)\in\scr{A}_{\eta}$ acts by zero on $\cN$.
 \end{lemma}
 
\begin{proof}
Let $\scr L$ be the twistor line bundle on $\scrM_\eta$, i.e.\ the line bundle 
satisfying the statement of Proposition \ref{twistor def}, 
and let $u\colon \mathsf{Tot}(\scr L^\times)\to\scrM_\eta$ be the projection.  Then the total space
$\mathsf{Tot}(\scr L^\times)$ is symplectic, and the fiberwise $\C^*$-action is Hamiltonian
with moment map $t$, where $t$ is the coordinate on $\aone$, and the map $u$ coinduces the Poisson structure
on $\scrM_\eta$. 

Since the quantization $\scr D$ is $\bS$-equivariant, its period will be of the form
$[\omega_{\scr M_\eta}] +h\la = t\eta + h\la$ for some $\la \in H^2(\scr M_\eta; \C)$. 
Let $\scr{U}$ be the quantization of
$\mathsf{Tot}(\scr L^\times )$ with period $u^*\la$.  As noted
by Bezrukavnikov and Kaledin \cite[6.2]{BK04a}, the algebra $\scr{U}$
carries a $\C^*$-equivariant structure for the fiberwise action, commuting with the
$\bS$-equivariant structure.  
By Proposition \ref{Hamilton}, $\scr{U}$ has a quantized
moment map for the $\C^*$-action; choose one, and
let $\tau\in \secs(\scr{U}[\hmon])$ be the image of the generator $y$ of $\operatorname{Lie}(\C^*)$.
By the definition of a non-commutative
moment map, $y-\tau$ commutes with the action of any $\C^*$-invariant
section of $\scr{U}$ on any
$\C^*$-equivariant 
module over this algebra.

As noted in the proof of Proposition \ref{dui-heck}, the invariant pushforward
$(u_*\scr{U}[\hmon])^{\C^*}$ is the quantization of
$\scrM_\eta$ with period $h\la+t\eta$, and is therefore isomorphic to our given 
quantization $\scr D$.  Thus, we have an equivalence
between good $\bS$-equivariant $\scrD$-modules and $\bS\times\C^*$-equivariant
$\scr{U}[\hmon]$-modules, induced by the adjoint functors $u^*$ and
$u_*^{\C^*}$.  

Recall that we are given a $\scrD$-module $\cN$ on $\scrM_\eta$ supported on a
Lagrangian subvariety of $\fM$.  Thus, $u^*\cN$ is supported on the
preimage of $\supp(\cN)$ which is Lagrangian. 
By a finiteness theorem of Kashiwara and Schapira \cite[7.1.10]{KSdq},
the self Ext-sheaf of $(u^*\cN)^{\an}$ is perverse, and in particular, its
endomorphism algebra commuting with $\bS$ is finite dimensional over $\C$.  By
Theorem \ref{th:GAGA}, the same holds for $u^*\cN$.\footnote{One can also use
the theory of Euler classes from \cite{KSdq}, which we will discuss
later in Section \ref{sec:index}, to show that such a module has finite
length, imitating the usual proof for D-modules \cite[3.1.2(ii)]{HTT}.  We thank the referee for this observation.}
As with any element of any finite dimensional algebra over $\C$, the
endomorphism $y-\tau$ has a minimal polynomial $q(x)$
such that $q(y-\tau)=0$.  
Since the structure map
$\pi^{-1}\fS_\aone\to \scrD$ is given by $t\mapsto h\tau$, we thus have
that the action of $\tau=h^{-1}t$  on 
the reduction $\cN=(u_*u^*\cN)^{\C^*}$ satisfies the same polynomial equation.
\end{proof}

\begin{remark}
Lemma \ref{minimal-polynomial} almost
certainly holds for general $\eta\in H^2(\fM;\C)$ rather than just 
classes in the image of integral cohomology; however the proof uses the line bundle $\cL$ in a very strong way.
Proving the general case will require understanding the theory of twistor
deformations of line bundles over gerbes.
\end{remark}

\begin{theorem}\label{derived-local}
Fix a class $\eta\in H^2(\fM; \Z)$ such that $\scrM_\eta(\infty)$ is affine.\footnote{
See Proposition \ref{ample-affine} and the preceding paragraph for a discussion of $\scrM_\eta(\infty)$.}
Derived localization holds at $\la+k\eta$ for all but finitely many $k\in \C$.  
\end{theorem}

\begin{proof}
  Let $\cD_k$ be the quantization with period $h(\la+k\eta)$.  By
  Theorem \ref{kaledin-DL}, we need to show that the map
  $\LLoc((A_k)_{\operatorname{diag}})\to
  (\cD_k)_{\operatorname{diag}}$ is an isomorphism for all but
  finitely many $k$; let $\mathcal{P}_k$ denote the cone of this map.  Let $\scr{D}$ be the quantization of
  $\scrM_\eta$ with period $t\eta + h\la$, and let
  $\sigma_k\colon\Delta\to\aone\times\Delta$ be the map associated to the
  polynomial $hk$ as in Section \ref{sec:period}.  Proposition
  \ref{general pullback} tells us that $\sigma_k^*\scrD\cong \cD_k$,
  which implies that the morphism
  $\LLoc((A_k)_{\operatorname{diag}})\to
  (\cD_k)_{\operatorname{diag}}$ on $\fM\times\fM$ is the pullback of
  the morphism $\phi\colon \LLoc(\scr{A}_{\operatorname{diag}}) \to
  \scr{D}_{\operatorname{diag}}$ on
  $\scrM_\eta\times_{\aone}\scrM_{\eta}$.  It follows that
  $\mathcal{P}_k \cong\sigma_k^*\scr{P}$ where $\scr{P}$ is the cone
  of $\phi$.

We now apply Lemma \ref{minimal-polynomial} to the symplectic resolution
$\fM\times \fM$, with sheaf $\cN = \scr{P}$ and cohomology class $(\eta,\eta)$.  The associated
twistor deformation is $\scrM_\eta\times_{\aone}\scrM_{\eta}$. 
The sheaf we will apply it to is $\mathcal{H}^j(\scr{P})$.  This is supported on the preimage of
$0\in \aone$, since all fibers over non-zero points of $\aone$ are
affine varieties, where obviously the map of interest is an
isomorphism.  If we localize $R(A _{\operatorname{diag}})$ to a
sheaf on $\fM_0\times\fM_0$, the result is supported on the diagonal.  In
fact, its classical limit is the structure sheaf of the diagonal
$\Delta_{\fM_0}$. Thus, its localization is supported the preimage of
the diagonal, that is, on the Steinberg variety
$\fM\times_{\fM_0}\fM\subset \fM\times \fM$. Since
$(\cD_k)_{\operatorname{diag}}$  is also supported on diagonal
$\Delta_{\fM}$, the sheaf $\scr{P}$ is also supported on the Steinberg.  Since
any symplectic resolution is semi-small, the Steinberg variety is isotropic.
That means that either the Steinberg is Lagrangian or $\scr{P}=0$, so the hypotheses
of Lemma \ref{minimal-polynomial} are satisfied.

If $k$ is not a root of the polynomial $p$ provided by the Lemma, then $h^{-1}t-k$ acts
invertibly on $\mathcal{H}^j(\scr{P})$, so the specialization of 
this sheaf at $k$ is trivial, and so we have
$\sigma_k^*\mathcal{H}^j(\scr{P})=0$.  Thus, for any integer $m$, we
can find a polynomial (the product of those for each individual
homological degree) where $\mathcal{H}^j(\mathcal{P}_k)$ is trivial
for $j\geq -m$.

  By \cite[3.3]{KalDEQ}, $\mathcal{P}_k$ is trivial if and only if it
  it has trivial homology in degrees above $-\ell$ where $\ell$ is the global dimension of
  $\cD_k\hat\boxtimes_{\C((h))}\cD_k^{\operatorname{op}}\mmod$ (which is
  finite since the same is true for $\fS_{\fM\times \fM}$).  By the
  argument above, this happens at all $k$ other than the roots of
  a polynomial with complex coefficients, and thus for all but finitely many $k$.
\end{proof}

\section{\texorpdfstring{$\mathbb{Z}$}{Z}-algebras}
\label{sec:Z-algebra}
A  {\bf \boldmath$\mathbb{Z}$-algebra} is an algebraic structure that mimics the
homogeneous coordinate ring of a projective variety in a
noncommutative setting. 
More precisely, it is an $\N\times\N$-graded vector space
$$Z = \bigoplus_{k \geq m\geq 0} {}_kZ_m$$
with a product that satisfies the condition ${}_kZ_\ell\cdot {}_\ell Z_m\subset {}_kZ_m$ for all $k\geq \ell\geq m$
and ${}_kZ_\ell\cdot {}_{\ell'} Z_m = 0$ if $\ell\neq\ell'$.
While $Z$ itself will usually not have a unit, each algebra ${}_kZ_k$ is
required to be unital; we will also always assume that ${}_kZ_{k}$ is
Noetherian, that ${}_kZ_{m}$ is finitely generated as a left
${}_kZ_{k}$-module and as a right ${}_mZ_m$-module, and that there exists a natural number $r$ such that $Z$ is
generated as an algebra by those ${}_kZ_m$ with $k-m\leq r$.
A left module over the $\mathbb{Z}$-algebra $Z$ is an $\mathbb{N}$-graded vector space $$N = \bigoplus_{m\geq 0} {}_mN$$
with an action of $Z$ such that ${}_kZ_m\cdot {}_mN\subset {}_kN$ for all $k\geq m$
and ${}_kZ_m \cdot {}_{m'} N = 0$ if $m \ne m'$.
It is called {\bf bounded} if ${}_mN = 0$ for $m\gg 0$.  

\begin{remark}
Some authors also
discuss {\bf torsion} modules, which are those isomorphic to a direct limit of bounded
modules.  We will only be interested in finitely generated $Z$-modules, and in this setting 
the conditions of being bounded and being torsion are equivalent.
\end{remark}

We also assume that $S$ is affine.
Our goal is to define a $\mathbb{Z}$-algebra whose localized module
category is equivalent to $\Dmod$.

\subsection{Quantizations of line bundles}
In this section, we continue the assumptions given above, though
projectivity of $\fX$ and the affinity of $S$ are not needed in this subsection.  In addition to modules over quantizations, we will also need to consider bimodules over pairs of quantizations.
Let $\cQ$ and $\cQ'$ be $\bS$-equivariant quantizations.
A $\cQ'-\cQ$ bimodule is a sheaf of modules over the sheaf  $\cQ'\otimes_{\C[[h]]}\cQ^{\op}$
of algebras on $\fX$.  Such a bimodule is called {\bf coherent} if it is a quotient
of a bimodule which is locally free of finite rank.  
The most important examples will be quantizations of line bundles.

Let $\cL$ be an $\bS$-equivariant line bundle on $\fX$ and let $\eta\in H^2_{DR}(\fX/S; \C)$ 
be the image of the Euler class of $\cL$.  
Fix an $\bS$-equivariant quantization $\cQ_0$ with period
$[\omega_{\fX}]+h\la$, and for any integer $k$, let $\cQ_k$ be
the quantization with period $[\omega_{\fX}] + h(\la+k\eta)$.

\begin{proposition}\label{line bundles}
For every pair of integers $k$ and $m$,
there exists a coherent
$\bS$-equivariant $\cQ_{k}-\cQ_{m}$ bimodule $\kmi$ with an isomorphism\footnote{This is an isomorphism of $\fS_\fX-\fS_\fX$ bimodules,
where the two actions of $\fS_\fX$ on $\cL^{k-m}$ are the same.} $\kmi/(h\cdot \kmi)\cong\cL^{k-m}$.
This bimodule is unique up to canonical isomorphism,
and tensor product with $\kmi$ defines an equivalence of categories
from $\cQ_m\mmod$ to $\cQ_k\mmod$. 
\end{proposition}

\begin{proof}
By the usual sheaf theory, the locally free $\bS$-equivariant modules of rank $1$ over
$\cQ_k$ are in bijection with $H^1_\bS(\fX;\cQ_k^\times)$.
We have a surjective map  of
sheaves of groups
$\cQ_k^\times\longrightarrow \fS_{\fX}^\times$.  The kernel of this
map is $1+h\cQ_k$. As a sheaf of groups, this possesses a filtration
by the subgroups $1+h^n\cQ_k$, with successive quotients isomorphic to the
structure sheaf $\fS_\fX$ considered as a sheaf of abelian groups,
since \[(1+h^na)(1+h^nb)\equiv 1+h^n(a+b)\pmod{h^{n+1}}.\]
Since $\fS_\fX$ has vanishing higher cohomology, an argument as in
\cite[2.12]{KR} shows the inverse limit $1+h\cQ_k$ has vanishing
higher cohomology as well. By the Hochschild-Serre spectral sequence, its higher equivariant
cohomology also vanishes.
In particular, we have
an induced isomorphism \[H^1_\bS(\fX;\cQ_k^\times)\cong
H^1_\bS(\fX;\fS_{\fX}^\times).\] 
The line bundle $\cL^{k-m}$ is classified by an element $[\cL^{k-m}
]$ of $H^1_\bS(\fX;\fS_{\fX}^\times)$, and we define $\kmi$ to be the
locally free rank $1$ left $\cQ_k$-module given by the corresponding
element $[\kmi ]$ of $H^1_\bS(\fX;\cQ_k^\times)$.  The structure maps of
$\kmi/(h\cdot \kmi)$ as a $\fS_\fX$-module are just the reduction mod
$h$ of the structure maps of $\kmi$, which tells us that $\kmi/(h\cdot
\kmi)\cong \cL^{k-m}$.

Now consider the sheaf of $\pi^{-1}\fS_S[[h]]$-algebras $\cQ'=\End_{\cQ_k}\!\!\big(\kmi\big)^{\op}$.
This sheaf is an $\bS$-equivariant quantization of $\fX$, and it is obtained from $\cQ_k$
by twisting the transition functions by the 1-cocycle representing $\kmi$.
We want to show that this quantization is isomorphic to $\cQ_m$.
In order to show this, it suffices to calculate the period of
$\cQ'$ and see that it agrees with that of $\cQ_m$.  

If we can show this in the case where $S$ is a point,
then it will imply that these periods agree after pullback to every
single point in $S$.  Any two sections of $H^2_{DR}(\fX/S; \C)$
that agree after pullback to every point in $S$ are the same.  
Thus we can assume that $S=\Spec\C$.  Since $\cQ'$ is $\bS$-equivariant, its period must be of the
form $[\omega_{\fX}] + h\la'$ by Proposition \ref{S-structure}.  By
definition, the period is the obstruction to lifting the torsor
corresponding to $\cQ'$ to $\mathsf{G}$ in the notation of
Bezrukavnikov and Kaledin \cite[(3.2)]{BK04a}.  
 The
class $\la'$ is determined by the reduction $\cQ'/h^3\cQ'$, since the
obstruction to lifting this to $\mathsf{G}_3$ is $[\omega_{\fX}] +
h\la' \in H^2_{DR}(\fX/S; \C)\otimes \C[h]/(h^3).$

As shown in the proof of \cite[1.8]{BK04a}, the set of
quantizations of a given symplectic structure up to second order is a torsor over
$H^1_{DR}(\fX;\EuScript{H})$ where $\EuScript{H}$ is, in the language of \cite{BK04a},
the localization
$\mathsf{Loc}(\mathcal{M}_s,\mathsf{H})$ 
of the module $\mathsf{H}$ of Hamiltonian vector fields on the formal
disk for the Harish-Chandra torsor.  

It is helpful to think about the classical rather than Zariski
topology in order to understand this action.  As we discuss in Section
\ref{GAGA}, associated to $\cQ$, there is a quantization of the
structure sheaf of the complex manifold $\fM^{\an}$, which we denote
$\cQ^{\an}$, and analytic versions of all the sheaves we have considered.
Since the higher
pushforwards $R^n\pi_*\fS_{\fX}$ or $R^n\pi_*\fS^{\an}_{\fX^{\an}}$ vanish, we have an isomorphisms of groups
\[H^1_{DR}(\fX;\EuScript{H})\cong H^2_{DR}(\fX; \C)\qquad
H^1_{DR}(\fX;\EuScript{H}^{\an})\cong H^2_{DR}(\fX^{\an}; \C)\] via the
boundary map $\delta$ for the short exact sequence of sheaves \[0\to
\fS_{\fX}\longrightarrow
J_\infty\fS_{\fX}\overset{H}\longrightarrow\EuScript{H}\to 0,\] (or
its counterpart in the classical topology).  By a classical result of
Grothendieck, algebraic and analytic de Rham cohomology of the
structure sheaf agree, so the same holds for
$H^1_{DR}(\fX;\EuScript{H})\cong
H^1_{DR}(\fX^{\an};\EuScript{H}^{\an})$.

The classical topology has the advantage that the de Rham cohomology of
$\fS_{\fX}^{\an}$ and $\EuScript{H}^{\an}$ agree with the usual sheaf
cohomology of their flat sections, which are locally constant
functions and Hamiltonian vector fields $\mathcal{H}^{\an}$
respectively; thus we can think of an element of $H^1_{DR}(\fX
^{\an};\EuScript{H}^{\an})\cong H^1(\mathcal{H}^{\an})$ as a 1-cocycle
in Hamiltonian vector fields.  In the torsor action, a 1-cocycle acts
on a first order quantization $\cQ^{\an}/h^3\cQ^{\an}$ by twisting it via the
action of $\EuScript{H}^{\an}$ on $\cQ ^{\an}/h^3\cQ ^{\an}$ by
$X\cdot a= a+h^2X(\bar a)$ where $X(\bar a)$ denotes the usual action of
a vector field on the function $\bar a$, which is the image of $a$ in
$\cQ ^{\an}/h\cQ^{\an} \cong \fS_{\fX}^{\an}$.  Note that this does
not change the underlying Poisson bracket.
The period mod $h^2$ changes by the image under the boundary map
$h\delta$. Note that the period map is normalized so that the $n$th
order describes the $(n+1)$st order of the quantization; for example,
the 0th order part, the symplectic form, describes the 1st order
part of the quantization.

Now, we have a map of abelian groups $\b\colon (\cQ_k^{\an})^\times\to
\EuScript{H}^{\an}$ uniquely determined by $a^{-1}qa=q+h^2\b(a)(\bar q)$,
which thus matches actions on  $\cQ_k^{\an}/h^3\cQ_k^{\an}$.  Thus, when we twist
by a 1-cocycle in $(\cQ_k^{\an})^\times$, this is the same as twisting by its
image under $\b$.
That is, the period mod $h^2$ of $\End_{\cQ_k^{\an}}\!\!\big(\kmi^{\an}\big)^{\op}$
is $[\omega_{\fX}] + h(\la+k\eta+\delta\circ \b_*([\cL^{k-m}]))$ where
$\b_*$ is the induced map $H^1(\cQ_k^\times) \to
H^1((\cQ_k^{\an})^\times)\to H^1(\mathcal{H}^{\an})\cong H^1_{DR}(\EuScript{H})$
induced by the map $\b\colon (\cQ_k^{\an})^\times\to \mathcal{H}^{\an}$.

Now, we calculate that \[aqa^{-1}= q-a^{-1}[a,q]\equiv q-h^2\{\log
\bar a,q\}\pmod {h^3}\] so $\beta(a)=-H(\log \bar a)=-H(\bar a)/a$.
Thus, we wish to understand the map induced on the composition
$\delta\circ \beta$ in first cohomology.  
 Consider the
diagram of sheaves in the analytic topology (we leave off superscripts
to avoid clutter)
with short exact rows, 
along with the relevant piece of the associated long exact sequences:
\[\tikz[ thick,->]{\matrix[row sep=11mm,column sep=16mm,ampersand
    replacement=\&]{
\& \& \node (a) {$\cQ^\times_k$}; \\
\node (b) {$\Z_{\fX}$};\& \node (d) {$\fS_{\fX}$};\&  \node (e) {$\fS_{\fX}^\times$}; \\
\node (f) {$\pi^{-1}\fS_{S}$};\& \node (g) {$\fS_{\fX}$};\&  \node (h) {$\EuScript{H}$}; \\
};
\draw (a)--(e);
\draw (b)--(d);
\draw (d)--(e)  node[ above,midway]{$\exp$};
\draw (f)--(g);
\draw (g)--(h)  node[ below,midway]{$H$};
\draw (b)--(f) node[ left,midway]{$-1$};
\draw (d)--(g) node[ left,midway]{$-1$};
\draw (e)--(h)  node[right,midway]{$-H\circ \log$};
}\hspace{.3in}\tikz[ thick,->]{\matrix[row sep=11mm,column sep=16mm,ampersand
    replacement=\&]{
\node (a) {$H^1(\fX;\cQ^\times_k)$}; \&\\
 \node (e) {$H^1(\fX;\fS_{\fX}^\times)$}; \&\node (b) {$H^2(\fX;\Z)$};\\
 \node (h) {$H^1(\fX;\EuScript{H})$}; \& \node (f) {$H^2(\fX;\pi^{-1}\fS_{S})$};\\
};
\draw (a)--(e);
\draw (e)--(b) node[above,midway]{$c_1$};
\draw (h)--(f) node[below,midway]{$\delta$};
\draw (b)--(f) node [right,midway]{$-1$};
\draw (e)--(h) node[left,midway]{$-H\circ \log$};
}\]

This shows that \[\delta\circ
\b([\cL^{k-m}])=-c_1(\cL^{k-m})=(m-k)\eta,\] so by our previous calculation
$\End_{\cQ_k}\!\!\big(\kmi\big)^{\op}$ and $\cQ_m$ have identical
periods and thus are isomorphic as
$\pi^{-1}\fS_{S}[[h]]$-algebras.  

Also, we wish to show that on
$\kmi/(h\cdot \kmi)$, the quotient $\cQ_m/h\cQ_m\cong
\fS_\fX$ acts by the usual module structure on $\cL^{k-m}$.  This is a
local question, so we may assume that the line
bundle $\cL$ is trivial, in which case, $\cQ_k\cong \kmi\cong
\cQ_m$ with the left and right actions just being left and right
multiplication, which both coincide with the usual $\fS_\fX$-action
after killing $h$.

Of course, $\kmi\otimes_{\cQ_m}{}_m\cT_k$ is a quantization of
$\cL^{k-m}\otimes_{\fS_\fM}\cL^{m-k}\cong \fS_\fM$, so by uniqueness,
$\kmi\otimes_{\cQ_m}{}_m\cT_k\cong \cQ_k$, and tensor product is
indeed an equivalence.
\end{proof}

\begin{remark}
In the next proposition and later in Section \ref{sec:bimodules} we will want to vary the periods of the quantizations in more than one-dimensional families, so we will use an alternate notation and label the quantizations and bimodules by elements of $H^2_{DR}(\fX/S; \C)$ instead of integers.
In other words, the quantization $\cQ_k$ will be written $\cQ_{\la+k\eta}$ and the bimodule denoted ${}_1\cT_0$ in the notation of Proposition \ref{line bundles} will be written ${}_{\la+\eta}\cT_{\la}$. 
\end{remark}

We conclude this section by studying quantizations of line bundles in
the context of Hamiltonian reduction.  Let $(\fX, \cQ)$ be a
quantization with a Hamiltonian action of a complex algebraic group $G$.
For any $\xi\in\chi(G)$, let $\cL_\xi$ be the line bundle on $\fX_{\red}$ descending from
the trivial bundle on $\fX$ with $G$-structure given by $\xi$.
Fix a quantized moment map $\eta$ for the action of $G$ on and a
pair of elements $\xi,\xi'\in\chi(\mg)$, and let $\cQ_{\red} = \cQ_{\mathsf{K(\xi)}}$ and
$\cQ_{\red}' = \cQ_{\mathsf{K(\xi')}}$ be the corresponding reductions.  Consider the
$\cQ_{\red}'-\cQ_{\red}$ bimodule 
\[{_{\xi'}\mathcal{S}_{\xi}}:=\psi_*(\sHom_{\cQ_{\mathfrak{U}}}(\mathcal{R}_{\xi'},\mathcal{R}_{\xi})).\]

  \begin{proposition}\label{S equals T} If $\xi'-\xi$ does not integrate to a character
    of $G$, then  $\displaystyle {_{\xi'}\mathcal{S}_{\xi}}$ is
    trivial.  If it does, then it is isomorphic to the quantization
    ${_{\mathsf{K}(\xi')}\cT_{\mathsf{K}(\xi)}}$ of $\cL_{\xi'-\xi}$.
  \end{proposition}
  
  \begin{proof}
    First, note that the sheaf $\mathcal{R}_{\xi}$ inherits a left
    $\mg$-module structure via the action of left multiplication by
    $\eta(x)-\xi(x)$; furthermore $\mathcal{R}_{\xi}/h
    \mathcal{R}_{\xi}\cong \fS_{\fU\cap\mu^{-1}(\chi(\mg))}$, with the induced
    $\mg$-action coinciding with the natural one on
    $\fS_{\fU\cap\mu^{-1}(\chi(\mg))}$.  In particular, it integrates
    to the group $G$.

The sheaf $\sEnd(\mathcal{R}_{\xi})^{\op}$
    is naturally isomorphic to the $\mg$-invariant subsheaf of $\mathcal{R}_{\xi}$ via the
    map that takes an endomorphism over any open set to the image of 
    $\bar 1\in \mathcal{R}_{\xi}(\fU)$ .  Similarly, a map
    $\mathcal{R}_{\xi'}\to \mathcal{R}_{\xi}$ must take $\bar 1\in
    \mathcal{R}_{\xi'}(\fU)$ to a section $r$ killed by
    $\eta(x)-\xi'(x)$, that is, one on which the $\mg$-action is of
    the form
    \[x\cdot r=(\eta(x)-\xi(x))r=(\xi'(x)-\xi(x))r.\] 
Since this action must integrate to an action of the group $G$, there
can be no such maps if $\xi'-\xi$ does not integrate.  If it does,
then 
    the pushforward ${_{\xi'}\mathcal{S}_{\xi}}$ is a
    quantization of the line bundle $\cL_{\xi'-\xi}$ and thus isomorphic to
    ${_{\mathsf{K}(\xi')}\cT_{\mathsf{K}(\xi)}}$.
\end{proof}

\subsection{The quantum homogeneous coordinate ring of $\fX$}
Fix an $\bS$-equivariant quantization $\cQ$ of
$\fX$ with period $[\omega_{\fX}]+h\la\in H^2_{DR}(\fX/S; \C)[[h]]$ and an $\bS$-equivariant
line bundle $\cL$ on $\fX$ that is very ample relative to the affinization of $\fX$.
To these data we will associate a $\Z$-algebra $Z = Z(\fX, \cQ, \cL)$.
Let $\eta\in H^2_{DR}(\fX/S; \C)$ be the Euler class of $\cL$, let
$\cQ_k$ be the quantization with period $[\omega_{\fX}]+h(\la + k\eta)$,
let $\cD := \cQ[\hmon]$ and $\cD_k := \cQ_k[\hmon]$,
and let $\kmi$ be the $\cQ_{k}-\cQ_{m}$ bimodule that quantizes the line
bundle $\cL^{k-m}$.   

\begin{definition}\label{Tprime}
Let ${_{k}\cT_{m}'}:={_{k}\cT_{m}}[\hmon]$ be the $\cD_{k}-\cD_{m}$ bimodule
associated to the $\cQ_{k}-\cQ_{m}$ bimodule ${_{k}\cT_{m}}$.
\end{definition}

\begin{definition}\label{qhcr} 
Let ${}_kZ_m := \secs({_{k}\cT_{m}'})$ with 
products induced by the canonical isomorphisms
${}_k\cT'_\ell \otimes_{\cD_\ell} {}_\ell\cT'_m \cong {}_k\cT'_m$.  
We call $Z$ {\bf the quantum homogeneous coordinate ring of $\fX$.}
\end{definition}

We filter the sheaf ${}_k\cT'_m$ by setting ${}_k\cT'_m(0) = {_{k}\cT_{m}}[\hon]$
and ${}_k\cT'_m(\ell) = h^{\nicefrac{\ell}{n}}{}_k\cT'_m(0)$, and
give ${}_kZ_m$ the induced filtration; it is compatible with the 
multiplication, so it makes $Z$ into a filtered $\Z$-algebra.  

Note that the associated graded of ${}_kZ_m$ is isomorphic to
$\Gamma(\fX; \cL^{k-m})$, and for any $Z$-module $N$ with a compatible filtration, 
the associated graded of $N$ is a module over the  $\Z$-algebra
$$\bigoplus_{k\geq m \geq 0} \Gamma(\fX; \cL^{k-m}).$$
We will use without comment the obvious equivalence between modules over this $\Z$-algebra and graded modules over the section ring $R(\cL):=\bigoplus_{k\ge 0}\Gamma(\fX; \cL^{k})$.
A filtration of $N$ is called \textbf{good} if its associated graded is
a finitely generated module over $R(\cL)$.  Then $N$ has a good filtration 
if and only if it is finitely generated over $Z$.

Let $\Zmod$ be the category of finitely
generated modules over $Z$, and let $\Zmodbd$ be the full subcategory
of $\Zmod$ consisting of bounded modules. 
We define the functors $$\secs^\Z\colon\Dmod\to\Zmod\and\Loc^\Z\colon\Zmod\to\Dmod$$
by putting
$$\secs^\Z(\cN) := \bigoplus_{k\geq 0}\secs\big({_{k}\cT_{0}'}\otimes_{\cD}\cN\big)
\and
\Loc^\Z(N) := \left(\bigoplus_{k\geq 0}{}_0\cT_k'\right)\otimes_Z N.$$

\begin{lemma}
  If $N$ is finitely generated over $\Za$, then $\Loc^\Z(N)$ is
  finitely generated over $\cD$.
\end{lemma}
\begin{proof}
  There is some integer $K$ such that $\bigoplus_{k=0}^KN_k$ generates
  $N$.  Thus
  $\Loc^\Z(N)$ is a quotient of $\bigoplus_{k=0}^K {}_0\cT_k'
  \otimes_{{}_k\Za_{k}}N_k$.  Since the latter module is clearly
  finitely generated, the former is as well.
\end{proof}

A coherent lattice $\cN(0)$ in $\cN$ induces a filtration on $\secs^\Z(\cN)$, which is 
good because we have an injection
\[\gr\secs^\Z(\cN)\hookrightarrow \bigoplus_{m\geq 0}\Gamma(\fX;\,\,\bcN\otimes\cL^{\otimes m}),\]
where we put $\bcN := \cN(0)/\cN(-1)$.  The cokernel of this map is
bounded, since if $m\gg 0$, then $H^1(\fX;\bcN\otimes\cL^{\otimes
  m})=0$, and consequently, $\gr
\secs\big({_{m}\cT_{0}'}\otimes_{\cD}\cN\big)\cong \Gamma(\fX;\bcN\otimes\cL^{\otimes
  m})$.   This shows, in particular, that
\begin{equation}
\overline{\Loc}(\gr\secs^\Z(\cN))\cong \bcN,\label{eq:LG1}
\end{equation}
where $\overline{\Loc}$ is the usual functor sending a graded module over $R(\cL)$
to a coherent sheaf on $\fX$ by the 
localization theorem for sheaves on a
projective (over affine) variety.  
Conversely, 
a good filtration on a $Z$-module $N$ induces a lattice in $\Loc^\Z(N)$, which 
is coherent because we have 
\begin{equation}
\overline{\Loc}\big(\gr N\big)\cong \overline{\Loc^\Z(N)}.\label{eq:LG2}
\end{equation}
The functor $\Loc^\Z$ is left-adjoint to $\secs^\Z$;
let $\iota_N\colon N\to \secs^\Z\!\left(\Loc^\Z(N)\right)$ and 
$\epsilon_\cN\colon\Loc^\Z\!\left(\secs^\Z(\cN)\right)\to \cN$
be the unit and co-unit of the adjunction.  The following theorem justifies our name for $Z$.

\begin{theorem}\label{Zloc}
The co-unit $\e_\cN$ is always an isomorphism and the unit $\iota_N$ is an isomorphism
in sufficiently high degree.  Furthermore, $\Loc^\Z$ kills all bounded modules, 
thus $\secs^\Z$ and $\Loc^\Z$ are biadjoint equivalences
between $\Dmod$ and the quotient of $\Zmod$ by $\Zmodbd$.
\end{theorem}

\begin{remark}We note that this theorem is quite close in flavor to several
others in the theory of $\Z$-algebras, such as
\cite[11.1.1]{SvdB}, but these typically assume finiteness hypotheses that are too
strong for our situation.
\end{remark}

\begin{remark}
If we dropped the assumption that $S$ is affine, we would expect to be able to prove a Theorem
similar to Theorem \ref{Zloc} in which the $\Z$-algebra is replaced by a sheaf of $\Z$-algebras
over $S$.
\end{remark}

\begin{proofZloc}
Combining Equations \eqref{eq:LG1} and \eqref{eq:LG2}, we have that 
the induced map \[\overline{\epsilon_\cN}\colon
\overline{\Loc^\Z(\secs^\Z(\cN))}\to
\overline{\Loc}(\gr\secs^\Z(\cN))\cong \bcN\] 
is an isomorphism.  By Nakayama's lemma, $\epsilon_\cN$ is an isomorphism as well.  
Similarly, the map 
$$\gr(\iota_N)\colon  \bcN\to\overline{\Loc^\Z(\secs^\Z(\cN))}$$ is an isomorphism in high degree,
thus the same is true for $\iota_N$.
If $N$ is bounded, then $\overline{\Loc^\Z(N)}\cong \overline{\Loc}\big(\gr N\big)$ is the zero sheaf,
thus $\Loc^\Z(N) = 0$, as well.
\end{proofZloc}

\begin{corollary}\label{cor:D-full}
  The functor $D^b(\cD\mmod)\to D^b_{\cD\mmod}(\cD\mMod)$ is fully faithful.
\end{corollary}

\begin{proof}
Let $\cN$ be a good $\cD$-module, and let
  $N:=\secs^\Z(\cN)$.  Since $N$ is finitely generated, there is some
  $m$ such that the evaluation map
  ${}_p\Za_m\otimes_\C {_mN}\to {_pN}$ is surjective for all $p\geq m$.
  Localizing, this shows we have a surjective map
  ${}_0\cT_m\otimes_\C {_mN}\to \cN$.  Taking a classical limit 
  (possibly after increasing $m$), we
  obtain a surjection $\cL^{-m}\otimes_\C  {_mN}\to \cN /h \cN $; thus
  we have described the quantization of the familiar construction of
  such a map in algebraic geometry.  Applying this inductively, we can resolve
  $\cN$ as a complex of locally free sheaves over $\cD$, each step
  given by sums of ${}_0\cT_{m_i}$ with $m_0<m_1<m_2<\cdots$.  

By 
  taking $m_0$ sufficiently large, we can assure that for any fixed
  good $\cM$, we have
  $H^i(\fM;\cL^{m_j}\otimes_{\fS_\fM}\cM/h\cM)=0$ for all $i> 0,j\geq 0$.  Thus, we also have $\Ext^{i}_{\cD}({}_0\cT_{m_j},\cM)=0$ for all
  $i> 0,j\geq 0$.  It follows that we can use this resolution to compute $\Ext(\cN,\cM)$ in
  either $ D^b(\cD\mmod)$ or $D^b_{\cD\mmod}(\cD\mMod)$ and we see
that the results are canonically isomorphic.
\end{proof}
\subsection{\texorpdfstring{$\mathbb{Z}$}{Z}-algebras and abelian localization}\label{abelian loc}

First, we discuss some basic results that hold whenever $\fX/S$
satisfies our running assumptions for this section.
We call a bimodule between two rings {\bf Morita} if it induces a Morita equivalence between the two rings.
We call a $\Z$-algebra Z {\bf Morita} if for all $k \ge m \ge 0$ the ${}_kZ_k-{}_mZ_m$-bimodule ${}_k Z_m$ is
Morita and the natural map
\begin{equation}\label{Morita}
{}_k \Za_{k-1}\otimes {_{k-1}Z_{k-2}}\otimes
\cdots \otimes {}_{m+1} \Za_m\to{_kZ_m}
\end{equation}
is an isomorphism. In the terminology
of \cite[\S 5.4]{GS}, this means that $\Za$ is isomorphic to the Morita $\Z$-algebra
attached to the bimodules $ {}_{m+1} \Za_m$.

\begin{definition}\label{def:Z-shift}
  For any natural number $p$, let $Z[p]$ be the $\mathbb{Z}$-algebra
  defined by putting ${_kZ[p]_m}:= {_{k+p}Z_{m+p}}$.
  For any $Z$-module $N$, we define a
  $Z[p]$-module $N[p]$ by $N[p]_k = N_{p+k}$.
\end{definition}

 It is clear that $Z[p]$ is isomorphic to the $\mathbb{Z}$-algebra
  $Z(\fX, \cQ_p, \cL)$. 

\begin{proposition}\label{Morita-iff-loc}
The $\Z$-algebra $Z$ constructed in Section \ref{sec:Z-algebra} is Morita if and only if, 
for all $k\geq 0$, localization holds for $\cD_k$.
\end{proposition}

\begin{proof}
Consider the functor $\gamma(M)=
\bigoplus_k\,{_kZ_0}\otimes _{A}M$ from finitely
generated modules over $A = {_0Z_0}$ to $\Zmod/\Zmodbd$.   Let $\b$ denote the adjoint to this functor;
one description of $\b$ is that $\b(\{{}_jN\})={}_0Z_j\otimes_{A_j}
{}_jN$ for $j\gg 0$.  There is a natural transformation $\b(\secs^\Z(\cM))\to
\secs(\cM)$, induced by the natural transformation $\Loc(M)\to
\Loc^\Z(\{{_kZ_0}\otimes _{A}M\})$.  The latter natural transformation
has inverse given by the multiplication map of sections
${}_0\cT_k'\otimes_{A_k}{_kZ_0}\to \cD_0$, tensored with $M$ over
$A$.  Thus the former natural transformation is an isomorphism as
well. 

In particular, if we assume that $Z$ is Morita then Gordon and
Stafford \cite[\S 5.5]{GS} show that $\gamma$ and $\b$ are
equivalences.  Thus combining this result with Theorem \ref{Zloc}, we
see that $\secs=\b\circ \secs^\Z$ is the composition of two
equivalences, and thus an equivalence itself and localization holds for
$\cD$.  Furthermore, if $Z$ is Morita, then $Z[k]$ is Morita for all
$k\geq 0$, so localization holds for $\cD_k$ for all $k\geq 0$.

Conversely, suppose that localization holds for $\cD_k$ for all $k\geq 0$.
We have a natural isomorphism of functors
\[{}_{k+1}Z_k\otimes -\cong
\secs({}_{k+1}\cT'_k\otimes\Loc(-))\] from $A_{k}\mmod$ to $A_{k+1}\mmod$.  
Since the right hand side is an
equivalence, so is the left hand side; this proves that the bimodule ${_{k+1}Z_k}$ is Morita for all $k\geq 0$.
Similarly, this implies that \[{}_{k+1}Z_k\otimes{_kZ_m}\cong \secs\Big( {}_{k+1} \cT'_{k}
\otimes \Loc({_kZ_m}) \Big)\cong \secs({}_{k+1}\cT'_m)\cong
{_{k+1}Z_m}. \] By induction, this implies that the map \eqref{Morita} is an isomorphism.  Thus, $\Za$ is Morita.
\end{proof}

For the remainder of the subsection, we consider the case of a
conical symplectic resolution $\fM$.  As in Proposition \ref{twistor def}, let 
$\pi\colon\scrM_\eta\to\aone$ be the twistor family of $\fM$ with
$\scrN_\eta$ the affinization of $\scrM_\eta$, and let $\scr{L}$
be the line bundle on $\scrM_\eta$ extending $\cL$.
Let $\scrQ_k$ be the $\bS$-equivariant quantization of $\scrM_\eta$ with period $[\omega_{\scrM_\eta}] + h(\la + k\eta)$.

Lemma \ref{minimal-polynomial} has an algebraic counterpart.  Assume $N$ is a $\mathscr{A}$-module such that:
  \begin{enumerate}
  \item $N\cong H^i(\scrM_\eta;\cN)^{\bS}$ for any sheaf $\cN$ satisfying
    the hypotheses of Lemma \ref{minimal-polynomial}.
\item The preimage $\pi^{-1}(S)$ of the support $S$ of the coherent
  sheaf $\gr N$ on $\mathscr{N}_\eta$ is contained in a Lagrangian
  subvariety of $\fM$.  
  \end{enumerate}
\begin{lemma}\label{algebraic-minimal}
There exists a nonzero polynomial $q(x)\in \C[x]$ such that $q(h^{-1}t)\in\scr{A}$ acts by zero on $N$.
\end{lemma}
\begin{proof}
  If $N\cong H^i(\scrM_\eta;\cN)^{\bS}$, then the minimal polynomial
  $q$ of $\cN$ provides the desired polynomial.  If hypothesis (2)
  holds, then $\Loc(N)$ is supported on $\pi^{-1}(S)$, so Lemma
  \ref{minimal-polynomial} applies to $\Loc(N)$.  Since the map
  $N\hookrightarrow \secs(\Loc(N))$ is injective, the polynomial $q$ such
  that $q(h^{-1}t)$ kills $\Loc(N)$ applies equally to $N$.  
\end{proof}

One particularly important application is to the product
$\fM\times\fM$, and its twistor deformation $\scrM_\eta\times_{\aone}
\scrM_{\eta}$.  The completed outer tensor product
$\scrQ_k\hat\boxtimes_{\fS[\aone]} \scrQ^{\op}_\ell$ is a quantization of
this product, with $\bS$-invariant section algebra $
\scrA_k\otimes_{\C[h^{-1}t]} \scrA^{\op}_\ell$.  Modules over this section
algebra are just $\scrA_k\operatorname{-}\scrA_\ell$-bimodules with the left
and right actions of $h^{-1}t$ coinciding.  An important example of such a
bimodule is ${_{k}\mathscr Z_{\ell}}:=\secs({_k\scr
  T_\ell}[\hmon])$ or a tensor product of such bimodules.  These have
the further special property that $\gr ({_{k}\mathscr Z_{\ell}})$ is
supported on the diagonal in $\scrN_\eta\times_{\aone}
\scrN_{\eta}$; the same is thus true of any tensor product of these modules.

The preimage of the diagonal under $\pi\times \pi$ is just $\scrM_\eta\times_{\scrN_{\eta}}
\scrM_{\eta}$, so its intersection with the preimage of any $a\in
\aone$ is Lagrangian (by the semi-small property).  Thus, we have that:
\begin{lemma}\label{dagger}
Let $B$ be a filtered $\scrA_k\operatorname{-}\scrA_m$-bimodule which
is a subquotient of a tensor product of filtered bimodules of the form
${_{k'}\mathscr Z_{m'}}$, and whose support lies in $\fM\times \fM$
(i.e. whose classical limit $\gr B$ is killed by
$t$).  Then
there exists a nonzero polynomial $q_B(x)\in\C[x]$ such that $q_B(h^{-1}t)$
acts by zero on $B$.
\end{lemma}

\begin{proposition}\label{asympt-local}
There is a positive integer $p$ such that
$\Za[p]$ is Morita.
\end{proposition}

\begin{proof}
The statement that $\Za[p]$ is Morita can be broken down into 3
smaller statements:
\begin{itemize}
\item[(a)] There exists $p$ such that the bimodule ${}_{k}\Za_{k-1}$ is Morita for all $k\geq p$.
\item [(b)]  There exists $p$ such that the map \eqref{Morita} is surjective for all $k>m\geq p$.
\item [(c)] There exists $p$ such that the map \eqref{Morita} is injective for all $k>m\geq p$.
\end{itemize}

We first prove (a).  The bimodule ${}_{k}\Za_{k-1}$ is Morita if and only if
the maps 
\begin{equation}\label{adjacent}
{_{k}\Za_{k- 1}}\otimes_{A_{k- 1}} {_{k- 1}\Za_k}\to A_k
\end{equation}
and
\begin{equation}
\label{adjacentreverse}
{_{k-1}\Za_{k}}\otimes_{A_{k}} {_{k}\Za_{k-1}}\to A_{k-1}
\end{equation} are both isomorphisms.  
Let ${_{0}\mathscr T_{- 1}}$ be the $\scrQ_0 - \scrQ_{- 1}$ bimodule quantizing $\mathscr{L}$,
and let ${_{- 1}\mathscr T_{0}}$ be the $\scrQ_{-1} - \scrQ_0$ bimodule quantizing $\mathscr{L}^{-1}$.
Using notation similar to that of Proposition \ref{general pullback}, we have 
$${_{k}\cT_{k- 1}} \cong \sigma_k^*\big({_{0}\mathscr T_{-
    1}}\big):={_{0}\mathscr T_{- 1}}\big/(t-kh) \cdot {_{0}\mathscr T_{- 1}}|_{\fM}$$
$${_{k- 1}\cT_{k}} \cong \sigma_k^*\big({_{- 1}\mathscr
  T_{0}}\big):={_{- 1}\mathscr
  T_{0}}\big/(t-hk) \cdot {_{- 1}\mathscr
  T_{0}}|_{\fM},$$
  which induces maps
  \begin{equation}
\sigma_k^*\big({_{0}\mathscr Z_{-1}}\big):={_{0}\mathscr Z_{-1}}\big/(t-kh) \cdot {_{0}\mathscr Z_{-1}}\to  {_{k}{Z}_{k- 1}}\label{eq:sigma-Z1}
\end{equation}  \begin{equation}
\sigma_k^*\big({_{-1}\mathscr
  Z_{0}}\big):={_{-1}\mathscr
  Z_{0}}\big/(t-kh) \cdot {_{-1}\mathscr
  Z_{0}}\to {_{k- 1} {Z}_{k}}.\label{eq:sigma-Z2}
\end{equation}
Consider the short exact sequence 
$$0 \longrightarrow {_{0}\mathscr T_{-1}}\overset{t-kh}\longrightarrow {_{0}\mathscr
  T_{- 1}} \longrightarrow {_{k}\cT_{k- 1}}\longrightarrow 0.$$
Adjoining $\hmon$ and taking sections, we obtain a long exact sequence $$ 0 \longrightarrow {_{0}\mathscr Z_{-1}}\overset{t-kh}\longrightarrow {_{0}\mathscr
  Z_{- 1}} \longrightarrow {_{k}Z_{k- 1}}\longrightarrow H^1(\scrM_\eta;{_k\scr
  T_{k-1}}[\hmon]) \longrightarrow \cdots.$$
This tells us that the map \eqref{eq:sigma-Z1} is injective, with cokernel equal to the
submodule of $H^1(\scrM_\eta;{_k\scr T_{k-1}}[\hmon])$ annihilated by $t-kh$.
  Note that the associated graded
of the bimodule $H^1(\scrM_\eta;{_k\scr
  T_{k-1}}[\hmon])$ is supported over $0\in \aone$, since all other
fibers of $\pi$ are affine (Proposition \ref{ample-affine}).  
By Lemma \ref{algebraic-minimal}, there exists a nonzero polynomial $f(x)$ such that $f(h^{-1}t)$
acts by zero on $H^1(\scrM_\eta;{_k\scr T_{k-1}}[\hmon])$.
If $t-kh$ fails to act injectively on $H^1(\scrM_\eta;{_k\scr T_{k-1}}[\hmon])$, then 
so does $h^{-1}t-k$, which implies that $k$ is a root of $f(x)$.
Since there are only finitely many roots, there exists a $p$ such that $t-kh$ acts injectively
for all $k\geq p$, and therefore the map \eqref{eq:sigma-Z1} is an isomorphism.  
The same argument with $k$ and $k-1$ reversed applies to the map \eqref{eq:sigma-Z2}.

Now, consider the tensor product map 
\begin{equation}\label{map}
{_{0}\mathscr
Z_{- 1}}\otimes_{\mathscr{A}_{- 1}} {_{- 1}\mathscr
Z_{0}}\to
\mathscr{A}_{0}.
\end{equation} 
Let $K$ be the kernel and $E$ be the cokernel of this map, which are
bimodules over $\mathscr{A}_{0}$.  Over non-zero elements of $\aone$,
the fibers are affine, so this map is an isomorphism.  Thus $\gr K$ and
$\gr E$ are killed by $t$, and 
Lemma \ref{dagger} applies.  Thus,  there are minimal polynomials for
$h^{-1}t$ acting on these modules given
by $q_K$ and $q_E$.  

The usual spectral sequence for tensor product
shows that the cokernel of the map \eqref{adjacent} is $\sigma_k^* E = E \big/ (h-tk)\cdot E$,
and the kernel of this map is an extension of $\mathbb{R}^1\sigma_k^*(
E)\cong \operatorname{Tor}^1_k(\C,E)$ and
$\sigma_k^*( K)$.  Possibly increasing the $p$ introduced earlier, we
can assume that for $k\geq p$, the element $h^{-1}t-k$ acts
invertibly on $E$ and $K$.  Thus, we
have \[\sigma_k^*(E)=\sigma_k^*( K)=\mathbb{R}^1\sigma_k^*(
E)=0.\]  
This shows that 
\eqref{adjacent} is an isomorphism.  A completely symmetric argument
shows that after increasing $p$ again, we may also conclude that the map
\eqref{adjacentreverse} is an isomorphism, and so (a) is established.

We next prove (b).
Fix an integer $r$ such that $R(\cL)$ is generated in degrees less than
or equal to $r$; it follows that $Z$ is generated by ${_kZ_m}$ for $k-m\leq r$.
For $k$ and $m$ such that $k-m\leq r$, we can proceed exactly as in the proof of (a) to find a $p$
such that the map \eqref{Morita} is a surjection whenever $m\geq p$.
For the rest of the cases, we can induct on the quantity $k-m-r$.
Our inductive hypothesis tells us that
the image of the map \eqref{Morita} contains the image of the
multiplication map ${}_k\Za_{q}\otimes_{A_q}{_q\Za_m}$ for all
$k>q>m$.  Thus, the associated graded of the image of \eqref{Morita}
contains all elements of $R(\cL)$ of degree $k-m$ which can be written as a sum of
products of lower degree elements.  Since elements of degree $r\leq
k-m$ generate $R(\cL)$, this implies that the map \eqref{Morita} is indeed surjective, and (b) is proved.

Finally, we use (a) to prove (c).  
Choose $p$ such that the map
${_{j+1}Z_j}\otimes{_jZ_{j+1}}\to {_{j+1}Z_{j+1}}$ is an isomorphism for all $j\geq p$.
Now let $k>m\geq p$ be given, and consider the maps
$${_kZ_{k-1}}\otimes{_{k-1}Z_{k-2}}\otimes\cdots\otimes{_{m+1}Z_m}\otimes{_mZ_{m+1}}\otimes\cdots\otimes{_{k-2}Z_{k-1}}
\otimes{_{k-1}Z_k}$$
\begin{equation}\label{first}\downarrow\end{equation}
$${_kZ_m}\otimes{_mZ_{m+1}}\otimes\cdots\otimes{_{k-2}Z_{k-1}}\otimes{_{k-1}Z_k}$$
\begin{equation}\label{second}\downarrow\end{equation}
$${_kZ_k} = A_k.$$
By our choice of $p$, the composition of the maps \eqref{first} and \eqref{second} is an isomorphism.
It follows that \eqref{first} is injective.
Since the map \eqref{first} is the tensor product of the map \eqref{Morita} with the Morita bimodule 
${_mZ_{m+1}}\otimes\cdots\otimes{_{k-2}Z_{k-1}}\otimes{_{k-1}Z_k}$,
the map \eqref{Morita} must also be injective. 
\end{proof}

Propositions \ref{Morita-iff-loc} and \ref{asympt-local} immediately yield the following corollary.

\begin{corollary}\label{large quantizations}
There is an integer $p$ such that localization holds for
$\cD_k$ for all $k\geq p$.
\end{corollary}

\begin{remark}\label{ample vs very ample}
Corollary \ref{large quantizations} is precisely the first statement of Corollary  \ref{first-cor}  from the introduction 
for very ample line bundles.
If $\eta$ is only ample, then there exists a positive integer $r$ such that $r\eta$ is very ample, and we obtain
Corollary \ref{first-cor} by applying Corollary \ref{large quantizations} with $\la' = \la + j\eta$ and $\eta' = r\eta$ for $j=0,1,\ldots,r-1$.
\end{remark}

It is still desirable to have a non-asymptotic result; that is, a necessary and sufficient condition 
for localization to hold for $\cD$ itself in terms of $\Z$-algebras. 
Let $Z^{(p)}$ be the $\Z$-algebra defined by ${}_kZ^{(p)}_m\cong {}_{kp}Z_{mp}$ with the obvious product structure.
It is clear that $Z^{(p)}$ is isomorphic to the $\Z$-algebra $Z(\fM, \cQ, \cL^p)$.

\begin{lemma}\label{restriction}
For all $p$, the restriction functor $\Zmod/\Zmodbd\to\Zpmod/\Zpmodbd$ is an equivalence of categories.
\end{lemma}

\begin{proof}
By Theorem \ref{Zloc}, both the source and the target are equivalent to $\Dmod$, and it is easy to check that
these equivalences are compatible with the restriction functor. 
\end{proof}

\begin{proposition}\label{morita-loc}
Localization holds for $\cD$ if and only if $\Zp$ is Morita for some $p$.  
\end{proposition}

\begin{proof}
If $\Zp$ is Morita, then the functor $\secs\colon\Dmod\to\Amod$ factors as
$$\Dmod\to\Zmod/\Zmodbd\to\Zpmod/\Zpmodbd\to \Amod,$$
where the first functor is the equivalence of Theorem \ref{Zloc}, the second is the equivalence
of Lemma \ref{restriction}, and the last is the equivalence of \cite[\S 5.5]{GS}.  Thus localization holds for $\cD$.

Conversely, assume that localization holds for $\cD$.  By Theorem \ref{asympt-local}, there is an integer $p$
such that $Z[p]$ is Morita, which easily implies that $\Zp[1]$ is Morita.  
We need to extend this to show that $\Zp$ is Morita, which involves showing that the bimodule ${_pZ_0}$ 
is Morita and the multiplication map ${_{2p}Z_p}\otimes{_pZ_0}\to {_{2p}Z_0}$ is an isomorphism.
The fact that ${_pZ_0}$ is Morita follows from the natural isomorphism of functors
$${_pZ_0}\otimes - \;\;\cong\;\; \secs\big({_p\cT'_0}\otimes\Loc(-)\big)$$ along with the fact that localization holds for both $\cD$
and $\cD_p$.
Similarly, the fact that the multiplication map is an isomorphism follows from the natural isomorphism of functors
$${_{2p}Z_p}\otimes - \;\;\cong\;\; \secs\big({_{2p}\cT'_p}\otimes \Loc(-)\big)$$
applied to the module ${_{p}Z_0}$.
\end{proof}

\begin{remark}
The ``if'' direction of Proposition \ref{morita-loc} is very close in content to \cite[2.10]{KR} (though they do not use the language
of $\Z$-algebras) and our proof draws heavily on theirs.
We note, however, that Proposition \ref{asympt-local} and Corollary \ref{large quantizations} have no analogues in \cite{KR}.
\end{remark}

\subsection{Comparison of the analytic and algebraic categories}
\label{GAGA}
We keep our running assumptions from the start of Section
\ref{sec:Z-algebra}, and assume for simplicity that $S$ is smooth and that $\C[\fX]^\bS=\C$.  Up until this point we have worked exclusively in the algebraic 
category, quantizing the sheaf of regular functions in the Zariski topology.
On the other hand, some other important papers have considered quantizations 
of the functions on an analytic variety, for example \cite{KR,KSdq}.  We will need to apply some 
results from these papers below, so we must prove a comparison theorem 
relating quantizations and their module categories for the nondegenerate Poisson scheme
$\fX$ and its analytification $\fX^{\an}$.


First, we note that every quantization $\cQ$ in the Zariski topology 
introduces a corresponding quantization $\cQ^{\an}$ 
of the structure sheaf in the analytic category.  To see this, we can consider 
the jet bundle $J_\infty\cQ$, which is a pro-vector bundle on $\fX$
with flat connection whose sheaf of flat sections is $\cQ$, as explained in \cite[1.4]{BK04a}.
The corresponding sheaf of analytic sections $(J_\infty\cQ)^{\an}$ again has
a flat connection, and we let
$\cQ^{\an}$ be its sheaf of flat sections.  We have a map $\a^{-1}\cQ\to \cQ^{\an}$
where $\a\colon \fX^{\an} \to \fX$ is the identity on
points.  
If $\cQ$ is $\bS$-equivariant, so is $\cQ^{\an}$. 

As in the Zariski topology, we
let $\cD^{\an}:=\cQ^{\an}[\hmon]$.  Similarly, for any $\cD$-module $M$, we let 
$M^{\an}:=\a^{-1}M\otimes_{\a^{-1}\cD} \cD^{\an}$.
As in \cite{KR}, we call an $\bS$-equivariant $\cD^{\an}$-module {\bf good}
if it admits a coherent $\bS$-equivariant $\cQ^{\an}|_{U}$-lattice on every relatively compact open subset of
$\fX$.  If $M$ is a good $\cD$-module, $M^{\an}$ is a good $\cD^{\an}$-module.  

\begin{theorem}\label{th:GAGA}
  The functor
  $(-)^{\an}\colon \cD\mmod\to \cD^{\an}\mmod$ is an equivalence of
  categories.  
\end{theorem}

\begin{proof}
  In essence, the proof is simply to observe that a version Theorem \ref{Zloc} holds in
  the analytic topology.  More precisely, we define the quantum homogeneous coordinate ring
  $Z^{\an}$ exactly as we defined $Z$.  There is a canonical map from $Z$ to $Z^{\an}$,
  and we claim that it is an isomorphism.
  
  In bidegree $(0,0)$, this map is the map from $\secs(\cD)$ to $\secs(\cD^{\an})$.
  To see that this is an isomorphism, it is enough to show that the associated
  graded map $\Gamma(\fS_\fM)\to\Gamma(\fS_\fM^{\an})^{\fin}$ is an isomorphism, 
  where $(-)^{\fin}$ denotes the subalgebra of
  $\bS$-locally finite vectors.   Since all $\bS$-weights on $\C[\fM]$ are
  positive, any $\bS$-weight vector in $\Gamma(\fS_\fM^{\an})$ can be
  interpreted as a section of a line bundle on the projectivization of
  $\fX_0$ for the $\bS$ action; by the classic GAGA theorem of Serre \cite{GAGA},
  this is in fact algebraic, and thus arises from an algebraic
  function on $\fX_0$.  The argument in arbitrary bidegree follows from a similar analysis
  of sections of line bundles.
  
  Now that we know that $Z$ and $Z^{\an}$ are isomorphic, we have a functor from $\cD^{\an}\mmod$
  to $\cD\mmod$ given by the composition
  $$\cD^{\an}\mmod \overset{\left(\secs^\Z\right)^{\an}}\longrightarrow Z^{\an}\mmod\cong \Zmod
  \overset{\Loc^\Z}\longrightarrow\cD\mmod.$$
This functor splits $(-)^{\an}$ and is exact (since the cohomology of a
  sufficiently high twist with $\cL$ vanishes), so to check that it
  gives an equivalence, we need only check that it kills no module $\mathcal{K}$.  Thus, we need only show that for any good
  $\bS$-equivariant $\cD$-module,  we must have that $\secs^\Z$ is not
  0.  Since $\cL$ is ample, $\mathcal{K}/h\mathcal{K}\otimes \cL^k$ has
  non-zero sections for $k\gg 0$ unless $\mathcal{K}/h\mathcal{K}=0$; then Nakayama's lemma tells us that
  ${}_{k}{\cT_0'}^{\an}\otimes_{\cD_0}\mathcal{K}$ has non-zero sections as
  well unless $\mathcal{K}=0$.  This completes the proof.
\end{proof}

\begin{remark} Hou-Yi Chen \cite{HYC} proves a version of Theorem \ref{th:GAGA} in the more general
context of DQ-algebroids, but subject to the hypothesis that $\fX$ is
projective over a point (which is never the case for a conical symplectic resolution of positive dimension).  
Chen uses a more direct reduction to Serre's classic
GAGA theorem than we do; it is possible that his techniques could be adapted to our setting, as well.
\end{remark}

\begin{remark}
  It might worry the reader that we used some analytic techniques in
  the proof of Proposition \ref{line bundles}, used that result in the
  proof of Theorem \ref{Zloc}, and then used that in the proof of Theorem \ref{th:GAGA}; at first glance, this looks as though it may be circular.  In fact, in the proof of Proposition \ref{line bundles}, we
  use only the comparison theorem between algebraic and analytic de
  Rham cohomology; nothing in the vein of GAGA.

Similarly, it might worry the reader that we use Theorem \ref{th:GAGA}
in the proof of Lemma \ref{minimal-polynomial} earlier in the paper,
but Lemma \ref{minimal-polynomial} is
only used in the proof of the localization results, Theorem
\ref{derived-local} and Proposition \ref{asympt-local}, which are not
used in this section.
\end{remark}

\subsection{Twisted modules and the Kirwan functor}
\label{sec:kirwan}
In this section, we return to the assumptions of Section \ref{QHr},
while keeping those introduced at the start of \ref{sec:Z-algebra}.  That
is we additionally assume that we
have a Hamiltonian action of 
a connected reductive algebraic group $G$ on $(\fX,\cQ)$ such that
$\C[\fX]^{G\times \bS}$ with quantized moment map $\eta\colon
U(\mg)\to A$, which induces a flat commutative moment map $\fX\to \mg^*$.
We fix a $G$-equivariant ample line bundle $\cL$ on $\fX$ and we let
$\fU$ be its semistable locus.
We assume that the $G$ action on $\fU$ is free and, and if
$\fX_{\red}$ is the reduced space (with $\cL_{\red}$ its induced ample line bundle), that we have an induced isomorphism
$\C[\mu^{-1}(0)]^G=\C[\fX_{\red}]$, and more generally an isomorphism
$\Gamma(\mu^{-1}(0);\cL^k)^G=\Gamma(\fX_{\red},\cL^k)$.  

Fix an element $\xi\in\chi(\mg)$. We'll let $\cD_{\red}$ be the 
quantization of $\fX_{\red}$ defined by reduction by $\eta-\xi$ (as
defined in Section \ref{QHr}).
We call a $G$-equivariant object $\cN$ of $\Dmod$ (respectively $\cD_\fU\mmod$) {\bf $\xi$-twisted}
if, for all $x\in\mg$, the action of $x$ on $\cN$ induced by the $G$-structure coincides with 
left multiplication by the element $\eta(x)-\xi(x)\in A$. 
Let $\Dmod_\xi$ (respectively $\cD_\fU\mmod_\xi$) denote the full subcategory of $\xi$-twisted objects of 
$\Dmod$ (respectively $\cD_\fU\mmod$).
Kashiwara and Rouquier \cite[2.8(ii)]{KR}
prove that $\cD_\fU\mmod_\xi$ is equivalent to $\cDred\mmod$ via the functor that takes
$\cN$ to $\psi_*\sHom(\cE_\xi,\cN)$, where
$\cE_\xi$ is the sheaf defined in Section \ref{QHr}.

Define the functor $\kappa\colon\Dmod\to\cD_{\red}\mmod$ by putting
$$\kappa(\cN) :=\psi_*\sHom(\cE_\xi,\cN_\fU)$$
for all $\cN$ in $\Dmod$.
We call $\kappa$ the {\bf Kirwan functor} in analogy with the Kirwan map
in (equivariant) cohomology.  Our main result in this section will be Theorem \ref{surjectivity},
which says that the Kirwan functor is essentially surjective.  To prove this theorem, we introduce
all of the analogous constructions in the context of $\Z$-algebras.

In Section \ref{sec:Z-algebra}, we defined a
$\Z$-algebra
$\Za = Z(\fX,\cD, \cL)$ and functor $\secs^\Z\colon \cD\mmod\to \Za\mmod$.
We may also define the $\Z$-algebra $\Za_{\red} = Z(\fX_{\red}, \cDred, \cL_{\red})$, with its own
sections functor $\Gamma^{\Z}_{\bS,\red}\colon \cD_{\red}\mmod\to
\Za_{\red}\mmod$.

By assumption, we have a ring
homomorphism $$\eta\colon U(\mg)\longrightarrow A = \secs(\cD)\cong {_0Z_0}.$$
Moreover, 
for all $m\geq 0$, there is a unique homomorphism
$$\eta_m\colon U(\mg)\longrightarrow\secs(\cD_m)\cong{_mZ_m}$$
such that $\eta_0 = \eta$ and for all $x\in\mg$, the action of $x$ on $\cL$ induced by the $G$-structure
coincides with that induced by the adjoint action, via $\eta_{m+1}$ and $\eta_m$, 
on the $\cD_{m+1}-\cD_m$ bimodule ${_{m+1}\cT_m}$
that quantizes $\cL$.  By Proposition \ref{quotient}, we can describe
$A_{\red}$ as an algebraic reduction of $A$, and similarly, we have a
map ${_{j}Z_k}\to \Gamma(\fU;{_j\cT_k})$, which induces a map
\[{_iY_j} := {_iZ_j}\Big{/}{_iZ_j}\cdot \langle
  \eta_j(x)-\xi(x)\mid x\in \mg\rangle\to \Gamma(\fU;{_j\cT_k}\otimes
  \cE_\xi).\]
  \begin{lemma}\label{Z-bimod}
The induced map    ${_iY_j}^G\to {_i(\Za_{\red})_j}$ is an isomorphism.
  \end{lemma}
  \begin{proof}
    The proof is essentially the same as Proposition
    \ref{quotient}.  The  associated graded map
    $\Gamma(\mu^{-1}(0);\cL^{i-j})^G\to\Gamma(\fX_{\red},\cL_{\red}^{i-j})$
    is an isomorphism
    by assumption, so this implies the same for the map under consideration.
  \end{proof}

We say that a $G$-equivariant $\Za$-module $N = \bigoplus_m {_mN}$  is
{\bf $\xi$-twisted} if, for all $x\in \mg$, the action of $x$ on ${_mN}$ induced by the $G$-structure
coincides with left multiplication by the element $\eta_m(x) - \xi(x)\in {_mZ_m}$.
We denote the category of such modules $\Zmod_\xi$.

Lemma \ref{Z-bimod} tells us that $Y$ is a naturally a $\Za-\Za_{\red}$
  bimodule.
We define the {\bf $\Z$-Kirwan functor} $$\kappa^\Z:=\Hom_{\Za}(Y,-)\colon\Zmod\to
\Za_{\red}\mmod$$ along with its left adjoint
 $$\kappa^\Z_! := Y\otimes_{\Za_{\red}}-: \Za_{\red}\mmod\to\Zmod.$$

\begin{remark}
Every $\Za$-module $N$ has a largest submodule $N_\xi$ on which
$\eta_j(x)-\xi(x)$ acts locally finitely and on which the $\mg$-action
integrates to a $G$-action.  The $G$-action
makes $N_\xi$ a $\xi$-twisted equivariant module in a canonical way.  
Because of the $\xi$-twisted condition, the $G$-invariant part of $N_\xi$ is already a module over $\Za_{\red}$ in
the obvious way, and we have a canonical isomorphism
$\kappa^\Z(N)\cong N_\xi^G$. 
\end{remark}

\begin{proposition}\label{mod-bounded}
  The functors $\kappa^\Z$ and $\kappa^\Z_!$ both
  preserve boundedness and thus
  induce functors 
\[\kappa^\Z\colon \Zmod/\,\Zmodbd\to
\Za_{\red}\mmod/\,\Za_{\red}\mmod_{\operatorname{bd}}\]
and
\[
\kappa^\Z_!\colon
\Za_{\red}\mmod/\,\Za_{\red}\mmod_{\operatorname{bd}}\to \Zmod/\,\Zmodbd.\]
\end{proposition}
\begin{proof}
  The functor $\kappa^\Z$ obviously sends bounded modules to bounded
  modules.
To see that $\kappa_!^\Z$ preserves boundedness, find integers $N$ and $M$ such that all of the higher cohomology groups of $\cL^{N}$ and 
$\cL^{M}_{\red}$ vanish.
Then for any non-negative integers $i,j,k$ with $i\geq  j+N \geq k+N+M$
the associated graded of the multiplication map
 ${}_iY_{j}\otimes {}_j(\Za_{\red})_k\to {}_iY_k$. 
is \[\Gamma(\mu^{-1}(0);\cL^{i-j})^G\otimes
\Gamma(\fX_{\red};\cL^{j-k}_{\red})\to \Gamma(\mu^{-1}(0);\cL^{i-k})^G.\]  Since $\oplus_n
\Gamma(\mu^{-1}(0);\cL^{n})^G$ is a finitely generated module over
$\oplus_n\Gamma(\fX_{\red};\cL^{n}_{\red})$, there is some $N'$ such
that $\oplus_{N\leq n\leq N'}\Gamma(\mu^{-1}(0);\cL^{n})^G$
generates $\oplus_{N\leq n}\Gamma(\mu^{-1}(0);\cL^{n})^G$.   That
is, if we fix $i$ and $k$ such that $i - k\geq N$, then $\Gamma(\fX;\cL^{i-k})^G$ 
is spanned by
the images of the maps \[\Gamma(\mu^{-1}(0);\cL^{i-j})^G\otimes
\Gamma(\fX_{\red};\cL^{j-k}_{\red}) \to \Gamma(\mu^{-1}(0);\cL^{i-k})^G\,\text{ for all $j$ such that }\,N\leq i-j\leq N'.\]   

We may as well assume that $N'\geq M+N$.
Thus, since a map whose associated graded is surjective is itself
surjective, we see that the
map \[\bigoplus_{j\geq i-N'}{}_iY_{j}\otimes {}_j(\Za_{\red})_k\to
{}_iY_k\] 
is surjective for all $k\leq i-N'$.
If $M$ is a $\Za_{\red}$-module, it follows that if $i\geq N'$ then
$\kappa^\Z_!(M)_{i}$ is spanned by the images of ${}_iY_{j}\otimes M_j$ for $j\geq
i-N'$.  Then if  $M_p=0$ for
$p\geq P$, we have $\kappa^\Z_!(M)_{i}=0$ whenever $i>P+N'$.

This shows that both $\kappa^\Z$ and $\kappa_!^\Z$ preserve bounded
modules and thus induce functors on the quotient categories.
\end{proof}

\begin{proposition}\label{square}
The following diagram commutes.
\[\tikz[->,thick]{
\matrix[row sep=10mm,column sep=20mm,ampersand replacement=\&]{
\node (a) {$\Dmod$}; \& \node (c) {$\Zmod/\,\Zmodbd$};\\
\node (b) {$\cD_{\red}\mmod$}; \& \node (d) {$\Za_{\red}\mmod/\,\Za_{\red}\mmod_{\operatorname{bd}}$};\\
};
\draw (a) -- (c) node[above,midway]{$\secs^\Z$};
\draw (a) -- (b) node[left,midway]{$\kappa$};
\draw[->] (b) --(d) node[above,midway]{$\Gamma^{\Z}_{\bS,\red}$};
\draw (c)--(d) node[right,midway]{$\kappa^\Z$};
}\]
\end{proposition}
(Note that the horizontal arrows are equivalences by Theorem \ref{Zloc}.)\\

\begin{proof}
Fix an object $\cN$ of $\Dmod$.  
First, note that we can assume that $\cN=\cN_\xi$,
that is, that $\cN$ has an $G$-equivariant structure agreeing with that
induced by $\eta-\xi$.  This is because passing to
the largest submodule that has such a structure doesn't change $\kappa$ or $\kappa^\Z$.

 With this assumption, we have a restriction map 
  $$\kappa^\Z\left(\secs^\Z (\cN)\right) = \;\left(\secs^\Z (\fX; \cN)\right)^G \longrightarrow\left(\secs^\Z (\fU; \cN)\right)^G
 \cong\;\;\Gamma^{\Z}_{\bS,\red}\!\big(\kappa(\cN)\big)$$
where $\secs^\Z(\fX;-) = \secs^\Z$, and $\secs^\Z(\fU; -)$ denotes the same functor 
defined using the set $\fU$ of stable points.
 As in the proof of Theorem \ref{Zloc}, let $\bar\cN :=
 \cN(0)/\cN(-1)$.

For each $m\in\Z$, the restriction from $\fX$ to $\fU$ gives the following long
  exact sequence in local cohomology.
  \begin{multline*}
    H^0_{\fX\setminus \fU}\left(\bar{\cN}\otimes \cL^m\right)^G\longrightarrow
    \Gamma\left(\fX;\bar{\cN}\otimes \cL^m\right)^G \longrightarrow
    \Gamma\left(\fU;\bar{\cN}\otimes \cL^m\right)^G\longrightarrow H^1_{\fX\setminus
      \fU}\left(\bar{\cN}\otimes \cL^m\right)^G\longrightarrow \cdots
  \end{multline*}
 The space 
\begin{equation}\label{sections1}
 \bigoplus_{m\geq 0}H^0_{\fX\setminus
   \fU}\left(\bar{\cN}\otimes \cL^m\right)
\end{equation} 
of sections of twists of $\bar\cN$ which are supported on $\fX\setminus
   \fU$ is finitely generated over the
 ring 
\begin{equation}\label{sections2}
 \bigoplus_{m\geq 0}\Gamma\left({\fX\setminus \fU}; \cL^m\right)
\end{equation} 
of sections of powers of the restriction of $\cL$ to $\fX\setminus\fU$.
 Since $G$ is reductive, the invariant part of \eqref{sections1} is finitely generated
 over the invariant part of \eqref{sections2}.  The invariant part of
 \eqref{sections2} is a single copy of $\C$, since any invariant section
 of $\cL^m$ for $m>0$ vanishes on all unstable points.
 Thus $H^0_{\fX\setminus \fU}\left(\bar{\cN}\otimes \cL^m\right)^G$
 vanishes for $m\gg 0$.

The module $H^1_{\fX\setminus \fU}\left(\bar{\cN}\otimes \cL^m\right)$
is {\it not} in general finitely generated as a module over the invariant section ring.  On
the other hand, the module $\bigoplus_{m\geq
  0}\Gamma\left(\fU;\bar{\cN}\otimes \cL^m\right)^G$ is the sections
of the twists of a coherent sheaf on the quotient $\fU/G$, which is
projective over an affine variety, and thus finitely generated over
the invariant section ring $\bigoplus_{m\geq 0}\Gamma\left({ \fU};\cL^m\right)^G.$
In particular, its image in $\bigoplus_{m\geq
  0} H^1_{\fX\setminus \fU}\left(\bar{\cN}\otimes \cL^m\right)^G$ under
the boundary map is finitely generated over the same ring.

 Since any positive degree
invariant section of $\cL$ vanishes on $\fX\setminus\fU$, its action on local
cohomology is locally nilpotent; this implies that there is some
integer $k$ such that all invariants of degree $\geq k$ act trivially
on the image of $\bigoplus_{m\geq
  0}\Gamma\left(\fU;\bar{\cN}\otimes \cL^m\right)^G$ under the
boundary map.  This in turn implies that the image is trivial for $m$ sufficiently large.  Note that we used the fact that the image is
finitely generated in both of these steps.

 It follows that the restriction map
 $$\left(\Gamma(\fX; \bar\cN\otimes\cL^m)\right)^G \longrightarrow\left(\Gamma (\fU; \bar\cN\otimes\cL^m)\right)^G$$
 is an isomorphism for $m\gg 0$.
 We next observe that 
 $$\left(\Gamma (\fX; \bar\cN\otimes\cL^m)\right)^G \cong\; \gr\left(\secs^\Z (\fX; {_m\cT'_0}\otimes_\cD\cN)\right)^G
 \cong\; \gr\left(\secs^\Z (\fX;\cN)[m]\right)^G$$
 and similarly $$\left(\Gamma (\fU; \bar\cN\otimes\cL^m)\right)^G \cong\; \gr\left(\secs^\Z (\fU;\cN)[m]\right)^G,$$
 where $[m]$ denotes a shift as in Definition \ref{def:Z-shift}.
 Since maps that induce isomorphism on associated
 graded are isomorphisms, we may conclude that 
 the restriction map
 $$\left(\secs^\Z (\fX;\cN)[m]\right)^G\longrightarrow \left(\secs^\Z (\fU;\cN)[m]\right)^G$$
 is an isomorphism for $m\gg 0$.  This is equivalent to the statement that the kernel and cokernel of the map
  $$\left(\secs^\Z (\fX;\cN)\right)^G\longrightarrow \left(\secs^\Z (\fU;\cN)\right)^G$$
are bounded, as desired.
\end{proof}

\begin{lemma}\label{adjoints}
The Kirwan functor $\kappa$ has a left adjoint $\kappa_!$ such that $\kappa\circ\kappa_!$
  is isomorphic to the identity functor on $\cD_{\red}\mmod$.
\end{lemma}

\begin{proof}
  By Theorem \ref{square}, we may work instead with the $\Z$-Kirwan
  functor $\kappa^\Z$ and its left adjoint $\kappa^\Z_!$.
  Let ${_iY'_j}\subset {_iY_j}$ be the sum of all non-trivial
  $G$-isotypic components.  Since $G$ is reductive, ${_iY_j}$ is
  isomorphic to ${_iY'_j}\oplus {_iY^G_j}$.  There is a natural map
  from ${_iY^G_j}$ to ${_i(\Za_{\red})_j}$ whose associated graded is
  the map $\Gamma(\mu^{-1}(0);\cL^{i-j})^G\to
  \Gamma(\fX_{\red};\cL_{\red}^{i-j})$.  This map is an isomorphism
  when $i-j$ is sufficiently large, which implies that the same is
  true of the map ${_iY^G_j}$ to ${_i(\Za_{\red})_j}$.  Thus, modulo
  bounded modules, we have $Y\cong Y' \oplus \Za_{\red}$ as a right module
  over $\Za_{\red}$.  
Then for any $\Za_{\red}$-module $N$, we have
\[\kappa^\Z\circ\kappa^\Z_!(N) = \kappa^\Z(Y\otimes_{\Za_{\red}}N)
\cong (Y\otimes_{\Za_{\red}}N)^G\cong
\Za_{\red}\otimes_{\Za_{\red}}N\cong N,\]
modulo bounded modules. 
\end{proof}

\begin{remark}
  One can use similar principles to construct a right adjoint
  as well as a left to $\kappa$.  One considers the $\Za_{\red}-\Za$
  bimodule \[{}_iW_{j} := {_iZ_j}\Big{/} \langle \eta_j(x)-\xi(x)\mid
  x\in \mg\rangle \cdot {_iZ_j}.\] The obvious guess for the right
  adjoint based on general nonsense is $\Hom_{\Za_{\red}}(W,-)$;
  however, we need to exercise care here since $W$ is not finitely
  generated as a left module.  On the other hand, it 
  is (as a left module) the direct sum $W=\bigoplus_{\chi\in \hat G}W^\chi$ 
  of its isotypic components $W^\chi$ according to the natural $G$ action, and each
  isotypic component is finitely generated even after taking the
  associated graded by a classical theorem of Hilbert.  We should
  emphasize that here $\hat G$ is the set of all finite dimensional
  representations, not just 1-dimensional ones.

  A replacement for $\Hom_{\Za_{\red}}(W,-)=\prod_{\chi\in \hat
    G}\Hom_{\Za_{\red}}(W^\chi,-)$ with better finiteness properties
  is the direct sum $\kappa_*(-)=\bigoplus_ {\chi\in \hat
    G}\Hom_{\Za_{\red}}(W^\chi,-)$ which we can consider as the
  subspace of $\Hom_{\Za_{\red}}(W,-)$ which kills all but finitely many isotypic
  components.  This is closed under the action of $\Za$ acting on the
  right since the $G$-action on $\Za$ is locally finite.

It is still not obvious that $\kappa_*$ takes finitely generated modules to finitely generated modules.
When $X$ is the cotangent bundle to a smooth affine $G$-variety, 
this is proved in a recent preprint by McGerty and Nevins \cite[6.1(3)]{MN-Morse}.
\end{remark}

The following theorem, which is an immediate consequence of Lemma \ref{adjoints}, 
may be regarded as a categorical, quantum version of Kirwan surjectivity.

\begin{theorem}\label{surjectivity}
The Kirwan functor $\kappa$
is essentially surjective.
\end{theorem}

\begin{proof}
For any object of $\cD_{\red}\mmod$, we can apply the left
adjoint from Lemma \ref{adjoints} to obtain a 
witness to essential surjectivity.
\end{proof}

\begin{remark}\label{ksmn}
McGerty and Nevins \cite{MN} always work with symplectic quotients of affine schemes,
and the category of quantizations that they consider is by definition the essential image of the Kirwan functor.
Thus Theorem \ref{surjectivity} establishes that their module category is the same as ours.
\end{remark}

\section{Convolution and twisting}
\label{sec:bimodules}
Throughout Section \ref{sec:bimodules}, we'll only consider conical
symplectic resolutions $\fM$. Let $\nu\colon\fM\to\fM_0$ be the
resolution map with
Steinberg variety $\fZ:=\fM\times_{\fM_0}\fM$.  Consider the three different projections
$p_{ij}\colon\fM\times\fM\times\fM\to\fM\times\fM$ as well as the two
projections $p_i\colon\fM\times\fM\to\fM$.  The cohomology 
$H^{2\dim\fM}_\fZ(\fM\times\fM; \C)$ with supports in $\fZ$ has a convolution product
given by the formula
$$\a\star\b := (p_{13})_*(p_{12}^*\a\cdot p_{23}^*\b),$$
making it into a semisimple algebra \cite[8.9.8]{CG97}.  
For any closed subvariety $\fL\subset\fM$ with the property that $\fL =
\nu^{-1}(\nu(\fL))$, there is a degree-preserving action of this 
algebra on the cohomology $H^*_{\fL}(\M; \C)$ given by the
formula $$\a\star\gamma := (p_2)_*(\a\cdot p_1^*\gamma).$$ 

\begin{example}
When $\fM$ is the
cotangent bundle of the flag variety, $H^{2\dim\fM}_\fZ(\fM\times\fM; \C)$ is
isomorphic to the group ring of the Weyl group \cite[3.4.1]{CG97}, and
$H^*(\fM; \C)$ is isomorphic to the regular representation.
\end{example}

In this section, we explain how to categorify this action.  In Section
\ref{sec:HC}, we define the category of Harish-Chandra bimodules over
a pair of quantizations.  There is both an algebraic and a geometric
version of this definition, and they are related by the localization
and invariant section functors.  In Section \ref{sec:index}, we show
that a (geometric) Harish-Chandra bimodule has a characteristic cycle in
$H^{2\dim\fM}_\fZ(\fM\times \fM; \C)$, and tensor products of bimodules
categorify convolution product of cycles.  Furthermore, an object $\cN$ of
$D^b(\Dmod)$ has a characteristic cycle in $H^{\dim\fM}_\fL(\fM; \C)$ for any 
$\fL\subset \fM$ containing $\supp \cN$, and
we show that the tensor product action of bimodules on modules
categorifies the convolution action.  
In Section \ref{twisting bimodules} we define a particularly nice collection of (algebraic) 
Harish-Chandra bimodules, which we use in 
Section \ref{twisting} to study a certain collection of auto-equivalences of
$D^b(\Amod)$ related to twisting functors on BGG category $\cO$.

\subsection{Harish-Chandra bimodules}\label{sec:HC}
Recall that, for any $\la\in\Ht$, we let $A_\la := \secs(\cD_\la)$ be the section ring
of the quantization of $\fM$ with period $\la$.  Recall from
Proposition \ref{restriction of sections}
that we can write this ring for each $\la$ as a quotient of the sections $\scrA$ of a canonical quantization of the universal deformation $\scrM$.  
Let $H$ be a finitely generated $A_{\la'}$-$A_{\la}$ bimodule.   Recall
that $\gr A_\la\cong \gr A_{\la'}\cong\C[\fM]$, thus for any
filtration $H(0)\subset H(1)\subset \ldots \subset H$ which is compatible 
with the filtrations on $A_\la$ and $A_{\la'}$, the $\C[\fM]\otimes\C[\fM]$-module 
$\gr H$ may be
interpreted as an $\bS$-equivariant sheaf on $\fM_0\times\fM_0 \cong
\Spec \big(\C[\fM]\otimes\C[\fM]\big)$.

When $n$ is greater than 1, we will be interested in a {\bf thickened associated graded}
$\gr_n H:= R(H)/hR(H)$.  This is a module over $R(A_\la\otimes A_{\la'}^{\op})/hR(A_\la\otimes A_{\la'}^{\op})\cong
\C[\fM_0]\otimes\C[\fM_0]\otimes \C[\hon]/(h)$, and thus
over $\C[\fM_0]\otimes\C[\fM_0]$.  The module $\gr_n H$ is an $n$-fold self-extension of
$\gr H$, but this can be a non-split extension, so $\gr_n H$ contains
more information.

\begin{definition}\label{hca}
We say that $H$ is {\bf Harish-Chandra} if it is finitely generated
and it admits a filtration
such that $\gr_n H$ is scheme-theoretically supported on
the diagonal.  Equivalently, we require that if $a_\la\in A_\la(k)$ and $a_{\la'}\in A_{\la'}(k)$
are specializations of the same element $a\in\scrA$, then for all $h\in H(m)$,
we have $a_\la\cdot h - h\cdot a_{\la'} \in H(k+m-n)$.
\end{definition}

\newcommand{\DHCa}[2]{D^b_{\mathbf{HC}}({A_{#1}\operatorname{-mod-}A_{#2}})}

\newcommand{\DHCg}[2]{D^b_{\mathbf{HC}}({\cD_{#1}\operatorname{-mod-}\cD_{#2}})}

Let ${_{\la'}^{\mbox{}}\HCa_{\la}}$
be the category of Harish-Chandra bimodules, and let
$\DHCa{\la'}{\la}$ be the full subcategory of $D^b(A_{\la'}\operatorname{-mod-}A_{\la})$
consisting of objects $H$ whose cohomology $\coho^i{( H)}$ is Harish-Chandra.

\begin{proposition}\label{prop:HCa-tensor}
 If $A_{\la'}$ has finite global dimension, $H_1\in \DHCa{\la'}{\la}$, and $H_2\in
 \DHCa{\la''}{\la'}$, then $H_2\Lotimes H_1\in \DHCa{\la''}{\la}$.
\end{proposition}
\begin{proof}
Consider the Rees modules  $R(H_1),R(H_2)$
associated to some good filtration.
These modules have locally free resolutions over $R(A_{\la'}\otimes A_{\la}^{\op})$
and $R(A_{\la''}\otimes A_{\la'}^{\op})$ such that, if $f$ is congruent to $f'$ modulo $h$,
then $f'\otimes
1-1\otimes f$ acts trivially on the cohomology of $R(H_i)$ modulo $h$.
By a standard result of homological algebra, there
exists a homotopy $p_i$ on the resolution of $R(H_i)$ such that $f'\otimes
1-1\otimes f+\partial p_i+p_i\partial$ acts trivially modulo $h$ on the
complex itself.  Then
$p_2\otimes 1+1\otimes p_1$ is a homotopy that plays the same role
for $f''\otimes 1-1\otimes f$ acting on the tensor product of these
complexes.  This shows that  $f''\otimes 1-1\otimes f$ acts by 0
modulo $h$ on the cohomology of $R(H_2)\Lotimes R(H_1)$, so $H_2\Lotimes H_1\in \DHCa{\la''}{\la}$.
\end{proof}

We can view $A_{\la'}\otimes A_{\la}^{\op}$ as the ring of $\bS$-invariant sections of a
sheaf $\cD_{\la'}\hat\boxtimes\cD_{\la}^{\op}$ on $\fM\times \fM$; we
must complete the naive tensor product in the $h$-adic topology in
order to satisfy the hypotheses of a quantization.  As a quantization, 
$\cD_{\la'}\hat\boxtimes\cD_{\la}^{\op}$ has period $(\la',-\la)$.  By a $\cD_{\la'}-\cD_{\la}$ bimodule,
we mean an $\bS$-equivariant sheaf of
$\cD_{\la'}\hat\boxtimes\cD_{\la}^{\op}$-modules on $\fM\times
\fM$.  We let $D^b(\cD_{\la'}\operatorname{-mod-}\cD_{\la})$ be the
bounded derived
category of good $\cD_{\la'}\hat\boxtimes\cD_{\la}^{\op}$-modules.

\begin{definition}\label{hcg}
  For any $\la,\la'\in\Ht$, let $\HCglala$ be the category of good
  $\cD_{\la'}-\cD_{\la}$ bimodules $\cH$ with ``thick classical limits'' that
  are scheme-theoretically supported on the Steinberg $\fZ\subset\fM\times\fM$.  More precisely, if $\scrQ$ is the canonical quantization of $\scrM$,
  we require $\cH$ to admit a lattice $\cH(0)$ such that for all sections $\tilde f$ of $\scrQ$,
  $\cH(0)$ is invariant under
  $h^{-1}(\tilde f_{\la'}\otimes 1-1\otimes \tilde f_{\la})$, where $\tilde f_{\la'}$ and
  $\tilde f_{\la}$ are the specializations of $\tilde f$ at $\la'$ and
  $\la$, respectively.  As in the algebraic setting, we define
  $\DHCg{\la'}{\la}$ to be the full subcategory of
  $D^b(\cD_{\la'}\operatorname{-mod-}\cD_{\la})$  consisting of objects $\cH$
  whose cohomology $\coho(\cH)$ lies in $\HCglala$.
\end{definition}
Considering these bimodules as modules over the quantization
$\cD_{\la',-\la}$ of $\fM\times\fM$, we can apply the (derived)
localization and sections functors as in previous sections.

\begin{theorem}\label{bimodule-loc}
  For every $\la,\la'$, we have $\Rsecs(\DHCg{\la'}{\la})\subset \DHCa{\la'}{\la}$. If $A_\la$ and
  $A_{\la'}$ have finite
  global dimension\footnote{The finite global dimension hypothesis is truly necessary.  If $A_\la$ does not have finite global dimension, the
derived localization $\LLoc(A_\la)$ as a bimodule may not be bounded and
thus not in $\DHCg{\la}{\la}$.}, then $\LLoc(\DHCa{\la'}{\la})\subset \DHCg{\la'}{\la}$.
\end{theorem}

\begin{proof}
Let $\cH$ be an object in $\HCglala$, and let $\cH(0)\subset\cH$ be a lattice
satisfying the required condition.
For every $m$, we have a long exact  sequence showing that
$H^p(\fM;\cH(0)/\cH(-mn))$ is a extension of a submodule of
$H^p(\fM;\cH(0)/\cH(-(m-1)n))$ and quotient of
$$H^p(\fM;\cH(-(m-1)n)/\cH(-mn))\cong H^p(\fM;\cH(0)/\cH(-n)).$$   Thus, $H^p(\fM;\cH(0)/\cH(-mn))$ has an $m$ step filtration compatible with $\cH(0) \supset
  \cH(-n)\supset\cdots$ such that  elements of $A_{\la'}\otimes A^{\op}_{\la}$ of
  the form $\tilde f_{\la'}\otimes 1-1\otimes \tilde f_{\la}$ act
  trivially on the associated graded.   Since
  $H^p(\fM;\cH(0))=\varprojlim H^p(\fM;\cH(0)/\cH(-mn))$ by Lemma
  \ref{projlim}, we have an induced filtration on this group such that
  $\tilde f_{\la'}\otimes 1-1\otimes \tilde f_{\la}$ acts trivially
  modulo $h$.  This shows that the cohomology of $\R \Gamma^\bS(\cH)$ is Harish-Chandra as well. 

Now let $H$ be an object of $\HCalala$ and put $\cH:=\LLoc(H)$.
A filtration of $H$ induces a lattice in $\coho^p(\cH)$.
For any $\tilde f\in \Gamma(\fM;\scrD(0))$, we have that
  \[(\tilde f_{\la'}\otimes 1-1\otimes \tilde f_{\la})\cdot
  R(H)\subset h\cdot R( H);\]  thus, on any projective resolution, the
  map induced by 
  $(\tilde f_{\la'}\otimes 1-1\otimes \tilde f_{\la})$ is
  null-homotopic mod $h$; this
  implies that our lattice in $\coho^p(\cH)$ has the required property.
\end{proof}

\begin{corollary}
If derived localization holds at $\la'$ and $-\la$, then $\LLoc$ and $\Rsecs$
are inverse equivalences between $\DHCa{\la'}{\la}$ and $\DHCg{\la'}{\la}$.
If localization holds at $\la'$ and $-\la$, then $\Loc$ and $\secs$
are inverse equivalences between $\HCalala$ and $\HCglala$.
\end{corollary}
 
Consider the convolution product 
defined by the formula
\begin{equation}\label{convolution}
\cH_1\star\cH_2:=(p_{13})_*(p_{12}^{-1}\cH_1\overset{L}\otimes_{p_2^{-1}\cD_{\la'}}p_{23}^{-1}\cH_2),
\end{equation}
where $p_{ij}$ is one of the three projections from $\fM\times\fM\times\fM$ to $\fM\times\fM$.
  
\begin{proposition}\label{prop:HCg-tensor}
  If $\cM\in D^b(\cD_{\la''}{\hat\boxtimes}
 \cD_{\la'}^{\op}\mmod)$ and  $\cN\in
  D^b(\cD_{\la'}{\hat\boxtimes}
 \cD_{\la}^{\op}\mmod),$ then we have $\cM\star\cN\in D^b(\cD_{\la''}{\hat\boxtimes}
 \cD_{\la}^{\op}\mmod).$
If furthermore
$\cM\in \DHCg{\la''}{\la'},$ and $\cN\in
  \DHCg{\la'}{\la}$, then $\cM\star\cN\in \DHCg{\la''}{\la}$.
\end{proposition}

\begin{proof}
  The modules $\cM(0)$ and $\cN(0)$ have finite resolutions
  \[\cdots \to M^{(1)}_1\hat\boxtimes M^{(2)}_1\to
  M^{(1)}_0\hat\boxtimes M^{(2)}_0\to \cM(0) \and \cdots \to N^{(1)}_1\hat\boxtimes N^{(2)}_1\to
  N^{(1)}_0\hat\boxtimes N^{(2)}_0\to \cN(0)\] with $M_j^{(1)}$ (resp. $M_j^{(2)},N_j^{(1)},N_j^{(2)}$)
  locally free over $\cQ_{\la''}$ (resp. $\cQ_{\la'}^{\op},\cQ_{\la'},\cQ^{\op}_{\la}$), since the same is true of
  coherent sheaves over $\fS_{\fM\times \fM}$.  Thus, we can apply convolution to these
  modules by taking the naive tensor product over
  $p_2^{-1}\cD_{\la'}$:
\[\cM\star\cN(0):=M_\bullet^{(1)}\,\hat{\boxtimes} \,\mathbb{H}^\bullet\big(\fM;
M_\bullet^{(2)}\otimes_{\cD_{\la'}}N_\bullet^{(1)}\big) \,\hat\boxtimes\,
N_{\bullet}^{(2)},\]
where the middle term is considered as a complex of vector spaces,
which is of finite length since $\fM$ is finite dimensional.  This
shows that $\cM\star\cN$ is a bounded length complex. 

The argument that $\cM\star\cN$ lies in ${}_{\la''}\HCg{}_{\la}$ if
$\cM,\cN$ are Harish-Chandra is exactly as in
Proposition \ref{prop:HCa-tensor}.
The action of
$f_{\la}\otimes 1-1\otimes f_{\la'}$ on any resolution of $\cM(0)$ is
homotopic to 0 modulo $h$ for a global function $f$, as is the action of $f_{\la'}\otimes
1-1\otimes f_{\la''}$ on any resolution of $\cN(0)$.  Thus, tensoring these
homotopies gives one for $f_{\la}\otimes 1-1\otimes f_{\la''}$ on
$\cM\star\cN(0)$.  This function thus kills the cohomology of the classical limit
$\cM\star\cN(0)/h\cdot \cM\star\cN(0)$.
\end{proof}

\begin{proposition}\label{tensor equivalence}
Suppose that derived localization holds for $\la,\la',-\la'$, and $-\la''$.
The derived sections functor $\R\secs$ intertwines the convolution
of bimodules with derived tensor product.  That is, given Harish-Chandra bimodules
$\cH_1\in \DHCg{\la'}{\la}$ and $\cH_2\in \DHCg{\la''}{\la'}$, we have an isomorphism
$$\R\secs(\cH_1\star\cH_2)\cong \R\secs(\cH_1)\Lotimes\R\secs(\cH_2).$$
In particular, if $\la=\la'=\la''$ and derived localization holds for $\pm\la$, then the derived localization
and sections functors are inverse equivalences of tensor categories.
\end{proposition}

\begin{proof}
  The complex of modules $\R\secs(\cH_1)$ has a free resolution over
  $$A_{\la''}\otimes (A_{\la'})^{\op}=\secs(\cD_{\la''}\hat\boxtimes \cD_{\la'}^{\op})$$ of the form
  \begin{equation}
\cdots \to  A_{\la''}\otimes U_1\otimes A_{\la'}\to  A_{\la''}\otimes U_0\otimes A_{\la'}\to \cdots,\label{eq:4}
\end{equation}
and similarly $\R\secs(\cH_2)$ has a free resolution over
  $A_{\la'}\otimes A_\la^{\op}=\secs(\cD_{\la'}\hat\boxtimes \cD_\la^{\op})$ 
  \begin{equation}
\cdots \to  A_{\la'}\otimes V_1\otimes A_\la\to  A_{\la'}\otimes V_0\otimes
A_\la\to \cdots .\label{eq:3}
\end{equation}

Since derived localization holds, the sheaves $\cH_1$ and $\cH_2$ have resolutions \[\cdots \to  \cD_{\la''}\otimes
U_1\otimes \cD_{\la'}\to  \cD_{\la''}\otimes U_0\otimes \cD_{\la'}\to\cdots \]
\[\cdots \to  \cD_{\la'}\otimes V_1\otimes \cD_\la\to  \cD_{\la'}\otimes V_0\otimes
\cD_\la\to \cdots.\]
Thus, the convolution $\cH_1\star \cH_2$ is given by the complex
\begin{equation}
\cdots \to \bigoplus_{i+j=k+1}\cD_{\la''}\otimes
U_i\otimes A_{\la'}\otimes V_j\otimes \cD_{\la}\to \bigoplus_{i+j=k}\cD_{\la''}\otimes
U_i\otimes A_{\la'}\otimes V_j\otimes \cD_{\la}\to\cdots\label{eq:2}.
\end{equation}
The sections of \eqref{eq:2} is the complex
\begin{equation}
\cdots \to \bigoplus_{i+j=k+1}A_{\la''}\otimes
U_i\otimes A_{\la'}\otimes V_j\otimes A_{\la}\to \bigoplus_{i+j=k}A_{\la''}\otimes
U_i\otimes A_{\la'}\otimes V_j\otimes A_{\la}\to\cdots.
\end{equation} 
This is also the tensor product of the complexes \eqref{eq:4} and
\eqref{eq:3}, so this shows that the convolutions and tensor products agree.
\end{proof}

Following \Caldararu and Willerton \cite{CW}, we define a
2-category $\Quag$ where
\begin{itemize}
\item   objects are elements of $H^2(\fM;\C)$,
\item 1-morphisms from $\la$ to $\la'$ are objects of $\DHCg{\la'}{\la}$ with composition given by
  $\star$, and 
\item 2-morphisms are the usual morphisms in $\DHCg{\la'}{\la}$.
\end{itemize}
Similarly, we can define a 2-category $\Quaa$ whose objects are those $\la$
for which $A_\la$ has finite global dimension (we should consider only
these because of Proposition \ref{prop:HCa-tensor}) and whose
1-morphisms are objects of $\DHCa{\la'}{\la}$, with composition given by derived
tensor product.

Let $\Cat$ denote the 2-category of all categories, and consider the functors 
$$F^g\colon\Quag\to\Cat \and F^a\colon\Quaa\to\Cat$$
taking $\la$ to $D^b(\cD_\la\mmod)$ and $D^b(A_\la\mmod)$, respectively.
On 1-morphisms, $F^g$ takes an object $\cH$ to the functor given by convolution with $\cH$,
defined exactly as in Equation \eqref{convolution}.  Similarly, $F^a$ takes an object $H$
to the functor given by tensor product with $H$.

Let $\fL_0\subset\fM_0$ be an $\bS$-equivariant closed subscheme, and let $\fL\subset\fM$
be its scheme-theoretic preimage.
We would like to use $\fL_0$ and $\fL$ to define subcategories of $\Amod$ and $\Dmod$
in a way that is analogous to the definitions of algebraic and geometric Harish-Chandra bimodules
(Definitions \ref{hca} and \ref{hcg}).  In fact, those definitions will specialize to these when
$\fM$ is replaced by $\fM\times\fM$ and $\fL_0$ is the diagonal subscheme of $\fM_0\times\fM_0$.

\begin{definition}\label{def:Cla}
Let $C_\la^{\fL_0}$ be the full subcategory of $A_\la\mmod$ consisting of modules $N$ admitting
a filtration with thickened associated graded $\gr_n N$
scheme-theoretically supported on $\fL_0$.  Equivalently, we require
that if the symbol of $a_{\la}\in A_\la(k)$ vanishes on $\fL_0$, then $a_{\la}\cdot N(m)\subset N(k+m-n)$.
Let $\DCLa$ be the full subcategory of $D^b(A_\la\mmod)$ consisting of
objects with cohomology
in $\CLa$.
\end{definition}

\begin{definition}
Let $\mathcal{C}^\fL_\la$ be the full subcategory of $\cD_\la\mmod$ consisting of modules $\cN$
that have thick classical limits that are scheme-theoretically supported on $\fL$.
More precisely, we require a lattice $\cN(0)$ such that for any section $\tilde f$ of $\cQ$ whose reduction modulo $h$ lies in
the ideal sheaf of $\fL$, $\cN(0)$ is preserved by the action of $h^{-1}\tilde f$.
Let $\DCLg$ be the full subcategory of $D^b(\cD_\la\mmod)$ consisting of
objects with cohomology
in $\CLg$.
\end{definition}

Proposition \ref{prop:HCg-tensor}, along with an easy extension of the
proof of Proposition \ref{prop:HCa-tensor}, show that we have functors
\[F_\fL^g\colon\Quag\to\Cat \and F_\fL^a\colon\Quaa\to\Cat\]
taking $\la$ to $\DCLg$ and $\DCLa$, respectively.

\begin{example}
Suppose that $\fL_0\subset\fM$ is the unique $\bS$-fixed point; then $\fL = \nu^{-1}(0)$
is the core of $\fM$ (Remark \ref{core}), possibly with a non-reduced scheme structure.
If the weight $n$ of the symplectic form is equal to 1, then $\fL$ is Lagrangian, 
and $C^{\fL_0}_\la$ is the category of finite-dimensional $A_\la$-modules.
When $n$ is greater than 1, the core may be too small, in which case $C^{\fL_0}_\la$ will be zero.  
For example, if $\fM$ is the Hilbert scheme
of points on $\C^2$ and $\bS$ acts by scaling $\C^2$ (with $n=2$), then the core is the punctual Hilbert
scheme, which has dimension one less than half the dimension of $\fM$.
\end{example}

\begin{example}\label{catO}
Suppose that $\fM$ is equipped with a Hamiltonian action of $\bT := \C^\times$ that commutes
with the action of $\bS$ and has finite fixed point set $\fM^\bT$, and consider the Lagrangian subvariety
$$\fL_0 := \left\{p\in\fM_0\bigmid \displaystyle\lim_{t\to 0}t\cdot p\;\text{exists}\right\}.$$
In this case, $C^{\fL_0}_{\la}$ is the category of finitely generated $A_{\la}$-modules that are locally finite
for the action of $A_{\la}^+$, where $A_{\la}^+$ is the subring of $A_{\la}$ consisting of elements
with non-negative $\bT$-weight.  
This is an analogue of a block of BGG category $\cO$, and will be the primary
object of study in our forthcoming paper \cite{BLPWgco} with Licata.

To explain the connection with BGG category $\cO$, 
take $\fM = T^*(G/B)$ and let $\rho\in\Ht$ be half of the Euler class of the canonical bundle.
Then the ring $A_{\la+\rho}$ is a central quotient of $U(\mg)$, and
$C^\fL_{\la+\rho}$ is the category of finitely generated, $U(\mb)$-locally finite
$U(\mg)$-modules with same central character as the Verma module $V_\la$ with highest weight $\lambda$,
where $H^2(\fM;\C)$ is identified with the dual Cartan $\mathfrak{h}^*$ 
via the Chern class map.  When
$\la$ is a regular integral weight, this category is equivalent in a non-obvious way to the block $\cO_\la$
of BGG category $\cO$ by \cite[Th. 1]{Soe86}. 
\end{example}

\subsection{Characteristic cycles}\label{sec:index}
Let $\cD$ be a quantization of $\fM$, and let $\cN\in D^b(\Dmod)$ be
an object of the bounded derived category.
We have isomorphisms
 $$ \sHom^\bullet_{\cD}(\cN,\cN)\cong
  \sHom^\bullet_{\cD}(\cN,\cD)\Lotimes_{\cD}\cN \cong
  \cD_{\Delta}\Lotimes_{\cD\hat\boxtimes\cD^{\op}}\big(\cN\hat\boxtimes
  \sHom^\bullet_{\cD}(\cN,\cD)\big),$$
  and evaluation defines a canonical map to the Hochschild homology 
  $$
  \mathcal{HH}(\cD) := \cD_{\Delta}\Lotimes_{\cD\hat\boxtimes\cD^{\op}}\cD_{\Delta}.$$
All this is completely general, and holds in both the Zariski and the
classical topology.  In the classical topology, we also have an isomorphism  $\mathcal{HH}(\cD^{\an}) \cong
 \C_{\fM_\Delta}[\dim\fM]((h))$ by \cite[6.3.1]{KSdq}.  (This is a
local calculation, so it suffices to check for the Weyl algebra, where
it follows from a Koszul resolution.)

We define the {\bf characteristic cycle} 
$$\suppc(\cN)\in H^0(\mathcal{HH}(\cD^{\an}))\cong H^{\dim\fM}\big(\fM; \C((h)) \big)$$
to be the image of $\id\in H^0(\sHom^\bullet_{\cD}(\cN^{\an},\cN^{\an}))$ along this map.
More generally, if $\cN$ is supported on a subvariety $j\colon\fL\hookrightarrow\fM$,
then we may consider the identity map of $\cN^{\an}$ to be a section of 
$j^!\sHom^\bullet_{\cD}(\cN^{\an},\cN^{\an})$.
Applying our map then gives us a class in 
$$\suppc(\cN)\in H^0(j^!\mathcal{HH}(\cD^{\an}))\cong H^{\dim\fM}_\fL\big(\fM; \C((h)) \big).$$
Our abuse of notation is justified by the fact that this class is functorial
for inclusions of subvarieties.
If we replace the conical symplectic resolution $\fM$ with the product $\fM\times\fM$,
then this construction associates to a Harish-Chandra bimodule $\cH\in \DHCg{\la'}{\la}$
a class $\suppc(\cH)\in H^{2\dim\fM}_\fZ\big(\fM\times\fM; \C((h))\big)$.

Kashiwara and Schapira \cite[7.3.5]{KSdq} show that the characteristic cycle of a holonomic $\cD$-module 
(that is, one with Lagrangian support) may be computed in terms of its classical limit.

\begin{proposition}[Kashiwara and Schapira]\label{multiplicities}
If $\cN\in\Dmod$ is supported on a Lagrangian subvariety $\fL$ with components $\{\fL_i\}$, then
for any $\cD(0)$-lattice $\cN(0)\subset\cN$, 
\[\suppc(\cN)=\sum_{i}\rk_{\fL_i}\!\!\big(\cN(0)/\cN(-1)\big)\cdot
[\fL_i]\in H^{\dim\fM}_\fL(\fM ;\C)\subset  H^{\dim\fM}_\fL\big(\fM;\C((h)) \big),\]
where $\rk_{\fL_i}$ denotes the rank at the generic point of $\fL_i$.
\end{proposition}

We can also take characteristic cycles in families for modules over quantizations of twistor
deformations.   For $\eta\in\Ht$, let $\mathscr{M}_\eta\to \aone$ be the twistor deformation 
defined in Section \ref{sec:deformations} with quantization $\scrD$ extending $\cD$.  
Let $\scrN$ be a good $\scrD$-module, and consider the
image of the identity via the natural morphisms
\begin{multline}\label{relative-HH}
  \sHom^\bullet_{\scrD}(\scrN,\scrN)\cong
  \sHom^\bullet_{\scrD}(\scrN,\scrD)\Lotimes_{\scrD}\scrN \cong
  \scrD_{\Delta}\Lotimes_{\scrD\hat\boxtimes_{\aone}\scrD^{\op}}\big(\scrN\hat\boxtimes_{\aone}
  \sHom^\bullet_{\scrD}(\scrN,\scrD)\big)\\ \to
  \scrD_{\Delta}^{\an}\otimes_{\scrD ^{\an}\hat\boxtimes_{\aone}\scrD^{\an,\op}}\scrD_{\Delta}^{\an}\cong
  \pi^{-1}\fS_{\aone}[\dim\fM]((h)).
\end{multline}
This defines a class in relative 
cohomology $\suppc(\scrN)\in H^{\dim\fM}_{\mathscr{L}}(\mathscr{M}_\eta/\aone;\C((h)))$ for any 
Lagrangian $\mathscr{L}\supset
\supp(\scrN)$.  
If we let $\fL=\fM\cap \mathscr{L}$, then we have a
natural restriction map $$H^{\dim\fM}_{\mathscr{L}}(\mathscr{M}_\eta/\aone;\C((h)))\to
\HLCh$$ given by dividing by the coordinate $t$ on $\aone$. We
also have a natural functor of restriction from
$\scrD\mmod\to \cD\mmod$ given by $\scrN|_{\fM}=
\scrN\overset{L}\otimes_{\C[t]}\C$. 
The following lemma says that these operations are compatible.

\begin{lemma}\label{restriction-commutes}
  If $\scrN$ is a good $\scrD$-module,
  then $\suppc(\scrN|_{\fM})=\suppc(\scrN)|_{\fM}$. 
\end{lemma}
\begin{proof}
  Consider the complex \eqref{relative-HH} of $
  \pi^{-1}\fS_{\aone}$ modules, and
take the derived tensor product with $\C$ over $\C[t]$.
We claim that we obtain corresponding sequence for  $\scrN|_{\fM}$.
That is, we obtain \begin{multline}\label{absolute-HH}
  \sHom^\bullet_{\scrD}(\scrN|_{\fM},\scrN|_{\fM})\cong
  \sHom^\bullet_{\scrD}(\scrN|_{\fM},\cD)\Lotimes_{\cD}\scrN|_{\fM} \cong
  \cD_{\Delta}\Lotimes_{\cD\hat\boxtimes\cD^{\op}}\big(\scrN|_{\fM}\hat\boxtimes
  \sHom^\bullet_{\cD}(\scrN|_{\fM},\cD)\big)\\ \to
  \cD_{\Delta}^{\an}\otimes_{\cD ^{\an}\hat\boxtimes\cD^{\an,\op}}\cD_{\Delta}^{\an}\cong
  \C_\fM[\dim\fM]((h)).
\end{multline} It
suffices to prove this for $\scrN$ locally free.  In this case,
$\sHom^\bullet (\scrN,\scrD)$ is concentrated in degree 0 and is itself
locally free, so the statement is clear.

Thus $\suppc(\scrN)|_{\fM}$ can be obtained as the image of the identity
under the map \eqref{absolute-HH}.
By definition $\suppc(\scrN|_{\fM})$ is the image of the identity under \eqref{absolute-HH}, so we are done.
\end{proof}

Consider the category $K(\fZ)$ with objects $H^2(\fM;\C)$ and morphisms
$H^{2\dim\fM}_\fZ\big(\fM\times\fM; \C\big)$ between any two objects, with composition given by the convolution
structure defined at the beginning of this section.
We also have a category $K(\HCg)$ with objects $H^2(\fM;\C)$ and morphisms $K({_{\la'}^{\mbox{}}\HCg_{\la}})$ from $\la$ to $\la'$, with composition given by convolution;
this is simply the decategorification of the 2-category defined in the previous section.

\begin{proposition}\label{K-conv}
  The characteristic cycle map defines a functor $K(\HCg)\to K(\fZ)$. 
\end{proposition}
\begin{proof}
The fact that the characteristic cycle of a morphism in $K(\HCg)$ is an element of $H^{2\dim\fM}_\fZ(\fM \times \fM; \C)$
rather than $H^{2\dim\fM}_\fZ\big(\fM\times\fM; \C((h))\big)$ follows from Proposition \ref{multiplicities}.
Since the map $\fZ\times_{\fM}\fZ\to \fZ$ is proper, the rest of the
proposition follows from \cite[6.5.4]{KSdq} and the fact that the
functor $(-)^{\an}$ is monoidal and preserves Hom-spaces.
\end{proof}

Now fix a subvariety $\fL_0\subset\fM_0$, and let $\fL\subset\fM$ be its scheme-theoretic preimage as in Section \ref{sec:HC}.
We assume for convenience that $\fL$ is Lagrangian.
Consider the functor
$$G_\fL\colon K(\HCg)\to \mathsf{Ab}$$
taking:
\begin{itemize}
\item The class $\la$ to $K(\mathcal{C}^\fL_\la)$, the Grothendieck group of
  objects in $\mathcal{C}^\fL_\la$ with finitely generated
  cohomology concentrated in finitely many degrees.  Note that by its
  definition, $\mathcal{C}^\fL_\la$ may not be a Serre subcategory, in
  which case we consider the subgroup of the Grothendieck group of all
  holonomic $\cD$-modules generated by the objects in $\mathcal{C}^\fL_\la$.
\item  The class $[\cH]\in
  K({_{\la'}^{\mbox{}}\HCg_{\la}})$ to the convolution operator
$$[\cH]\star - : K(\mathcal{C}^\fL_\la)\to K(\mathcal{C}^\fL_{\la'})$$
defined by the formula $$[\cN]\star \a :=
(p_{13})_*(p_{12}^*[\cN]\cdot p_{23}^*\a).$$
\end{itemize}

We also have a functor 
$$H_\fL\colon K(\HCg)\to \mathsf{Ab}$$
taking every object $\la$ to
$H^{\dim\fM}_\fL(\fM; \C)$, where the map on morphisms is defined by the convolution action
of $H^{2\dim\fM}_\fZ(\fM\times\fM;\C)$ on $H^{\dim\fM}_\fL(\fM; \C)$.

\begin{proposition}\label{K-act}
  The characteristic cycle map $$\label{HCCC}\suppc\colon K(\mathcal{C}^\fL_\la)\to
  H^{\dim\fM}_\fL(\fM; \C)$$
  defines a natural transformation from $G_\fL$ to $H_\fL$.  That
  is, for all $\cH\in \DHCg{\la'}{\la}$ and $\cN\in D^b(\mathcal{C}^\fL_\la)$, 
  $$\suppc(\cH)\star\suppc(\cN) = \suppc(\cH\star\cN).$$
\end{proposition}

\begin{proof}
 Since the map $\fZ\times_{\fM}\fZ\to \fL$ is proper, this 
follows immediately from \cite[6.5.4]{KSdq}.
\end{proof}

Thus, these bimodules provide a natural categorification of the
convolution algebra of a symplectic singularity, and at least certain
of its natural convolution modules.  Of course, the 
characteristic cycle maps need not be isomorphisms, but in many contexts, they are.

\begin{example}
  In the case where $\fM=T^*(G/B)$, the category
  ${}_{\la}^{\mbox{}}\HCg_{\la}$ is equivalent to the category of regular
  twisted D-modules on $G/B\times G/B$ for the twist
  $(\la+\rho,-\la+\rho)$ which are smooth on diagonal $G$-orbits; as
  long as $\la+\rho$ is integral, this is the same as the category of
  perverse sheaves smooth along the same stratification.  The fact
  that these categorify the symmetric group (and thus, implicitly,
  that $\suppc$ is an isomorphism in this case) goes back at least as
  far as \cite{SpIH}.  This perspective is
  Koszul dual to the usual categorification of the Weyl group by
  projective functors \cite[5.16]{BG}.
\end{example}

\begin{example}
In the case where $\fM$ is a hypertoric variety, the map from $K({}_{\la}^{\mbox{}}\HCg_{\la})$ to
$H^{2\dim\fM}_\fZ(\fM\times \fM;\C)$ is surjective by \cite[7.11]{BLPWtorico}, which allows us to conclude
that every irreducible representation of the convolution algebra remains
irreducible over $K(\HCg)$.  The dimensions of these
representations are computed in \cite{PW07} to be $h$-numbers of various matroids.
\end{example}

\begin{example}
  In the case of Nakajima quiver varieties, it is more natural to
  consider all quiver varieties associated to a highest weight $\mu$
  jointly, and thus define a 2-subcategory $\Qua(\mu)$ of modules over
  the exterior products of quantizations of quiver varieties
  associated to $\la$ and possibly different dimension vectors.

  However, even with different dimension vectors, we still have a
  notion of ``diagonal'' in the product of two quiver varieties with
  the same highest weight.   The affinization of a quiver variety is the
  moduli space of semi-simple representations of the pre-projective
  algebra of a given dimension, and we say a pair of such representations
  lies in the {\bf stable diagonal} if they become isomorphic after
  the addition of trivial representations.  We can define a 2-categories
  $\HCg(\mu)$ by replacing the diagonal and its vanishing
  ideal with that of the stable diagonal.  The third author \cite[Theorem A]{Webcatq} relates this construction to
  works by Cautis and Lauda \cite{CaLa} and Nakajima \cite{Nak98}.

\begin{proposition}[Webster]
 There is a 2-functor from the version of the 2-quantum group
    $\mathcal{U}$ defined by Cautis and Lauda to
    $\HCg(\mu)$ with the property that the induced map of K-groups is exactly the geometric
    construction of $U(\mg)$ defined by Nakajima.
  \end{proposition}
\end{example}

\subsection{Twisting bimodules}\label{twisting bimodules}
For the rest of this paper, we will assume that the Picard group of
$\fM$ is torsion-free, so that a line bundle is determined by its Euler class in $\Ht$.
This assumption is not strictly necessary, but it greatly simplifies the notation (see Remark \ref{notation}).

Consider the universal Poisson deformation $\scrM$ of $\fM$.
Let $\mathscr{L}$ be a line bundle on $\scrM$, let $\cL$ be its restriction to $\fM$, and let 
$\gamma\in H^2(\fM;\Z)\cong H^2(\scrM;\Z)$ be the Euler class of $\mathscr{L}$ or $\cL$.
Let ${}_\gamma\mathscr{T}_0$ be the quantization of $\mathscr{L}$ constructed in Proposition
\ref{line bundles}, and let ${}_\gamma\mathscr{T}'_0 := {}_\gamma\mathscr{T}_0[\hmon]$. This is a right $\scrD$-module and a left module over $\scrD_{\gamma}$, the
quantization with period $I+h\gamma$.\footnote{We note that all quantizations of $\scrM$ are isomorphic 
as sheaves of algebras, but they are not isomorphic as sheaves of $\pi^{-1}\fS_{\Ht}$-algebras.}
Then $\secs(\scrM;{_\gamma\mathscr{T}'_0})$ is a family over $H^2(\fM;\C)$ via the
right action of $\scr{A} = \secs(\scrM; \scr{D})$. 

Recall the map $c\colon \C[H^2(\fM;\C)]\to \Gamma(\scrM;\scrD)$ from Section \ref{section ring},
and the fact that $h^{-1}c(x)\in\scrA$ for all $x\in\Ht^*$.  Also recall that, by Proposition \ref{restriction of sections},
the specialization of $\scrA$ at $h^{-1}c(x)=\la(x)$ for all $x\in\Ht^*$ is isomorphic to $A_\la$.

\begin{definition}\label{T-def}
  Let ${}_{\la+\gamma} T_{\la}$ denote the $A_{\la+\gamma}-A_{\la}$ bimodule that we obtain by specializing
  $\secs(\scrM;{_\gamma\mathscr{T}'_0})$ at $h^{-1}c(x)=\la(x)$
  for all $x\in \Ht^*$.
\end{definition}

\begin{remark}\label{notation}
The purpose of the assumption at the beginning of this section was to ensure that the bimodule ${}_{\la+\gamma} T_{\la}$
is actually determined by $\la$ and $\gamma$; without the assumption, the bimodule would depend on an additional
choice of a line bundle with Euler class $\gamma$.
\end{remark}

\begin{proposition}
The bimodule ${_{\la+\gamma}T_{\la}}$ is Harish-Chandra.
\end{proposition}

\begin{proof}
By definition, ${_{\la+\gamma}T_{\la}}$ is a specialization of $\secs(\scrM;{_\gamma\mathscr{T}'_0})$.
It carries a natural filtration, where ${_{\la+\gamma}T_{\la}}(m)$ is the same specialization of 
$\secs(\scrM;h^{\nicefrac{-m}{n}}{_\gamma\mathscr{T}_0}[\hon])$.  
We claim that the associated graded module with respect
to this filtration is scheme-theoretically supported on the diagonal.

To see this, consider a function $f\in\C[\fM]$ of $\bS$-weight $\ell$.  
We can choose a lift $\tilde f \in \secs(\scrD(\ell))$ so that its image 
in $\gr \secs(\scrD) \cong \C[\scr M]$ restricts to $f$ on $\fM$.
Let $\scrD_\gamma$ be the quantization of $\scrM$ with period $\gamma$;
since $\gr \secs(\scrD_\gamma) \cong \gr \secs(\scrD)$, we can choose a lift
$\tilde f_\gamma \in \secs(\scrD_\gamma(\ell))$ of $f$ similarly.
To show that $f\otimes 1 - 1\otimes f$
annihilates $\gr_n({_{\la+\gamma}T_{\la}})$, it is sufficient to show that $\tilde f_\gamma\otimes 1 - 1\otimes \tilde f$
takes  $\secs(\scrM;h^{\nicefrac{-m}{n}}{_\gamma\mathscr{T}_0}[\hon])$ to 
$\secs(\scrM;h^{1-\nicefrac{\ell+m}{n}}{_\gamma\mathscr{T}_0}[\hon])$.
This follows from the fact that ${_\gamma\mathscr{T}_0}$ is the
quantization of a line bundle on $\scrM$, so the left action of 
$\tilde f_\gamma$ and the right action of $\tilde f$ agree modulo $h$.
\end{proof}

The following two propositions are bimodule analogues of 
Corollary \ref{same-sections} and Proposition \ref{Weyl rings}.
Since their proofs are essentially identical, we omit them.
 
\begin{proposition}\label{same-sections-bimodules}
Let $\fM$ and $\fM'$ be two conical symplectic resolutions of the same cone.
Fix elements $\la,\gamma\in\Ht\cong H^2(\fM';\C)$, 
where $\gamma$ is the Euler class of a line bundle on $\fM$ or its strict transform on $\fM'$.
The isomorphism of rings in Corollary \ref{same-sections} induces an isomorphism of bimodules
${}_{\la+\gamma}T_{\la}\cong{}_{\la+\gamma}T'_{\la}$.
\end{proposition}

\begin{proposition}\label{Weyl bimodules}
For any $\la,\gamma\in\Ht\cong H^2(\fM';\C)$, 
where $\gamma$ is the Euler class of a line bundle on $\fM$, and any $w\in W$,
the isomorphisms of Proposition \ref{Weyl rings} induce isomorphisms of bimodules
${}_{\la+\gamma}T_{\la}\cong{}_{w\cdot(\la+\gamma)}T_{w\cdot\la}$.
\end{proposition}

We would like to have an analogue of Proposition \ref{restriction of sections}, as well, though
an extra hypothesis is needed.
The following proposition gives a natural map from ${}_{\la+\gamma} T_{\la}$ to
$\secs(\fM;{}_{\la+\gamma}\cT_{\la}')$, and gives a sufficient (though not necessary) condition for it to be an isomorphism.  (Note that it is always injective.)

\begin{proposition}\label{bi-sections}
There is a natural map from the bimodule ${}_{\la+\gamma} T_{\la}$ to
  $\secs(\fM;{}_{\la+\gamma}\cT'_{\la})$.
If $H^1(\fM;{}_{\la+\gamma} \cT'_{\la})=0$, then this map is an isomorphism.
\end{proposition}

\begin{proof}
The pullback of ${}_{\gamma}\scr{T}'_{0}$ along the map $\Delta \to H^2(\fM;\C)\times \Delta$ given by $h \mapsto (h\la, h)$ is a quantization of $\cL$.  By the uniqueness of the quantized line bundles constructed in Proposition \ref{line bundles}, this pullback is isomorphic to ${}_{\la+\gamma}\cT'_{\la}$.  Since ${}_{\la+\gamma} T_{\la}$
is obtained from ${}_{\gamma}\scr{T}'_{0}$ by first taking sections and then specializing, this defines the required map.

Now suppose that  $H^1(\fM;{}_{\la+\gamma} \cT'_{\la})=0$.  To prove that our map is 
surjective, we factor the pullback into two steps.
 Choose $\nu\in H^2(\fM;\C)$ with
  $\scrM_\nu(\infty)$ affine.  Let
  ${}_{\la+\gamma}{\scr{T}_{\la}'}^{(\nu)}$ be the bimodule on $\scrM_\nu$ obtained by pulling ${}_{\gamma}\scr{T}'_{0}$ back along the map $\aone\times
  \Delta\to H^2(\fM;\C)\times \Delta$ taking $(t,h)$ to $(t\nu+h\la, h)$.  
  Thus ${}_{\la+\gamma}\cT'_{\la}$ is obtained from this sheaf by pulling back further by the map $\Delta \to \aone\times \Delta$ given by $h\mapsto (0, h)$.
  Let 
  ${}_{\la+\gamma}O_{\la}:=\secs(\scrM_\nu;{}_{\la+\gamma}{\scr{T}_{\la}'}^{(\nu)})$.
  To show that our map is surjective, it will suffice to show that 
  \begin{enumerate}
  \item the map from
    $\secs(\scrM;{}_{\gamma}\scr{T}_{0}')$ to
    ${}_{\la+\gamma}O_{\la}$ is surjective, and 
\item the map
    from ${}_{\la+\gamma}O_{\la}$ to
    $\secs(\fM;{}_{\la+\gamma}\cT_{\la}')$ is surjective.
  \end{enumerate}

Consider the variety $\scr{N} := \Spec\C[\scrM]$ from Section \ref{sec:Weyl-gp-def},
along with the related variety $\scr{N}_\nu := \Spec\C[\scrM_\nu]\subset \scr{N}$.
Let $\scr{N}^{\operatorname{sm}}$ and $\scr{N}_\nu^{\operatorname{sm}}$ be their smooth loci;
since the affinization maps for $\scrM$ and $\scrM_\nu$ are isomorphisms over the smooth loci,
we may regard $\scr{N}^{\operatorname{sm}}$ as a subvariety of $\scrM$ and
$\scr{N}_\nu^{\operatorname{sm}}$ as a subvariety of $\scrM_\nu$.

Let ${}_{\gamma}\scr{S}_{0}'$ be the sheaf on $\scr{N}$ obtained from ${}_{\gamma}\scr{T}_{0}'$
by first restricting it to $\scrN^{\operatorname{sm}}$ and then pushing it forward
to $\scrN$; since the complement of $\scrN^{\operatorname{sm}}$ in $\scrM$
has codimension at least 2, we have 
$$\secs(\scrM;{}_{\gamma}\scr{T}_{0}')
\cong \secs(\scrN;{}_{\gamma}\scr{S}_{0}').$$
Similarly, we define a sheaf ${}_{\la+\gamma}{\scr{S}_{\la}'}^{(\nu)}$ on  $\scrN_\nu$ 
obtained from ${}_{\la+\gamma}{\scr{T}_{\la}'}^{(\nu)}$
by first restricting it to $\scrN_\nu^{\operatorname{sm}}$ and then pushing it forward
to $\scrN_\nu$, and we have
$${}_{\la+\gamma}O_{\la}
\cong \secs(\scrN_\nu;{}_{\la+\gamma}{\scr{S}_{\la}'}^{(\nu)}).$$
To see that the map from $\secs(\scrM;{}_{\gamma}\scr{T}_{0}')$ to
${}_{\la+\gamma}O_{\la}$ is surjective, it suffices to check that the associated graded is surjective.
When we pass to the associated graded, we obtain a map between spaces of sections
of two coherent sheaves on $\scrN$, namely the classical limits
$\overline{{}_{\gamma}\scr{S}'_{0}}$ and
$\overline{{}_{\la+\gamma}{\scr{S}_{\la}'}^{(\nu)}}$. By definition,
the restriction of
$\overline{{}_{\la+\gamma}{\scr{S}_{\la}'}^{(\nu)}}$ to
$\scrN^{\operatorname{sm}}_\nu$ is a quotient of the restriction of
$\overline{{}_{\gamma}\scr{S}'_{0}}$ to $\scrN^{\operatorname{sm}}$.
Since the singular locus has codimension 3 on both $\scrN$ and $\scrN_\nu$, the induced map between pushforward sheaves is surjective, and
since $\scrN$ is affine, the same is true of the sections.

We now turn to the second surjectivity statement.  Consider the exact sequence
\[0\longrightarrow{}_{\la+\gamma}{\scr{T}_{\la}'}^{(\nu)}\overset{h^{-1}t}\longrightarrow
{}_{\la+\gamma}{\scr{T}_{\la}'}^{(\nu)}\longrightarrow
{}_{\la+\gamma} \cT_{\la}'\longrightarrow 0\]
of sheaves on $\scrM_\nu$ and its associated long exact sequence
\begin{multline*}
  0 \longrightarrow {}_{\la+\gamma}O_{\la}\overset{h^{-1}t}\longrightarrow {}_{\la+\gamma}O_{\la}\longrightarrow
  \secs(\fM;{}_{\la+\gamma}\cT_{\la}')
  \longrightarrow H^1(
  \scrM_\nu;{}_{\la+\gamma}{\scr{T}_{\la}'}^{(\nu)})^\bS\\
  \overset{h^{-1}t}
\longrightarrow H^1(
  \scrM_\nu;{}_{\la+\gamma}{\scr{T}_{\la}'}^{(\nu)})^\bS\longrightarrow
  H^1(\fM;{}_{\la+\gamma} \cT_{\la}')^\bS\longrightarrow\cdots .
\end{multline*}
The surjectivity statement that we need is equivalent (by exactness) to injectivity of the action of $h^{-1}t$
on $H^1(\scrM_\nu;{}_{\la+\gamma}{\scr{T}_{\la}'}^{(\nu)})^\bS$.

Since the generic fiber of $\scrM_\nu$ is affine, 
$H^1(\scrM_\nu;{}_{\la+\gamma}{\scr{T}_{\la}'}^{(\nu)})^\bS$
is supported on the fiber over 0.  This bimodule is Harish-Chandra, so its
localization has Lagrangian support in $\fM \times \fM$.  Applying Lemma
\ref{minimal-polynomial}, we see that $h^{-1}t$ satisfies a polynomial
equation on $H^1(\scr{M}_\nu; {}_{\la+\gamma}{\scr{T}_{\la}'}^{(\nu)})^\bS$, so the bimodule
is the sum of finitely many generalized eigenspaces for $h^{-1}t$
and $h^{-1}t$ acts with finite length.  In particular, if $0$ is a root of
this minimal polynomial, the map $h^{-1}t$ is not surjective (since its
stable image is a proper summand), and thus $H^1(\fM; {}_{\la+\gamma} \cT_{\la})$ 
is not 0.  This is impossible by assumption, so 0 cannot
be a root.  Thus, $h^{-1}t$ does act invertibly, so the desired map is
surjective.
\end{proof}

The following proposition says that derived tensor product with a twisting bimodule
does not change the characteristic cycle of the localization.  Let $N$ be an object of
$D^b(A_\la\mmod)$, so that $\LLoc({}_{\la+\gamma}T_{\la}\overset{L}{\otimes} N)$ is an object of $D^b(\cD_{\la+\gamma}\mmod)$.

\begin{proposition}\label{same-cycle}
Assume derived localization holds at $\la$ and $\la+\gamma$.  Then we have that
  \[\suppc(\LLoc(N))=\suppc(\LLoc({}_{\la+\gamma}T_{\la}\overset{L}{\otimes} N)).\]  
\end{proposition}

\begin{proof}
As in the proof of Proposition \ref{bi-sections}, choose $\nu\in\Ht$ such that $\scrM_\nu(\infty)$ is affine, and
  consider the sheaf ${}_{\la+\gamma}{\scr{T}_{\la}'}^{(\nu)}$. At any
  point $p$ of $\aone$, the derived functor of base change to the fiber
  $\pi^{-1}(p\nu)$ over
  $p$ sends ${}_{\la+\gamma}{\scr{T}_{\la}'}^{(\nu)}$ to the derived
  localization $\LLoc({}_{\la+\gamma}O_\la/(t-p))$ as a module over a
  quantization of $\scrM_\nu \times \scrM_\nu$, since the module
  ${}_{\la+\gamma}O_\la$ is flat over $\aone$.

  If $p$ is not 0, then the fiber is affine, and
  $\LLoc({}_{\la+\gamma}O_\la/(t-p))$ is a line bundle on the diagonal
  in $\pi^{-1}(p\nu)\times \pi^{-1}(p\nu)$.  In particular, the class
  $\suppc({}_{\la+\gamma}{\scr{T}_{\la}'}^{(\nu)})$ thus must be the
  class of the diagonal over every non-zero point in $\aone$.  By
  Lemma \ref{restriction-commutes}, we thus have that \[\suppc(\LLoc({}_{\la+\gamma}T_{\la}))=\suppc(
  {}_{\la+\gamma}{\scr{T}_{\la}'}^{(\nu)}|_{\pi^{-1}(0)})=[\fM_{\Delta}].\]
By Proposition \ref{K-conv}, the characteristic cycle map intertwines derived tensor product with convolution.
Since convolution with the diagonal is trivial, this implies the desired equality. 
\end{proof}

We conclude this section by computing these bimodules explicitly in
the case where $\fM$ is a
symplectic quotient of a vector space, as in Example \ref{quotient construction}.
Let $G$ be a connected reductive algebraic group acting on a vector space $V$
with flat moment map $\mu:T^*V\to\mg^*$;
let $\fM$ be the symplectic quotient of $T^*V$ at a generic character $\theta$ of $G$,
and suppose that the Kirwan map $\mathsf{K}:\chi(\mg)\to \Ht$ is an isomorphism.
Let $A_{T^*V}$ be the section
ring of the unique quantization of $T^*V$; this is isomorphic to the ring of differential operators on $V$.
Fix a quantized moment map $\eta:U(\mg)\to A_{T^*V}$ and an element $\xi\in\chi(\mg)$, and let
$\cD_\xi$ be the associated quantization of $\fM$ (Section \ref{QHr}) with section ring $A = \secs(\cD)$.  
By Proposition \ref{quotient}, we have $A\cong \End_{A_{T^*V}}(Y_\xi)$.

Fix a second character $\xi'$ such that $\xi'-\xi$ integrates to a character of $G$, and consider the
$A'-A$ bimodule 
\begin{equation}\label{bimodule}
\Hom(Y_{\xi'},Y_{\xi})
  \cong\Big(A_{T^*V}\big{/}A_{T^*V}\cdot \langle \eta(x) - \xi(x)\mid
  x\in\mg\rangle\Big)^{\xi'-\xi}.
\end{equation}
By Proposition \ref{S equals T}, we have a natural map from
  $\Hom(Y_{\xi'},Y_{\xi})$ to
  ${_{\mathsf{K}(\xi')}T_{\mathsf{K}(\xi)}}$.  
  This map is is always injective but it need not be an isomorphism; the restriction to
the semistable locus can can cause new sections to appear. 

\begin{lemma}\label{twist-reduce}
If $\xi' = \xi + m\theta$ for $m\gg 0$, then the map from $\Hom(Y_{\xi'},Y_{\xi})$ to
${_{\mathsf{K}(\xi')}T_{\mathsf{K}(\xi)}}$ is an isomorphism.
\end{lemma}

\begin{proof}
The associated graded of this map is the natural map from $\C[\mu^{-1}(0)]_{m\theta}$
to $\Gamma(\fM,\cL_{m\theta})$, where the subscript in the source indicated the $\bS$-weight space.
This map is an isomorphism for sufficiently large $m$, thus so is our original map.
\end{proof}

\begin{remark}\label{theta doesn't matter}
We note that, by Corollary \ref{same-sections} and Proposition \ref{same-sections-bimodules},
the source and target of the map in Lemma \ref{twist-reduce} (along with the map itself)
are independent of the choice of conical symplectic resolution.
Thus Lemma \ref{twist-reduce}
simply says that our map is an isomorphism when $\xi$ and $\xi'$ are sufficiently far apart in any generic direction.
\end{remark}

\subsection{Twisting functors}\label{twisting}
By Theorem \ref{dream}, the set $I$ of isomorphism classes of conical symplectic resolutions of $\fM_0$
is finite.  For each $i\in I$, let $\fM_i$ be a representative resolution.  By Remark \ref{chambers},
the chambers of the hyperplane arrangement $\cH$ are in canonical bijection with $I\times W$, where $W$
is the Weyl group from Section \ref{sec:Weyl-gp-def}.  For each pair $(i,w)$, let $\Pi_{i,w}\subset P_\R$
be the set of parameters $\la$ in the corresponding chamber of $\cH$ with the additional property
that localization holds at $w\la$ on $\fM_i$ and derived localization
holds at $w'\la$ and $-w'\la$ on $\fM_{i'}$ for all pairs $(i',w')$. 
Let $$\Pi := \displaystyle\bigcup_{I\times W} \Pi_{i,w}\subset P_\R.$$

\begin{lemma}\label{finitely-many}
If $w\eta$ is an ample class on $\fM_i$, then for any $\la$, the class
$\la+k\eta$ lies in $\Pi_{i,w}$ for all but finitely many $k\in \Z_{\geq 0}$.
\end{lemma}

\begin{proof}  Recall from Remark \ref{chambers} that the chamber of $\cH$ indexed by $(i,w)$
is equal to the $w$ translate of the ample cone of $\fM_i$.
Since $w\eta$ is ample on $\fM_i$, so is $w(\la+k\eta)$ when $k$ is sufficiently large.
The fact that localization holds at $w(\la+k\eta)$ for large $k$
follows from Corollary \ref{large quantizations}, and  Theorem
\ref{derived-local} shows the required derived localization
statements. The fact that there are only finitely many elements of $I$
shows that only finitely many $k$ need to be removed.
\end{proof}

Let $A_\la$ be the invariant section ring of the quantization with period $\la$.
(Note that, by Corollary \ref{same-sections}, the ring $A_\la$ does not depend on the choice of resolution of $\fM_0$.)
For any pair of elements $\la,\la'\in H^2(\fM;\C)$ that differ by an integral class,
let 
\begin{equation}\label{pure twist}
\Phi^{\la'\!,\la}\colon D(A_{\la}\mMod)\to D(A_{\la'}\mMod)
\end{equation}
be the functor obtained by derived tensor product with the bimodule $\bi$.
For any $\la\in\Pi$ and $w\in W$, let
\begin{equation}\label{impure twist}
\Phi^{\la}_w\colon D(A_{w\la}\mMod)\to D(A_{\la}\mMod)
\end{equation}
be the equivalence obtained from the isomorphism of Proposition \ref{Weyl rings}.
Note that the compatibility in the statement of Proposition \ref{Weyl rings} implies that
the composition $\Phi^{w\la}_{w^{-1}}\circ\Phi^{\la}_w$ is naturally isomorphic to the identity functor.

\begin{proposition}\label{alg-geom}
Suppose that $\la'\in\Pi_{i,w}$.  Then the functor $\Phi^{\la'\!,\la}$ is naturally isomorphic to the composition
\begin{multline*}
  D(A_\la\mMod) \overset{\Phi^{\la}_w}{\xrightarrow{\hspace*{1cm}}}
  D(A_{w\la}\mMod)\overset{\LLoc_i}{\xrightarrow{\hspace*{1cm}}}
  D(\cD_{w\la}\mMod)\\\overset{{{_{w\la'}\!\cT'_{w\la}}}\otimes -}{\xrightarrow{\hspace*{1.5cm}}}
  D(\cD_{w\la'}\mMod)\overset{\mathbb{R}\Gamma_{\bS,i}}{\xrightarrow{\hspace*{1cm}}}
  D(A_{w\la'}\mMod) \overset{\Phi^{w\la'}_{w^{-1}}}{\xrightarrow{\hspace*{1cm}}} 
  D(A_{\la'}\mMod),
\end{multline*}
where the subscript $i$ on $\Rsecs$ and $\LLoc$ refers to the fact
that we are using the resolution $\fM_i$.  
\end{proposition}

\begin{proof}
Since $\la'\in\Pi_{i,w}$, localization holds at $w\la'$, which implies that the higher cohomology of
${}_{w\la'}\cT'_{w\la}$ is trivial.  Then Proposition \ref{bi-sections} tells us that
${}_{w\la'}T_{w\la}\cong\mathbb{R}\Gamma_{\bS,i}({}_{w\la'}\cT'_{w\la})$, 
and therefore that
${}_{w\la'}\cT'_{w\la}\cong\LLoc_i({}_{w\la'}T_{w\la})$.
The proposition follows immediately using Proposition \ref{Weyl bimodules}.
\end{proof}

\begin{corollary}
For all $\la\in \Pi$
the functor $\Phi^{\la'\!,\la}$ induces a functor $D^b(A_\la\mmod)\to
D^b(A_{\la'}\mmod)$.  If $\la'\in\Pi$ as well, then this functor is an
equivalence.   
\end{corollary}

\begin{proof}
The functor $\LLoc_i$ induces an equivalence $D^b(A_{w\la}\mmod)
  \to D^b(\cD_{w\la}\mmod)$ as discussed in Remark \ref{rem:bounded-enough}.  The functor
  ${{_{w\la'}\!\cT'_{w\la}}}\otimes -$ is an equivalence of abelian
  categories with inverse ${{_{w\la}\!\cT'_{w\la'}}}\otimes -$
by the uniqueness part of Proposition \ref{line bundles}.  The functor
$\mathbb{R}\Gamma_{\bS,i} $ induces a functor
$D^b(\cD_{w\la'}\mmod)\to 
  D^b(A_{w\la'}\mmod)$ by Proposition \ref{prop:Rsecs-bounded}, which
  is also an equivalence if  $\la'\in\Pi$.
\end{proof}

\begin{corollary}\label{within}
If $\la$ and $\la'$ lie in the same chamber of $\cH$, then $\Phi^{\la,\la'}\circ\Phi^{\la'\!,\la}$
is naturally isomorphic to the identity functor.
\end{corollary}

\begin{proof}
This follows similarly from Propositions \ref{line bundles} and \ref{alg-geom}.
\end{proof}

Fixing a particular $\la\in\Pi$,
we define {\bf twisting functors} to be the group of endofunctors of $D(A_{\la}\mMod)$ 
(or of the full subcategory $D^b(A_{\la}\mmod)$)
obtained by composing functors of 
the form \eqref{pure twist} and \eqref{impure twist} and their inverses,
and we define {\bf pure twisting functors}
to be the subgroup obtained using only functors of the form \eqref{pure twist} and their inverses.
Note that Corollary \ref{within}
implies that any such composition that never leaves the chamber in which $\la$ lives is trivial.  However,
when one crosses a wall and then crosses back, one can and does obtain something nontrivial (see Proposition
\ref{shuf-twist} for the case of the Springer resolution).

For Lemma \ref{straight-line} we adopt the notational convention, introduced in Section \ref{sec:derived}, 
whereby we fix $\eta\in H^2(\fM;\Z)$ and $\la\in\Ht$ and use $k$ in a subscript or superscript in place of
$\la+k\eta$.

\begin{lemma}\label{straight-line}
Suppose that $w\eta$ is very ample on $\fM_i$.
Then for any natural numbers $k_{\ell}> k_{\ell-1}>\cdots>k_1\geq k_0$,
there is a natural isomorphism of functors
\[
\Phi^{k_\ell,k_{0}}\simeq
\Phi^{k_\ell,k_{\ell-1}}\circ \cdots \circ \Phi^{k_1,k_{0}}.\]
\end{lemma}

\begin{proof}
Let $\cL$ be the line bundle on $\fM_i$ with Euler class $\eta$.  For any $k'>k$,
the higher cohomology of $\cL^{k'-k}$ vanishes. Therefore the higher
cohomology of ${_{k'}\cT'_k}$ vanishes as well.
By the same argument that we used in the proof of Proposition \ref{alg-geom}, 
Proposition \ref{bi-sections} tells us that
\begin{multline*} {}_{k_{\ell}}T_{k_0}\cong
  \Rsecs({}_{k_{\ell}}\cT_{k_0}) \cong
  \Rsecs({}_{k_\ell}\cT_{k_{\ell-1}}{\otimes}_{\cD_{k_{\ell-1}}} \cdots
  {\otimes}_{\cD_{k_{1}}} {}_{k_{1}}\cT_{k_0})\\
  \cong\Rsecs({}_{k_\ell}\cT_{k_{\ell-1}})\overset{L}{\otimes}_{A_{k_{\ell-1}}}\cdots
  \overset{L}{\otimes}_{A_{k_{1}}}\Rsecs( {}_{k_{1}}\cT_{k_0})
\cong
  {}_{k_\ell}T_{k_{\ell-1}}\overset{L}{\otimes}_{A_{k_{\ell-1}}}
  \cdots \overset{L}{\otimes}_{A_{k_{1}}} {}_{k_{1}}T_{k_0}
\end{multline*}
as desired.  Since $\Phi^{k_{\ell},k_0} = {_{k_\ell}}
T_{k_0}\overset{L}{\otimes}_{A_{k_0}}\!-$, the isomorphism follows.
\end{proof}

Let $$E := \Ht\smallsetminus\displaystyle\bigcup_{H\in\cH}H_\C$$ be the complement of the complexification of $\cH$.
The main theorem of this section says that the fundamental group of $E/W$ acts on our category by twisting functors.

\begin{theorem}\label{twisting braid}
For any $\la\in\Pi$, there is a natural homomorphism from $\pi_1(E/W, [\la])$ 
to the group of twisting functors on $D(A_\la\mMod)$.  The subgroup $\pi_1(E, \la)$
maps to the group of pure twisting functors.
\end{theorem}

\begin{proof}
For each element $(i,w)\in I\times W$, choose an integral class $\eta_{i,w}$ such that $w\eta_{i,w}$ is ample on $\fM_i$.
By Lemma \ref{finitely-many}, we may choose a natural
number $k_{i,w}$ such that $\la_{i,w} := \la + k_{i,w}\eta_{i,w}$ lies in $\Pi_{i,w}$.
The {\bf Deligne groupoid} of $\cH$ is the full sub-groupoid of the fundamental
groupoid of $E$ with objects $\{\la_{i,w}\mid(i,w)\in I\times W\}$.  Note that different choices would lead
to a canonically isomorphic groupoid; the only important thing is that we have chosen one representative
of each chamber.

The {\bf Deligne quiver} of a real hyperplane arrangement is the quiver with nodes indexed by chambers 
and arrows in both directions between any two adjacent chambers.
Paris \cite{Paris} proves that the Deligne groupoid is isomorphic
to the quotient of the fundamental groupoid of the Deligne quiver obtained by identifying any pair of positive paths
of minimal length between the same two nodes.\footnote{A path can travel forward or backward along arrows;
a positive path is one that always travels forward.}
Thus, to construct an action of the Deligne groupoid, it is sufficient
to first define an action of the Deligne quiver and then check Paris's relations.

Recall that the chambers of $\cH$ are in bijection with $I\times W$.
We begin by associating the category $D(A_{\la_{i,w}}\!\!\mMod)$
to the node indexed by $(i,w)$.  If the chambers indexed by $(i,w)$ and $(j,v)$
are adjacent, then we assign the functor $\Phi^{\la_{j,v},\la_{i,w}}$ to the corresponding arrow
in the Deligne quiver.  We now need to check the relations.
Salvetti defines a CW complex which is a $W$-equivariant homotopy model for the space $E$.  
As described in \cite[pp. 611-2]{Salv}, the
1-skeleton of this complex is the Deligne quiver, and so the attaching
maps of the 2-cells completely describe the relations in the
fundamental groupoid.  There is one 2-cell for each pair of a
codimension 2 face $F$ and an adjacent chamber $C$, and the attaching
map  identifies the two minimal positive paths from $C$ to its
opposite across $F$.  Thus, we need only check that composition along
these paths gives the same functors.

Suppose we are given two such chambers, labeled by $(i,w)$ and $(i'\!,w')$.
Let $H$ be a generic cooriented hyperplane that contains $F$
and bisects both chambers.  Figure \ref{rank-2} illustrates a 2-dimensional slice transverse to $F$,
so that $F$ appears as a point and $H$ appears as a line, which in the
picture we draw as dotted.

Choose elements $\mu$ and $\nu$ of $\Pi_{i,w}$ that differ from $\la_{i,w}$ by an integral class, 
with $\mu$ on the positive side and $\nu$ on the negative side of $H$.
Choose $\mu'$ and $\nu'$ in $\Pi_{i'\!\!,w'}$ similarly.
Let $\mu=\mu_1,\mu_2,\dots,\mu_n=\mu'$ be colinear integral representatives of all the
chambers on the positive side of $H$, and let $\nu=\nu_1,\dots, \nu_\ell=\nu'$ be colinear
representatives of all the chambers on the negative side of $H$.
We may arrange these classes such that for all $k$, $\mu_k-\mu_{k+1}$ and $\nu_k-\nu_{k+1}$ both lie 
in the chamber indexed by $(i,w)$.  Put differently, we may assume that $w\mu_k - w\mu_{k+1}$ 
and $w\nu_k-w\nu_{k+1}$ are both ample on $\fM_i$.  All of this is illustrated in Figure \ref{rank-2}.

\begin{figure}[h]
  \centering
  \begin{tikzpicture}[thick,scale=3 ]
    \draw (-.4,1) -- (.4,-1);
    \draw (-1,.4)-- (1,-.4);
    \draw[dashed] (-.8,.8)--(.8,-.8);
\draw (.2,1.06) -- (-.2,-1.06);
\draw (1.06,.2) -- (-1.06,-.2);
\node[label=180:{$\mu_2$},circle,fill=black,inner sep=1.5pt] at (-.42,.07){};
\node[label=135:{$\mu$},circle,fill=black,inner sep=1.5pt] at (-.9,.55){};
\node[label=135:{$\nu$},circle,fill=black,inner sep=1.5pt] at
(-.55,.9){};
\node[label=90:{$\nu_2$},circle,fill=black,inner sep=1.5pt] at (-.07,.42){};
\node[label=-45:{$\mu'$},circle,fill=black,inner sep=1.5pt] at (.55,-.9){};
\node[label=-45:{$\nu'$},circle,fill=black,inner sep=1.5pt] at
(.9,-.55){};
\node[label=-90:{$\mu_{\ell-1}$},circle,fill=black,inner sep=1.5pt] at (.07,-.42){};
\node[label=0:{$\nu_{\ell-1}$},circle,fill=black,inner sep=1.5pt] at (.42,-.07){};
\node at (-.5,-.5){$\ddots$};
\node at (.5,.5){$\ddots$};
  \end{tikzpicture}
\caption{A 2-dimensional slice.}
\label{rank-2}
\end{figure}
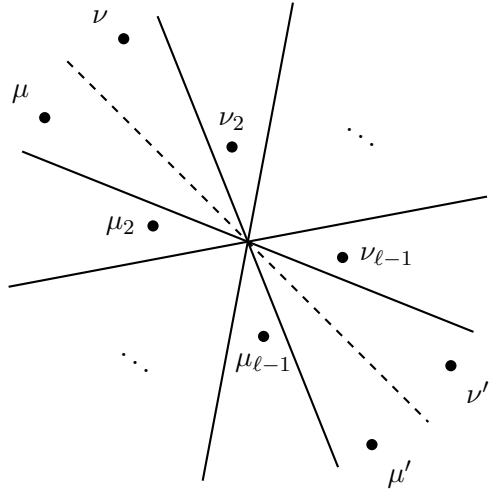

By Corollary \ref{within}, we may reduce the theorem to checking that the functors
\begin{equation*}
\Phi^{\la_{i'\!\!,w'},\,\mu'}\circ\Phi^{\mu'\!\!,\,\mu_{\ell-1}}\circ\cdots \circ\Phi^{\mu_2,\mu}\circ\Phi^{\mu,\la_{i,w}}
\and
\Phi^{\la_{i'\!\!,w'},\,\nu'}\circ\Phi^{\nu'\!\!,\,\nu_{\ell-1}}\circ\cdots \circ\Phi^{\nu_2,\nu}\circ\Phi^{\nu,\la_{i,w}}
\end{equation*}
from $D(A_{\la_{i,w}}\mMod)$ to $D(A_{\la_{i'\!\!,w'}}\mMod)$ are naturally isomorphic.
By Corollary \ref{within} and Lemma \ref{straight-line}, both are equivalent to $\Phi^{\la_{i'\!\!,w'},\la_{i,w}}$.

We have now established that the Deligne groupoid acts on the derived
categories $D(A_{\la_{i,w}}\mMod)$ for all $(i,w)\in I\times W$.
Specializing to a single parameter, we conclude that $\pi_1(E, \la)$ acts on
$D(A_\la\mMod)$ via pure twisting functors.  Furthermore, by
Proposition \ref{Weyl rings}, we have an
action of $W$ on the categories   $D(A_{\la_{i,w}}\mMod)$ via the
functors $\Phi^{\la}_{w}$.  The uniqueness of the quantizations of
line bundles (Proposition \ref{line bundles}) shows that
\[\Phi^{\la}_{w}\circ \Phi^{\la,\la'}\cong \Phi^{w\la,w\la'}\circ
\Phi^{\la'}_{w},\] so this action is
compatible with the action of $W$ on the Deligne groupoid $D$, considered
as a subgroupoid of the fundamental groupoid. 
This shows that the semi-direct product $D\rtimes W$ acts on the
categories  $D(A_{\la_{i,w}}\mMod)$.  The automorphisms of a point
$\la$ in the semi-direct product are isomorphic to $\pi_1(E/W, [\la])$.
\end{proof}

\begin{remark}\label{twisting fukaya}
We have already remarked that $\Dmod$ (and therefore $\Amod$, when localization holds)
may be thought of as a twisted algebraic version of the
Fukaya category of $\fM$ (Remark \ref{fukaya}).  
In this interpretation, we expect the action in Conjecture \ref{twisting braid} to be given by parallel
transport in the universal deformation, along the lines of the construction in \cite{SS} for Slodowy slices of type A.
\end{remark}

\begin{remark}\label{preserves subcategory}
As in Section \ref{sec:HC}, we may replace $D(A_\la\mMod)$ in the statement of Theorem \ref{twisting braid}
with $\DCLa$ (see Definition \ref{def:Cla}) for any $\bS$-equivariant $\fL_0\subset\fM_0$, or with
the bounded derived category $D^b(\CLa)$.  These categories are
related by a realization functor
$D^b(\CLa)\to \DCLa$, which may or may not be fully faithful.
\end{remark}

If $\fM$ is a hypertoric variety and 
$\fL$ is as in Example \ref{catO}, we obtain the twisting functors
studied in \cite[\S 6]{GDKD} and \cite[8.4]{BLPWtorico}.
To see this, we need to apply Lemma \ref{twist-reduce} and Remark \ref{theta doesn't matter}, 
because the functors in \cite[8.4]{BLPWtorico}
 are defined using the bimodules in Equation \eqref{bimodule}.

Recall that BGG category $\cO$ is the subcategory of finitely
generated $U(\mathfrak{g})$-modules on which $\mathfrak{b}$ acts
locally finitely, and $\mathfrak{h}$ acts semi-simply.  Let
$\cO_\la$ for a weight $\lambda$ be the Serre subcategory where the
center of $U(\mathfrak{g})$ acts with the same generalized character
as on the Verma module with highest weight $\la$.
If $\fM = T^*(G/B)$ and $\fL$ is as in Example \ref{catO}, then for
any regular integral weight $\la$, the category $C^\fL_{\la+\rho}$
is equivalent to $\cO_\la$ by Soergel's functor.  
As discussed above, this means that we have a realization functor $\mathscr{R}_\la\colon D^b(\cO_\la)\cong
D^b(C^\fL_{\la+\rho})\to D_{\fL}(A_{\la+\rho}\mmod)$, which is not obviously
fully faithful.  These functors
obviously commute with the translation equivalences between $C^\fL_{\la+\rho}$
and $C^\fL_{\la'+\rho}$ where $\la,\la'$ are both dominant and integral; thus the functor $\mathscr{R}_\la$ is
either fully faithful for all dominant integral $\la$ or for none.
The result \cite[5.13]{BLPWgco} shows that it must be fully faithful
for all $\la$ in an open subset $\fU\subset H^2(G/B)$, so it must be
an fully faithful for all dominant $\la$.  Thus, we can consider
$D^b(C^\fL_{\la+\rho})$ as a subcategory of $D_{\fL}(A_{\la+\rho}\mmod) \subset D^b(\Amod)$ in
this case.

The following result says that this equivalence
identifies the functors we call twisting functors with Arkhipov's
twisting functors \cite{Arktwist, AS}.
More precisely, Arkhipov defines a collection of derived auto-equivalences $\{T_w\mid w\in W\}$ of the category
$\cO_\la$ satisfying the relation $T_w\circ T_{w'} \cong T_{ww'}$ whenever
the length of $ww'$ is equal to the sum of the lengths of $w$ and $w'$, which means that these functors
generate an action of the generalized braid group.  In this case the discriminantal arrangement
is equal to the Coxeter arrangement for $W$, so the fundamental group $\pi_1(E/W, [\la])$
is also isomorphic to the generalized braid group.  

\begin{proposition}\label{shuf-twist}
Suppose that $\fM = T^*(G/B)$ and let $\fL$ be as in Example \ref{catO}.  
If $\la\in\Ht$ is regular, integral, and dominant, then Soergel's equivalence from the block $\cO_\la$ of BGG category $\cO$ to 
the category $C^\fL_{\la+\rho}$ intertwines Arhkipov's twisting action
 on $D^b(\cO)$ 
with the twisting action 
 on $ D^b(C^\fL_{\la+\rho}) \subset D(\AMod)$ 
from Theorem \ref{twisting braid}.
\end{proposition}
\begin{proof}
We begin by showing that
Arkhipov's twisting functors are uniquely characterized by the following two properties:
\begin{itemize}
\item $T_w$ strongly commutes with projective functors \cite[Lemma 2.1]{AS}.  
That is, for any projective functor $F$,  there is an isomorphism $T_w\circ F\cong F\circ T_w$, 
and these isomorphisms are compatible with natural transformations of projective functors.
\item For all $w\in W$, $T_wV_{\la}\cong V_{w\cdot \la}$, where $V_\la$ is the Verma module 
with highest weight $\la$.
\end{itemize}
Indeed, let $\{T'_w\mid w\in W\}$ be any other collection of functors satisfying these
conditions.  By \cite[3.3(iib)]{BeGe}, for any irreducible projective object of $\cO_\la$,
there is a projective functor taking $V_{\la}$ to that object.  Since $\cO_\la$ has enough projectives,
for any object $N$ of $\cO_\la$, there is a complex $F_N$ of projective functors taking $V_{\la}$
to $N$.  Furthermore, projective functors may be regarded as modules over $\mg\times\mg$ \cite{Back01},
and we have $\Hom_{\mg}(N,N')\cong \Hom_{\mg\times\mg}(F_N,F_{N'})$.  
We therefore have 
\[T'_wN\cong T'_wF_NV_{\la}\cong F_NT'_wV_{\la}\cong 
F_NV_{w\cdot \la}\cong F_NT_wV_{\la}
\cong T_wF_NV_{\la}\cong T_wN,\] 
and the strong commutativity condition ensures that this induces
an isomorphism of functors.  

Since $\la$ is dominant, Soergel's equivalence between $\cO_\la$ and 
$C^\cL_{\la+\rho}$ is given by composing the functors
\begin{equation}
\tikz[very thick,->,baseline=0pt]{
\node (a) at (0,0) {${}_\la^1 \mathscr{H}_\la^\infty$};
\node (b) at (-6,0) {$\cO_\la$};
\node (c) at (6,0) {$C^\cL_{\la+\rho}$};
\draw (a.160) -- (b.30) node[above,midway]{$(-)^\circ\otimes V_\la$};
\draw (a.20) -- (c.158) node[above,midway]{$\varprojlim_i(-\otimes
  V_\la^i)$};
\draw (b.-30) -- (a.-160) node[below,midway]{$\Hom_{\C}^{\operatorname{\fin}}(V_\la,-)^{\circ}$};
\draw (c.-158) -- (a.-20) node[below,midway]{$\varinjlim_i \Hom_{\C}^{\operatorname{\fin}} (V_\la^i,-)$};
}\label{equivalences}
\end{equation}
where 
\begin{itemize}
\item ${}_\la^1 \mathscr{H}_\la^\infty$ denotes the category of
Harish-Chandra bimodules (in the usual sense) for $U(\mg)$ with
generalized central character $\la$ for both the left and right actions, with
the center acting on the left semi-simply, 
\item $\Hom_{\C}^{\operatorname{\fin}}(V_\la,N)$ is the
Harish-Chandra bimodule of $U(\mg)_\Delta$-locally finite $\C$-linear maps
$V_\la\to N$, 
\item $(-)^\circ$ denotes the functor on
$U(\mg)$-$U(\mg)$ bimodules which switches the left and right
actions, twisting by the antipode of $U(\mg)$, 
\item $V^i_\la$ denotes
the length $i$ thickened Verma module $V^i_\la:=U(\mg)\otimes_{U(\mathfrak{b})}(U(\mathfrak{h})/\mathfrak{m}_\la^i)$,
where $\mathfrak{m}_\la$ is the kernel of the action of
$U(\mathfrak{h})$ on the $\la$-weight space.
\end{itemize}
Thus, we need only show that our twisting functors 
on $D^b(C^\fL_{\la+\rho})$, transported to $D^b(\cO_\la)$ via Soergel's equivalence, satisfy these two conditions.

For any element $w\in W$, let
$R^w_\la:=\Phi_w^{\la+\rho}({_{w(\la+\rho)}T_{\la+\rho}})$, where ${_{w(\la+\rho)}T_{\la+\rho}}$ is
regarded as a left $A_{w({\la+\rho})}$-module.  Consider the twisting functor $$S_w:=
\Psi_w^{\la+\rho}\circ\Phi^{w(\la+\rho),\la+\rho}\cong R^w_\la\Lotimes-.$$
Under the bi-adjoint equivalences of $C^\cL_{\la+\rho}$ with ${}_\la^1
\mathscr{H}_\la^\infty$ described in Equation \eqref{equivalences} of Example \ref{catO}, this functor is intertwined with
$R^w_\la\Lotimes -$, now regarded as a functor on Harish-Chandra bimodules, since tensor product on the left 
commutes with $\varprojlim(-\otimes
V^i_\la)$.  On the other hand, the equivalence to
$\cO_\la$, described in the same equation, involves exchanging the left and right actions.
Thus, any projective functor
$F\cong F\big(U(\mg)\big)\otimes_{U(\mg)}-$ is intertwined with
$-\otimes_{U(\mg)}F(U(\mg))^\circ\colon {}_\la^1
\mathscr{H}_\la^\infty\to {}_\la^1
\mathscr{H}_\la^\infty$, which obviously commutes with $R^w_\la\otimes -$.

Checking the second condition is an easy geometric calculation.  
Since $\la$ is dominant and regular, localization holds at $\la$ \cite{BB}.
The localization of $V_\la$ is an object of $D_{\la+\rho}\mmod$, which we may regard as a twisted D-module
by Proposition \ref{microlocalization}.
Concretely, it is the restriction of the line bundle $\cL_\la$ to the open
Bruhat cell, where only the action of $\mg$ depends on $\la$.
Tensoring with ${_{w(\la+\rho)}\cT'_{\la+\rho}}$ takes us to the restriction of $\cL_{w\cdot \la}$ to that cell.
The sections of that restriction are exactly the Verma module $V_{w\cdot \la}$, since it is
generated by a unique $U$-invariant section of weight $w\cdot \la$ (here $U$ is the nilpotent radical of $B$),
and the dimension of weight spaces matches the character of the Verma module.
\end{proof}

We end by analyzing the twisting action of Theorem \ref{twisting
  braid} on the level of the Grothendieck group.  Assume $\la\in \Pi$.
Every twisting functor $\Phi\colon D^b(A_\la\mmod)\to D^b(A_\la\mmod)$ is
induced by derived tensor product with an algebraic Harish-Chandra
bimodule $K_\Phi$; by Proposition \ref{tensor equivalence},
this implies that the
corresponding functor $\LLoc\circ \Phi \circ \Rsecs\colon D^b(\cD_\la\mmod)\to
D^b(\cD_\la\mmod)$ is induced by convolution with a geometric Harish-Chandra bimodule $F_\Phi\in {}_\la\HCg_{\la}$.
By Proposition \ref{K-conv}, the effect of $\Phi$ on characteristic cycles is given by
convolution with the characteristic cycle $\suppc(F_\Phi)$.  Thus we obtain an algebra homomorphism
$$\a\colon \C[\pi_1(E/W, [\la])]\to H^{2\dim\fM}_\fZ(\fM\times\fM; \C).$$

\begin{proposition}\label{K-triv}
The subalgebra $\C[\pi_1(E, \la)]\subset\C[\pi_1(E/W, [\la])]$ is contained in the kernel of $\a$,
thus we obtain an induced homomorphism
$$\bar\a\colon\C[W]\to  H^{2\dim\fM}_\fZ(\fM\times\fM; \C).$$
\end{proposition}

\begin{proof}
By Proposition \ref{same-cycle}, pure twisting functors preserve characteristic cycles.
Since the subalgebra $\C[\pi_1(E, \la)]\subset\C[\pi_1(E/W, [\la])]$ acts by pure twisting functors, the result follows.
\end{proof}

\begin{remark}\label{W-action}
The map $\bar\a$ also has a direct geometric construction, which
precisely matches the one given by Chriss and Ginzburg \cite[3.4.1]{CG97}
for $\fM=T^*G/B$. Applying the argument of the proof of Proposition \ref{same-cycle} 
to an impure twisting functor shows that
the class corresponding to $w$ is a specialization of the graph of the
map $w\colon \pi^{-1}(\nu)\to \pi^{-1}(w\cdot \nu)$.  
\end{remark}

\bibliography{./symplectic}
\bibliographystyle{amsalpha}
\end{document}